\documentclass[11pt]{scrartcl}

\usepackage[utf8]{inputenc}
\usepackage[T1]{fontenc}
\usepackage{lmodern}
\usepackage{alphabeta}

\usepackage{titling}
\usepackage{xspace}
\usepackage{soul} 

\usepackage[letterpaper, top=25.4mm, bottom=25.4mm, left=25.4mm, right=25.4mm, includefoot]{geometry}

\DeclareSectionCommand[%
level=4,
indent=0pt,
beforeskip=1ex plus 1ex minus .2ex,
afterskip=-1em,
font={},
tocindent=7em,
tocnumwidth=4.1em,
counterwithin=subsubsection
]{paragraph}

\usepackage[inline]{enumitem}

\usepackage{dsfont}
\usepackage{setspace}

\usepackage{bm}
\usepackage{microtype}
\usepackage{amsmath}
\usepackage{amssymb}
\usepackage{amsfonts}
\usepackage{mathtools}
\usepackage[mathscr]{euscript}

\usepackage[hyphens]{url}
\usepackage{amsthm}

\usepackage{tikz}
\usepackage{xcolor}
\usepackage{subcaption}
\usepackage{graphicx}
\usepackage{wrapfig}

\usepackage{hyperref}
\usepackage{zref-clever}

\usepackage{nicefrac}
\usepackage[textwidth=1.5in]{todonotes}

\usepackage{thmtools}
\usepackage{thm-restate}

\colorlet{myGreen}{green!50!black}
\colorlet{myLightgreen}{green}
\colorlet{myRed}{red!90!black}
\definecolor{myBlue}{rgb}{0.25, 0.0, 1.0}
\definecolor{myLightBlue}{rgb}{0.39, 0.58, 0.93}
\colorlet{myViolet}{myBlue!55!myRed}
\definecolor{myOrange}{rgb}{1.0, 0.66, 0.07}

\definecolor{CornflowerBlue}{rgb}{0.39, 0.58, 0.93}
\definecolor{DarkGoldenrod}{rgb}{0.72, 0.53, 0.04}
\definecolor{BritishRacingGreen}{rgb}{0.0, 0.26, 0.15}
\definecolor{DarkMagenta}{rgb}{0.55, 0.0, 0.55}
\definecolor{AO}{rgb}{0.0, 0.5, 0.0}
\definecolor{BostonUniversityRed}{rgb}{0.8, 0.0, 0.0}
\definecolor{myRed}{rgb}{0.8, 0.0, 0.0}
\definecolor{DarkMidnightBlue}{rgb}{0.0, 0.2, 0.4}
\definecolor{DarkTangerine}{rgb}{1.0, 0.66, 0.07}
\definecolor{AppleGreen}{rgb}{0.55, 0.71, 0.0}
\definecolor{BrightUbe}{rgb}{0.82, 0.62, 0.91}
\definecolor{Amethyst}{rgb}{0.6, 0.4, 0.8}
\definecolor{DarkGray}{rgb}{0.52, 0.52, 0.51}
\definecolor{Gray}{rgb}{0.66, 0.66, 0.66}
\definecolor{BananaYellow}{rgb}{1.0, 0.88, 0.21}
\definecolor{Amber}{rgb}{1.0, 0.75, 0.0}
\definecolor{LightGray}{rgb}{0.83, 0.83, 0.83}
\definecolor{PrincetonOrange}{rgb}{1.0, 0.56, 0.0}
\definecolor{DeepCarrotOrange}{rgb}{0.91, 0.41, 0.17}
\definecolor{CarrotOrange}{rgb}{0.93, 0.57, 0.13}
\definecolor{MidnightBlue}{rgb}{0.1, 0.1, 0.44}
\definecolor{Magenta}{rgb}{0.50, 0.0, 0.50}
\definecolor{BrightPink}{rgb}{1.0, 0.0, 0.5}
\definecolor{BrilliantRose}{rgb}{1.0, 0.33, 0.64}
\definecolor{ChromeYellow}{rgb}{1.0, 0.65, 0.0}
\definecolor{HotMagenta}{rgb}{1.0, 0.11, 0.81}
\definecolor{Auburn}{rgb}{0.43, 0.21, 0.1}
\definecolor{BrightTurquoise}{rgb}{0.03, 0.91, 0.87}
\definecolor{DarkCyan}{rgb}{0.0, 0.55, 0.55}

\vfuzz \hfuzz
\setlist[itemize]{topsep=0pt,partopsep=0pt,itemsep=0pt,parsep=0pt}
\setlist[itemize,1]{label={\small\textbullet}}
\setlist[itemize,2]{label={\tiny\textbullet}}
\setlist[itemize,3]{label=$\cdot$}
\setlist[enumerate]{topsep=0pt,partopsep=0pt,itemsep=0pt,parsep=0pt}
\setlist[enumerate,1]{label=\roman*)}
\setlist[enumerate,2]{label=\alph*)}
\setlist[enumerate,3]{label=\arabic*)}

\hypersetup{
colorlinks=true,
linkcolor=AO!65!black,
citecolor=AO!65!black,
urlcolor=AppleGreen!65!black,
bookmarksopen=true,
bookmarksnumbered,
bookmarksopenlevel=2,
bookmarksdepth=3
}

\theoremstyle{definition}

\newtheorem{environment}{Environment}[section]

\newtheorem{lemma}[environment]{Lemma}
\AddToHook{env/lemma/begin}{\zcsetup{countertype={environment=lemma}}}
\zcRefTypeSetup{lemma}{
Name-sg = Lemma ,
name-sg = Lemma ,
Name-pl = Lemmas ,
name-pl = Lemmas ,
}

\newtheorem*{lemma*}{Lemma}
\AddToHook{env/lemma*/begin}{\zcsetup{countertype={environment=lemma*}}}
\zcRefTypeSetup{lemma*}{
Name-sg = Lemma ,
name-sg = Lemma ,
Name-pl = Lemmas ,
name-pl = Lemmas ,
}

\newtheorem{proposition}[environment]{Proposition}
\AddToHook{env/proposition/begin}{\zcsetup{countertype={environment=proposition}}}
\zcRefTypeSetup{proposition}{
Name-sg = Proposition ,
name-sg = Proposition ,
Name-pl = Propositions ,
name-pl = Propositions ,
}

\AddToHook{env/corollary/begin}{\zcsetup{countertype={environment=corollary}}}
\zcRefTypeSetup{corollary}{
Name-sg = Corollary ,
name-sg = Corollary ,
Name-pl = Corollaries ,
name-pl = Corollaries ,
}

\newtheorem{theorem}[environment]{Theorem}
\AddToHook{env/theorem/begin}{\zcsetup{countertype={environment=theorem}}}
\zcRefTypeSetup{theorem}{
Name-sg = Theorem ,
name-sg = Theorem ,
Name-pl = Theorems ,
name-pl = Theorems ,
}

\newtheorem*{theorem*}{Theorem}
\AddToHook{env/theorem*/begin}{\zcsetup{countertype={environment=theorem*}}}
\zcRefTypeSetup{theorem*}{
Name-sg = Theorem ,
name-sg = Theorem ,
Name-pl = Theorems ,
name-pl = Theorems ,
}

\AddToHook{env/conjecture/begin}{\zcsetup{countertype={environment=conjecture}}}
\zcRefTypeSetup{conjecture}{
Name-sg = Conjecture ,
name-sg = Conjecture ,
Name-pl = Conjectures ,
name-pl = Conjectures ,
}

\newtheorem*{hypothesis*}{Hypothesis}
\AddToHook{env/hypothesis*/begin}{\zcsetup{countertype={environment=hypothesis*}}}
\zcRefTypeSetup{hypothesis*}{
Name-sg = Hypothesis ,
name-sg = Hypothesis ,
Name-pl = Hypotheses ,
name-pl = Hypotheses ,
}

\newtheorem{observation}[environment]{Observation}
\AddToHook{env/observation/begin}{\zcsetup{countertype={environment=observation}}}
\zcRefTypeSetup{observation}{
Name-sg = Observation ,
name-sg = Observation ,
Name-pl = Observations ,
name-pl = Observations ,
}

\AddToHook{env/example/begin}{\zcsetup{countertype={environment=example}}}
\zcRefTypeSetup{example}{
Name-sg = Example ,
name-sg = Example ,
Name-pl = Examples ,
name-pl = Examples ,
}

\AddToHook{env/remark/begin}{\zcsetup{countertype={environment=remark}}}
\zcRefTypeSetup{remark}{
Name-sg = Remark ,
name-sg = Remark ,
Name-pl = Remarks ,
name-pl = Remarks ,
}

\zcRefTypeSetup{equation}{
Name-sg = Equation ,
name-sg = Equation ,
Name-pl = Equations ,
name-pl = Equations ,
}

\zcRefTypeSetup{chapter}{
Name-sg = Chapter ,
name-sg = Chapter ,
Name-pl = Chapters ,
name-pl = Chapters ,
}

\zcRefTypeSetup{section}{
Name-sg = Section ,
name-sg = Section ,
Name-pl = Sections ,
name-pl = Sections ,
}

\zcRefTypeSetup{algorithm}{
Name-sg = Algorithm ,
name-sg = Algorithm ,
Name-pl = Algorithms ,
name-pl = Algorithms ,
}

\AddToHook{env/notation/begin}{\zcsetup{countertype={environment=notation}}}
\zcRefTypeSetup{notation}{
Name-sg = Notation ,
name-sg = Notation ,
Name-pl = Notations ,
name-pl = Notations ,
}

\newtheorem{question}[environment]{Question}
\AddToHook{env/question/begin}{\zcsetup{countertype={environment=question}}}
\zcRefTypeSetup{question}{
Name-sg = Question ,
name-sg = Question ,
Name-pl = Questions ,
name-pl = Questions ,
}

\AddToHook{env/problem/begin}{\zcsetup{countertype={environment=problem}}}
\zcRefTypeSetup{problem}{
Name-sg = Problem ,
name-sg = Problem ,
Name-pl = Problems ,
name-pl = Problems ,
}

\AddToHook{env/claim/begin}{\zcsetup{countertype={environment=claim}}}
\zcRefTypeSetup{claim}{
Name-sg = Claim ,
name-sg = Claim ,
Name-pl = Claims ,
name-pl = Claims ,
}

\AddToHook{env/definition/begin}{\zcsetup{countertype={environment=definition}}}
\zcRefTypeSetup{definition}{
Name-sg = Definition ,
name-sg = Definition ,
Name-pl = Definitions ,
name-pl = Definitions ,
}

\zcRefTypeSetup{figure}{
Name-sg = Figure ,
name-sg = Figure ,
Name-pl = Figures ,
name-pl = Figures ,
}

\usetikzlibrary{calc}
\usetikzlibrary{fit}
\usetikzlibrary{decorations}
\usetikzlibrary{decorations.pathmorphing}
\usetikzlibrary{decorations.text}
\usetikzlibrary{shapes,hobby}

\tikzset{
	position/.style args={#1:#2 from #3}{
		at=($(#3)+(#1:#2)$)
	}
}

\tikzset{
  v:main/.style = {draw, circle, scale=0.8, thick,fill=black,inner sep=0.7mm},
  v:ghost/.style = {inner sep=0pt,scale=1},
  >={latex},
  e:marker/.style = {line width=8.5pt,line cap=round,opacity=0.35,color=DarkGoldenrod},
  e:main/.style = {line width=1pt},
}

\title{Quickly excluding an annotated planar graph}
\predate{}
\date{}
\postdate{}

\preauthor{}
\DeclareRobustCommand{\authorthing}{
	\begin{center}
		Maximilian Gorsky\footnote{\href{mailto:m.gorsky@pm.me}{m.gorsky@pm.me}}~~\!\thanks{Supported by the Institute for Basic Science (IBS-R029-C1).} \\
        {\small Discrete Mathematics Group, Institute for Basic Science (IBS), Daejeon, South Korea} \\
		  \medskip
        Evangelos Protopapas\thanks{Supported by the ERC project BUKA (n°\! 101126229).}~~\!\footnote{\href{mailto:eprotopapas@mimuw.edu.pl}{eprotopapas@mimuw.edu.pl}} \\
		{\small Faculty of Mathematics, Informatics and Mechanics, University of Warsaw, Poland} \\
		  \medskip
		Sebastian Wiederrecht\footnote{\href{mailto:sebastian.wiederrecht@gmail.com}{wiederrecht@kaist.ac.kr}} \\
		{\small School of Computing, KAIST, Daejeon, South Korea}
\end{center}}
\author{\authorthing}
\postauthor{}

\begin{document}
\maketitle

\begin{abstract}
We provide proofs certifying that the structure theorem for vertex sets of bounded bidimensionality holds with polynomial bounds.
The bidimensionality of vertex sets is a common generalisation of both treewidth and the face-cover-number of vertex sets in planar graphs.
As such, it plays a crucial role in extensions of Courcelle's Theorem to $H$-minor-free graphs.
Recently, bidimensionality and similar parameters have emerged as key for extensions of known parameterized algorithms for problems defined on a terminal set $R$.
A prominent example for such a problem is Steiner Tree, which admits efficient algorithms on planar graphs whenever $R$ can be covered with few faces.

Key to the algorithmic applications of bidimensionality is a structure theorem that explains how a graph $G$ can be decomposed into pieces where the behaviour of $R$ is highly controlled.
One may see this structure theorem as a rooted analogue of Robertson and Seymour's celebrated Grid Theorem.
Combining recent advances in obtaining polynomial bounds in the Graph Minors framework with new techniques for handling annotated vertex sets, we show that all parameters in the structure theorem above admit polynomial bounds.
As an application, we also provide a sketch showing how our techniques imply polynomial bounds for the structure theorem for graphs excluding an apex minor.
\end{abstract}


\let\sc\itshape
\thispagestyle{empty}

\newpage

\newpage
\thispagestyle{empty}
\tableofcontents

\newpage

\setcounter{page}{1}


\section{Introduction}\label{sec:intro}
A central approach to dealing with computational intractability in graph problems is offered by \textsl{structural graph theory}.
Through a wide range of structural notions and results, it supports the systematic design of efficient algorithms for hard problems on well-behaved graph classes.
A particularly powerful toolkit from structural graph theory is the design of \emph{graph parameters} capturing key features that facilitate the design of efficient algorithms.
The algorithmic study of such graph parameters makes up a rich subfield of \textsl{parameterized algorithms}.
A prime example of this interplay between structural and algorithmic graph theory is the parameter \textsl{treewidth} popularised by Robertson and Seymour \cite{RobertsonS1986Graphb} (see for example \cite{Bodlaender1986Classes,Courcelle1990Monadic,CyganFKLMPPS2015Parameterized,Korhonen2023SingleExponential}).
While treewidth is a powerful tool for the design of parameterized algorithms for a wide range of problems, it also naturally gives rise to a problem:
For which classes of problems exist structural parameters that allow for the design of parameterized algorithms \textsl{beyond} the regime of treewidth?

A \emph{annotated graph} is a pair $(G,R)$ where $G$ is a graph and $R\subseteq V(G)$ is the set of \emph{annotated vertices}.
In the following, we will refer to $R$ as the set of \emph{red} vertices.
Recently, a family of parameters has emerged that aims to target problems defined on annotated graphs such as the \textsc{Steiner Tree} problem \cite{Frederickson1991Planar,Curticapean2016Counting,KrauthgamerLRr2019FlowCut,PandeyvL2022PlanarMultiway}.
The theme of such parameters is to restrict the structural properties of the red vertices instead of the global structure of the graph.
This is motivated by the observation that for many such problems, treewidth already defines their tractability horizon within minor-closed graph classes.
To be more precise, the seminal Grid Theorem of Robertson and Seymour \cite{RobertsonS1986Grapha} says that a minor-closed graph class $\mathcal{G}$ has bounded treewidth if and only if $\mathcal{G}$ does not contain all planar graphs.
Often problems like \textsc{Steiner Tree} are \textsf{NP}-hard already on planar graphs \cite{DreyfusW1971Steiner,EricksonMV1987SendAndSplit,KisfaludiBakNvL2020NearlyETH}, implying that such problems are tractable on a minor-closed graph class $\mathcal{G}$ if and only if $\mathcal{G}$ has bounded treewidth.
In the emerging theory of \textsl{colorful minors}, sometimes called \textsl{rooted minors}, several positive algorithmic results \cite{Frederickson1991Planar,Curticapean2016Counting,KrauthgamerLRr2019FlowCut,PandeyvL2022PlanarMultiway,GroendlandNK2024Polynomial,JansenS2024SteinerTree} hint at an paradigm providing a way to escape such dichotomies:
\vspace{-8pt}
\begin{center}
\textsl{Define parameters to restrict the structure relative to $R$ instead or \\ in addition to restricting the structure of $G$ as a whole.}
\end{center}
\vspace{-8pt}
To better capture the situation where a planar graph ``rooted'' on a fixed set $R$ of vertices is excluded, Thilikos and Wiederrecht defined the notion of \textsl{bidimensionality} \cite{ThilikosW2024Bidimensionality}.
The \emph{bidimensionality} of an annotated graph $(G,R)$ is the largest integer $k$ such that there exists a minor-model\footnote{A \emph{minor-model} of a graph $H$ is a graph $G$ is a collection $\{ G_v\}_{v\in V(H)}$ of pairwise vertex-disjoint connected subgraphs of $G$, called the \emph{branch sets}, such that for all $uv\in E(H)$ there is an edge in $G$ between $V(G_u)$ and $V(G_v)$.} $\mathcal{X}$ of the $(k \times k)$-grid where $R\cap V(X) \neq \emptyset$ for all branch sets $X\in \mathcal{X}$.
We call such a grid minor a \emph{red grid}.
See \zcref{fig:IntroGridModel} for an illustration.

The notion of bidimensionality has since found key applications in the algorithmic theory of model checking in $H$-minor-free graphs \cite{SiebertzV2024Avances,SauGT2025Parameterizing}.
On the structural side, Protopapas, Thilikos, and Wiederrecht \cite{ProtopapasTW2025Colorful} proved a structure theorem providing an approximate description of annotated graphs of small bidimensionality akin to Robertson and Seymour's duality between grid minors and treewidth.
However, the bounds provided by Protopapas et al.\@ for their structure theorem depend exponentially on the bidimensionality.

\textbf{Our main contribution.}
We give an independent proof for this structure theorem providing, for the first time, \textsl{polynomial bounds} for it.
Our proof is constructive, yielding a fixed-parameter tractable algorithm that either decides that the bidimensionality of $(G,R)$ is at least $k$ or finds the structural decomposition from \cite{ProtopapasTW2025Colorful}, where all involved parameters are in $k^{\mathbf{O}(1)}$.

\subsection{Our result}\label{subsec:Result}
To state our main theorem, we require some additional definitions from Robertson and Seymour's theory of graph minors.
This is because in certain situations, the number of red vertices may be unbounded \textsl{and} they might be highly connected to each other, but the way they attach to the rest of the graph is restricted in other ways.
To see this, consider for example a large grid where only the very first column is red (see \zcref{fig:IntroNegativeExamples}).
In this annotated graph the number of red vertices is unbounded, at the same time it is impossible to decompose it along small-order separations in a way that distributes the red vertices into smaller pieces.
However, the entire graph is planar and all red vertices sit on a single face.
In general the structure emerging from bounding the bidimensionality may be more complicated, but this example illustrates that in order to grasp the structural aspects of bidimensionality, one must account for certain topological conditions.
To describe the separation into well-behaved pieces, we make use of the notion of tree-decompositions.

\begin{figure}[ht]
    \centering
    \begin{tikzpicture}

        \pgfdeclarelayer{background}
		\pgfdeclarelayer{foreground}
			
		\pgfsetlayers{background,main,foreground}
			
        \begin{pgfonlayer}{main}
        \node (C) [v:ghost] {};

        \node(L) [v:ghost,position=180:5cm from C] {
            \begin{tikzpicture}

                \pgfdeclarelayer{background}
		          \pgfdeclarelayer{foreground}
			
		          \pgfsetlayers{background,main,foreground}

                \begin{pgfonlayer}{background}
                    \pgftext{\includegraphics[width=4cm]{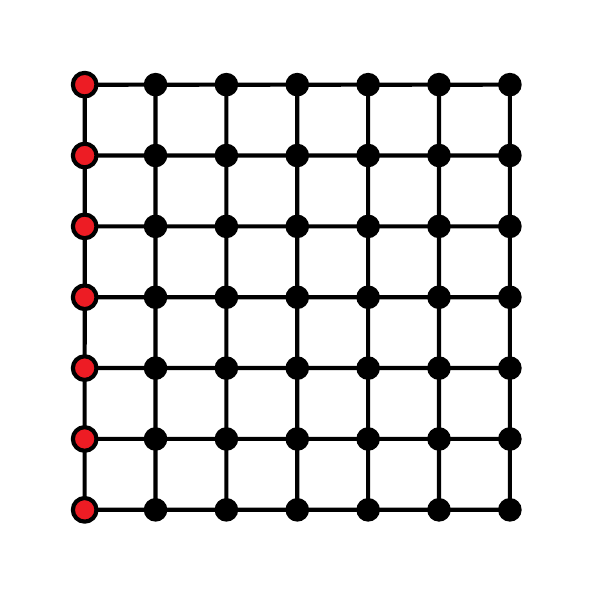}} at (C.center);
                \end{pgfonlayer}{background}
			
                \begin{pgfonlayer}{main}
                    \node (C) [v:ghost] {};
            
                \end{pgfonlayer}{main}
        
                \begin{pgfonlayer}{foreground}
                \end{pgfonlayer}{foreground}

            \end{tikzpicture}
        };

        \node (ghost) [v:ghost,position=0:0mm from C] {};

        \node(M) [v:ghost,position=90:3mm from ghost] {
            \begin{tikzpicture}

                \pgfdeclarelayer{background}
		          \pgfdeclarelayer{foreground}
			
		          \pgfsetlayers{background,main,foreground}

                \begin{pgfonlayer}{background}
                    \pgftext{\includegraphics[width=4cm]{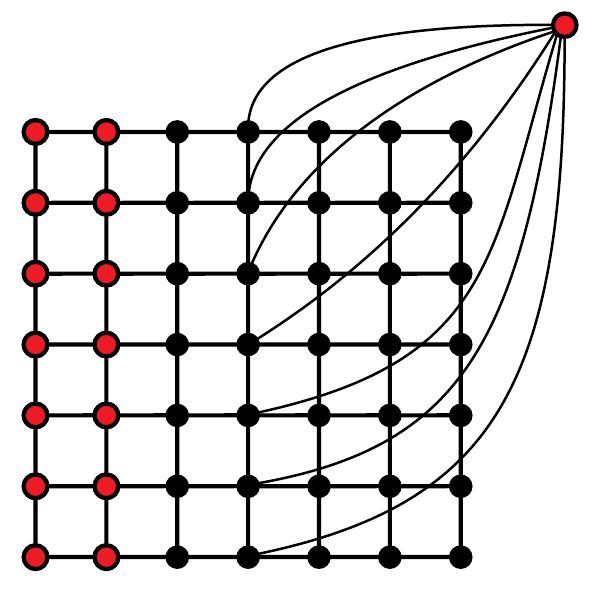}} at (C.center);
                \end{pgfonlayer}{background}
			
                \begin{pgfonlayer}{main}
                    \node (C) [v:ghost] {};
            
                \end{pgfonlayer}{main}
        
                \begin{pgfonlayer}{foreground}
                \end{pgfonlayer}{foreground}

            \end{tikzpicture}
        };

        \node (Mlabel) [v:ghost,position=270:2.3cm from M] {(iii)};

        \node (Llabel) [v:ghost,position=180:5cm from Mlabel] {(ii)};

        \node(LL) [v:ghost,position=180:10cm from C] {
            \begin{tikzpicture}

                \pgfdeclarelayer{background}
		          \pgfdeclarelayer{foreground}
			
		          \pgfsetlayers{background,main,foreground}

                \begin{pgfonlayer}{background}
                    \pgftext{\includegraphics[width=4cm]{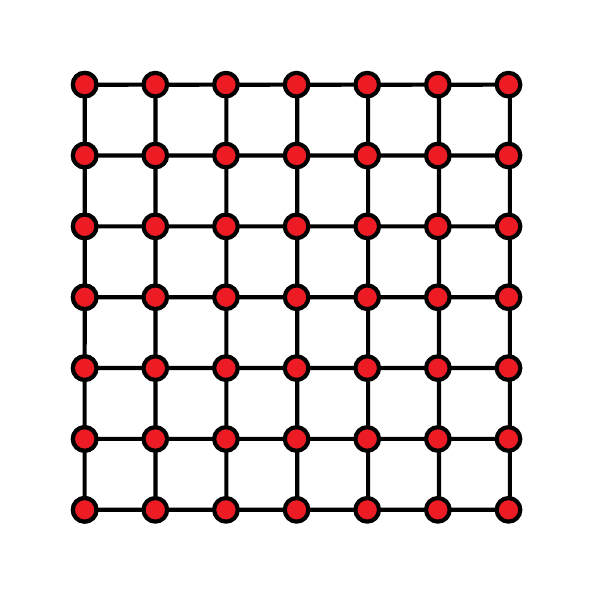}} at (C.center);
                \end{pgfonlayer}{background}
			
                \begin{pgfonlayer}{main}
                    \node (C) [v:ghost] {};
            
                \end{pgfonlayer}{main}
        
                \begin{pgfonlayer}{foreground}
                \end{pgfonlayer}{foreground}

            \end{tikzpicture}
        };

        \node (LLlabel) [v:ghost,position=180:5cm from Llabel] {(i)};

        \end{pgfonlayer}{main}
        
        \begin{pgfonlayer}{foreground}
        \end{pgfonlayer}{foreground}

    \end{tikzpicture}
    \caption{Diagrams of (i) the red $(7 \times 7)$-grid and two examples illustrating that the structure emerging from excluding a red grid resembles the structure of excluding an arbitrary minor. (ii) A grid with its outermost column being red: no large red grid exists but at the same time the red vertices cannot be easily removed from the grid. (iii) a more complicated example where embeddability can only be achieved by deleting an apex vertex and few faces do not suffice to describe the structure of the red vertices.}
    \label{fig:IntroNegativeExamples}
\end{figure}
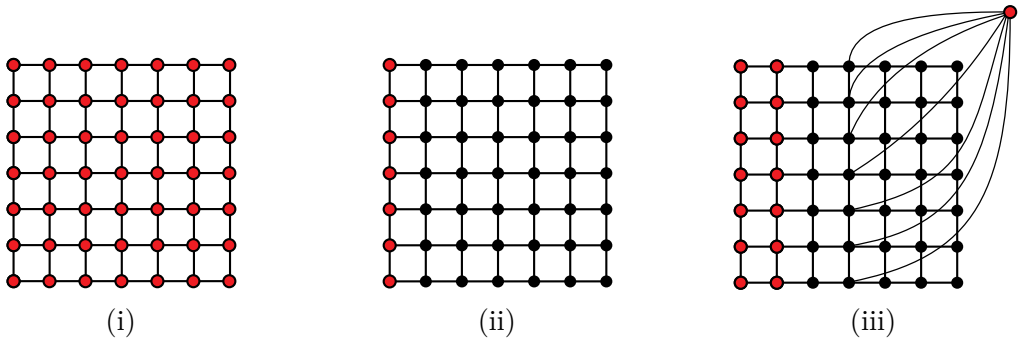

\smallskip
A \emph{tree-decomposition} for a graph $G$ is a pair $(T,\beta)$ such that $T$ is a tree, $\beta\colon V(T)\to 2^{V(G)}$ assigns to each node of $T$ a subset of the vertices of $G$ known as a \emph{bag}, $\bigcup_{t\in V(T)}G[\beta(t)] = G$, and for each $v\in V(G)$, the set $\{ t\in V(T) \colon v\in\beta(t) \}$ is connected.
The \emph{adhesion} of $(T,\beta)$ is $0$ if $T$ has only one node and $\max_{dt \in E(T)}|\beta(d) \cap \beta(t)|$ otherwise.
The \emph{width} of $(T,\beta)$ is defined as $\max_{t\in V(T)} |\beta(t)|-1$.
\smallskip

The topological condition we require is more complicated.
Notice that in the example above, if we would also colour the second column in red, the bidimensionality of the resulting annotated graph would still be bounded by a constant, independent of the order of the grid.
Moreover, even if we add one additional red vertex and make it adjacent to the entire centre column of the grid, we would only increase the bidimensionality by $1$ for all such grids.
But the resulting graph cannot be embedded into a surface with reasonable Euler-genus.
Thus we must allow for the deletion of a small vertex set in order to reveal topological behaviour and we need to relax our condition on what it means to cover the red vertices with few faces.
See \zcref{fig:IntroNegativeExamples} for an illustration.

\smallskip
We say that a graph $G$ has a \emph{$k$-near embedding} in a surface $\Sigma$ if there exists $A\subseteq V(G)$ with $|A|\leq k$, such that $G-A=G_0\cup G_1 \cup \dots \cup G_{\ell}$, $\ell\leq k$, such that $G_0$ has an embedding into $\Sigma$ with $\ell$ pairwise vertex-disjoint faces $F_i$ where for all $i\in \{ 1,\dots,\ell\}$, $V(G_0)\cap V(G_i)= V(F_i)$, for $i\in\{ 1,\dots,\ell\}$, the $G_i$'s are pairwise vertex-disjoint and each such $G_{i}$ has a path-decomposition $(P_i,\beta_i)$\footnote{A \emph{path-decomposition} of a graph is a tree-decomposition $(T,\beta)$ where $T$ is a path.} of width at most $k$ such that $V(P_i)=V(F_i)$, the vertices of $P_{i}$ appear in agreement to their cyclic ordering on the boundary of $F_{i}$, and $v\in \beta_i(v)$, for all $v\in V(F_i)$.
The set $A$ is called the \emph{apex set}, the graphs $G_i$, $i\in\{1,\dots,\ell\}$, are called the \emph{vortices}, and the sets $V(G_i)\setminus V(F_i)$ are the \emph{interiors} of the vortices.

In the context of annotated graphs, we need some additional information regarding the red vertices.
Let $(T,\beta)$ be a tree decomposition of an annotated graph $(G,R)$.
The \emph{annotated torso} $(G_t,R_t)$ of $(G,R)$ at a node $t\in V(T)$ is the annotated graph obtained from $(G[\beta(t)],R)$ by turning, for every $dt\in E(T)$, the set $\beta(d)\cap\beta(t)$ into a clique and turning $v$ red for all $v\in\beta(d)\cap\beta(t)$ such that $\bigcup_{d'\in V(T_d)}\beta(d')\cap R \neq \emptyset$ where $T_d$ is the unique component of $T-dt$ that contains $d$.

With these definitions, our main theorem reads as follows.

\begin{theorem}\label{thm:BidimensionalityIntro}
There exists a function $f_{\ref{thm:BidimensionalityIntro}}\colon\mathbb{N}\to\mathbb{N}$ such that for all non-negative integers $k$, and all annotated graphs $(G,R)$ one of the following holds:
\begin{enumerate}
\item $(G,R)$ has bidimensionality at least $k$, or
\item $(G,R)$ has a tree-decomposition $(T,\beta)$ of adhesion at most $f_{\ref{thm:BidimensionalityIntro}}(k)$ such that for all $t\in V(T),$ either $t$ is a leaf with the parent $d$ and $(\beta(t)\setminus\beta(d)) \cap R = \emptyset$, or the annotated torso $(G_t,R_t)$ of $(G,R)$ at $t$ has an $f_{\ref{thm:BidimensionalityIntro}}(k)$-near embedding in a surface of Euler-genus at most $f_{\ref{thm:BidimensionalityIntro}}(k)$ and every vertex $v\in\beta(t)\cap R$ belongs to the apex set or the interior of some vortex.
\end{enumerate}
Moreover, $f_{\ref{thm:BidimensionalityIntro}}(k)\in k^{\mathbf{O}(1)}$ and there exists an algorithm that finds either a $(k\times k)$-grid-minor-model witnessing that $(G,R)$ has bidimensionality at least $k$, or a tree-decomposition as above in time $2^{k^{\mathbf{O}(1)}} \cdot |E(G)|^{3}|V(G)|\log(|V(G)|).$
\end{theorem}

\subsection{Related work, consequences, and applications}\label{subsec:applications}

\paragraph{Generalising treewidth through annotation.}
Notice that by our definition, a minor-model of a red $(k \times k)$-grid does not have to be a \textsl{minimal} minor-model of a $(k \times k)$-grid.
That is, we explicitly allow for the red vertices to be connected to the grid in the form of dangling paths.
See \zcref{fig:IntroGridModel} for an example.
This notion of ``red minor'' has already been studied by a variety of authors, often under the name of ``rooted minors''.
However, depending on the context, rooted minors are sometimes required to respect a fixed bijection between the coloured vertices of $(G,R_G)$ and the coloured vertices of $(H,R_H)$.
For better distinction between the two concepts, we say that $(H,R_H)$ is a \emph{red minor} of $(H,R_G)$ if there exists a minor-model $\{ G_v\}_{v\in V(H)}$ of $H$ in $G$ such that $R_G \cap V(G_v) \neq \emptyset$ for all $v\in R_H$.
This is in reference to the term ``colorful minor'' introduced by Protopapas, Thilikos, and Wiederrecht \cite{ProtopapasTW2025Colorful} which allows for vertices to carry more than one colour.

\begin{figure}[ht]
    \centering
    \begin{tikzpicture}

        \pgfdeclarelayer{background}
		\pgfdeclarelayer{foreground}
			
		\pgfsetlayers{background,main,foreground}

        \begin{pgfonlayer}{background}
            \pgftext{\includegraphics[width=11cm]{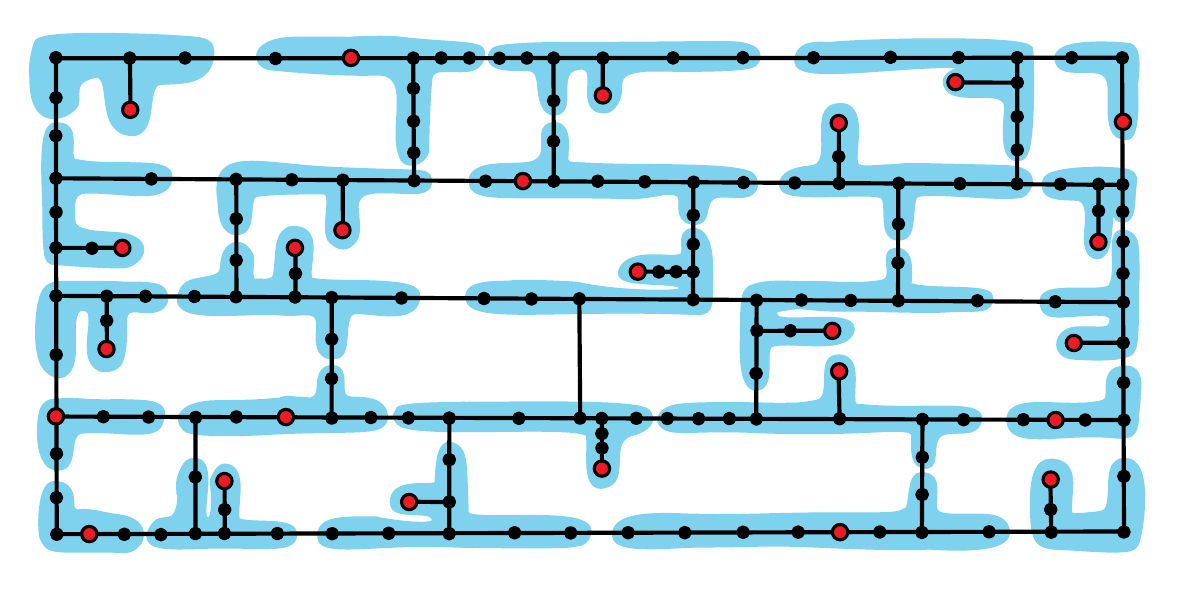}} at (C.center);
        \end{pgfonlayer}{background}
			
        \begin{pgfonlayer}{main}
        \node (C) [v:ghost] {};
            
        \end{pgfonlayer}{main}
        
        \begin{pgfonlayer}{foreground}
        \end{pgfonlayer}{foreground}

    \end{tikzpicture}
    \caption{A model of a red $(5 \times 5)$-grid.}
    \label{fig:IntroGridModel}
\end{figure}

The work of Protopapas et al.\@ provides several structural theorems for annotated graphs excluding fundamental patterns, including a variant of \zcref{thm:BidimensionalityIntro} with exponential bounds, a structure theorem for excluding a red clique, and one for the exclusion of an \emph{outer red grid}, i.e.\@ an annotated graph $(G,R)$ where $G$ is a $(k \times k)$-grid for some $k$ and $R$ is the vertex set of its first column.
A small outer red grid can be found in \zcref{fig:IntroNegativeExamples}.
While Protopapas et al.\@ investigate a variant of the outer red grid with more than one colour, the first result describing the structure of graphs excluding an outer red grid is due to Marx, Seymour, and Wollan \cite{MarxSW2017Rooted}.
It was later observed by Hodor, La, Micek, and Rambaud \cite{HodorLMR2025Quickly} that the result of Marx et al.\@ provides a min-max characterisation of a type of modulation parameter that falls into a family of structural parametrisations started by Bulian and Dawar with their introduction of \textsl{elimination distance} \cite{BulianD2014Isomorphsm,BulianD2017Fixed}.
\smallskip

The \emph{treewidth} of a graph $G$, denoted by $\mathsf{tw}(G)$, is the smallest integer $k$ such that $G$ has a tree-decomposition of width at most $k$.

Let $G$ be a graph and $X\subseteq V(G)$ be a vertex set.
The \emph{torso} of $G$ at $X$ is the graph $G_X$ obtained from $G[X]$ by turning the set $N_G(J) \subseteq X$ into a clique for every component $J$ of $G-X$.

The \emph{torso treewidth} of an annotated graph $(G,R)$, denoted by $\mathsf{ttw}(G,R)$, is the smallest integer $k$ such that there exists a set $X\subseteq V(G)$ with $R \subseteq X$ and $\mathsf{tw}(G_X) \leq k$.
\smallskip

As observed by Hodor et al.\@ \cite{HodorLMR2025Quickly}, the result of Marx et al.\@ \cite{MarxSW2017Rooted} implies that every annotated graph with large torso treewidth contains a big outer red grid as a red minor.
One may observe a hierarchy of parameters as follows:
\begin{enumerate}
     \item The Grid Theorem of Robertson and Seymour explains the structure of graphs excluding a planar graph as minor,
     \item the theorem of Marx et al.\@ explains the structure of annotated graphs excluding a planar graph with a single red face as a red minor, and
     \item \zcref{thm:BidimensionalityIntro} explains the structure of an annotated graph excluding an arbitrary annotated planar graph as a red minor.
\end{enumerate}
Among these \zcref{thm:BidimensionalityIntro} is the most general, as it in particular implies Robertson and Seymour's Grid Theorem, since the treewidth of any graph $G$ equals the bidimensionality of $(G,V(G))$.
Moreover, due to torso treewidth being sandwiched between treewidth and bidimensionality, in annotated graphs where all vertices are red all three parameters collapse into one.

\paragraph{An application: Excluding an apex minor.}
A graph $H$ is said to be an \emph{apex graph} if there exists some vertex $v\in V(H)$ such that $H-v$ is planar.
Apex graphs themselves are very close to planar graphs.
However, apex-minor-free graph classes generalise planar graphs and any graph class of bounded Euler-genus.
Graphs in these classes allow for a range of interesting algorithmic results \cite{Eppstein1999Subgraph,Eppstein2000Diameter,Grohe2003Local,DemaineH2004Equivalence,DraganFG2008PTAS,DemaineHK2009Approximation,FominLPPS2016Subexponential,KorhonenNPS2024Fully}.
Many of these algorithmic applications are in the realm of approximation algorithms.
Indeed, as proposed by Eppstein \cite{Eppstein1999Subgraph,Eppstein2000Diameter} and later built upon by Grohe \cite{Grohe2003Local} and Demain and Hajiaghayi \cite{DemaineH2004Equivalence}, apex-minor-free graphs have bounded ``local treewidth'' which allows for an application of a variant of Baker's Technique \cite{Baker1994Approximation}, facilitating the design of polynomial approximation schemes (PTAS) for problems which are hard to approximate on general graphs.

Indeed, while the structure theorem for apex-minor-free graphs was known in the community, the first written proof for it can be found in the appendix of a paper by Dvo\v{r}{\'a}k and Thomas \cite{DvorakT2016Listcoloring} on approximating list colourings.
Since their proof builds on the original results of Robertson and Seymour, they do not come with any estimates on the bounds for the constants describing the running times of their algorithm.
This issue is shared among many of the applications mentioned above and gets explicitly mentioned by Grohe \cite{Grohe2003Local}.
Our result shows that, in the case of apex-minor-free graphs, it is possible to give explicit polynomials bounds directly through the use of modern graph minor theory.

By interpreting the red vertices in each step of the proof as the combined neighbourhood of the apex vertices produced, our techniques show that one may either find any fixed apex graph as a minor, or create a region in the surface-embedded part which avoids the neighbourhood of all apex vertices.
This yields a polynomial version of Dvo\v{r}{\'a}k and Thomas' \textsl{local} structure theorem.
The term ``local'' here refers to the fact that the structure theorem is stated with respect to a large wall.
A local version of \zcref{thm:BidimensionalityIntro} can be found in the form of \zcref{thm:strongest_localstructure}.

We say that a graph $G$ containing a wall $W\subseteq G$ has a \emph{weak $k$-near embedding centred at $W$} in a surface $\Sigma$ if there exists $A\subseteq V(G)$ with $|A|\leq k$, such that $G-A=G_0\cup G_1 \cup \dots \cup G_{\ell} \cup J_1 \dots, J_{h}$, $\ell\leq k$ and $h\in \mathbb{N}$, such that $G_0$ has an embedding into $\Sigma$ with $\ell+h$ faces $F_i$ where 
\begin{itemize}
    \item the graphs $G_i-V(G_0)$ and $J_j-V(G_0)$ are vertex-disjoint for all $i\in \{1,\dots,\ell \}$ and $j\in\{ 1,\dots,h\}$,
    \item for all $j\in\{ 1,\dots, h\}$, the face $F_j$ contains a set $S_j$ of at most three vertices such that all $S_j$'s are pairwise disjoint and $V(J_j) \cap V(F_j) = S_j = V(J_j) \cap V(G_0)$, and such that, if $|S_j|=3$, $V(F_j)=S_j$ and if $|S_j|=2$ then the vertices of $S_j$ are adjacent on $F_j$,
    \item for all $i\in \{ h+1,\dots,h+\ell\}$, $V(G_0)\cap V(G_i)= V(F_i)$, for $i\in\{ 1,\dots,\ell\}$, the $G_i$'s are pairwise vertex-disjoint and each such $G_{i}$ has a path-decomposition $(P_i,\beta_i)$ of adhesion at most $k$ such that $V(P_i)=V(F_i)$, the vertices of $P_{i}$ appear in agreement to their cyclic ordering on the boundary of $F_{i}$, and $v\in \beta_i(v)$, for all $v\in V(F_i)$, and
    \item $W$ is vertex-disjoint from $\bigcup_{i=1}^{\ell}G_i$ and for each $j\in\{ 1,\dots,h\}$, $J_j$ contains at most one vertex of degree $3$ from $W$.
\end{itemize}
The set $A$ is called the \emph{apex set}, the graphs $G_i$ are called the \emph{vortices}, and the sets $V(G_i)\setminus V(F_i)$ are the \emph{interiors} of the vortices and the graphs $J_j$ are called the \emph{flaps}.

\begin{theorem}\label{thm:ApexMinorIntro}
There exist functions $g_{\ref{thm:ApexMinorIntro}}\colon\mathbb{N}^2 \to \mathbb{N}$ and $f_{\ref{thm:ApexMinorIntro}} \colon \mathbb{N} \to \mathbb{N}$ such that for all non-negative integers $r \geq 3$ and $k$, every apex graph $H$ on at most $k$ vertices, every graph $G$, and every $w$-wall $W \subseteq G$ with $w \geq g_{\ref{thm:ApexMinorIntro}}(k,r)$ one of the following holds:
\begin{enumerate}
    \item $G$ contains $H$ as a minor, or
    \item there exists an $r$-subwall $W'\subseteq W$ such that $G$ has a weak $f_{\ref{thm:ApexMinorIntro}}(k)$-near embedding centred at $W'$ such that all neighbours of the apex set are contained in the interiors of the vortices.
\end{enumerate}
Moreover, $g_{\ref{thm:ApexMinorIntro}}(r,k) \in (r + k)^{\mathbf{O}(1)}$, $f_{\ref{thm:ApexMinorIntro}}(k) \in k^{\mathbf{O}(1)}$, and there exists an algorithm that takes as input, $G$, $W$, $H$, and $r$ as above and finds either a minor model of $H$, or an $r$-subwall $W'$ together with a weak $f_{\ref{thm:ApexMinorIntro}}(k)$-near embedding centred at $W'$ for $G$ in time $\mathbf{poly}(|V(H)|+r+k)|E(G)|^{3}$.
\end{theorem}

It is important to stress that \zcref{thm:ApexMinorIntro} is a consequence of our \textsl{methods} and the lemmas we develop along the way, not directly of the main statements that eventually give rise to \zcref{thm:BidimensionalityIntro} -- or more directly to \zcref{thm:strongest_localstructure}.
In \zcref{subsec:ProofOverview} we provide some additional insight and in \zcref{sec:ApexExclusion} we explain the derivation of \zcref{thm:ApexMinorIntro} from our methods in more depth.
\smallskip

Moreover, the reason why we only state the local variant of Dvo\v{r}{\'a}k and Thomas' structure theorem for apex-minor-free graphs here is three-fold.

First, almost all Robertson-Seymour-style structure theorems are proven in this way: One first proves a local variant with respect to a wall or ``tangle'' and then applies a well-known technique to turn the local structure theorem into a global one based on tree decompositions -- see \zcref{sec:localtoglobal} for this local-to-global step proving \zcref{thm:BidimensionalityIntro}.
Novel arguments are usually only necessary to prove the local theorems, with the local-to-global step seeing few innovations over the years.

Second, while the global theorems are better known, the local structure theorems tend to be the more important ones.
Robertson and Seymour point this out in their own proof of the Graph Minor Structure Theorem \cite{RobertsonS2003Grapha}, declaring the global structure theorem to be a ``red herring''.

Finally, the third reason is of a technical nature.
While the process of going from local to global is by now routine, its approximate nature comes with the drawback of introducing additional vertices to the apex set.
Dvo\v{r}{\'a}k and Thomas have laid out a technique to capture the neighbourhoods of these new apices inside additional vortices \cite{DvorakT2016Listcoloring} (see also \cite{MorellePTW2025Excluding} for a more recent take on the technique).
However, this step introduces additional technicalities which we believe to be outside the scope of this paper and so we leave this part to the interested readers.

\paragraph{The algorithmic potential of annotated graph parameters.}
As described above, the notion of bidimensionality for annotated vertex sets may be seen as a vast structural generalisation of the face-cover-number of annotated vertex sets in planar graphs.
A wide range of problems defined on annotated graphs are known to be tractable on planar annotated graphs whenever the face-cover-number is bounded.
Interesting problems that fall into this framework are the \textsc{Multiway Cut} problem \cite{PandeyvL2022PlanarMultiway} and the problem of counting perfect matchings with defects \cite{Curticapean2016Counting}.
Even the \textsc{Disjoint Paths} problem on planar graphs exhibits strong algorithmic properties when the face-cover-number of its terminals is considered \cite{Mazoit2013Single,Verbeek2022Disjoint}.
A leading question in this newly arising theory of structural parameters for annotated graphs is the following:

\begin{question}\label{quest1}
Given an \textsf{NP}-hard computational problem $\Pi$ defined on annotated graphs which is tractable on planar graphs where the red vertices can be covered by a constant number of faces, which proper red-minor-closed classes of annotated graphs admit polynomial-time algorithms for $\Pi$? 
\end{question}

A prime example of such a behaviour is the \textsc{Steiner Tree} problem \cite{KisfaludiBakNvL2020NearlyETH}.
Recently, \textsc{Steiner Tree} has become the subject of a streamlined investigation into the power of structural parameters for annotated graphs.
Jansen and Swennenhuis \cite{JansenS2024SteinerTree} showed that \textsc{Steiner Tree} is fixed-parameter tractable when parameterized by the torso treewidth of the annotated input instance.
Groenland, Nederlof, and Koana \cite{GroendlandNK2024Polynomial} have shown that excluding a red $K_4$-minor also gives rise to a class of polynomially solvable instances of \textsc{Steiner Tree}.
Notice that, due to the structural result of Marx et.\@ al, annotated graphs excluding a red $K_4$-minor may have unbounded torso treewidth, however, since the red $K_4$ is a minor of the red $(3 \times 3)$-grid, excluding $K_4$ as a red minor implies bounded bidimensionality.
Given that \textsc{Steiner Tree} is \textsf{NP}-hard on the class of all annotated planar graphs \cite{GareyJ1977Rectilinear}, it follows that it is also \textsf{NP}-hard on all red-minor-closed classes of unbounded bidimensionality.
However, there is strong evidence, that bidimensionality might be what precisely delineates the tractability of \textsc{Steiner Tree} in red-minor-closed classes.

\begin{question}\label{quest2}
Are there computable functions $f$ and $g$ such that \textsc{Steiner Tree} can be solved in time $f(k) n^{g(k)}$ on annotated graphs of bidimensionality at most $k$?
\end{question}

Due to a result of Krisfaludi-Bak, Nederlof, and Leeuwen \cite{KisfaludiBakNvL2020NearlyETH}, the dependency on $k$ in the exponent of $n$ in \zcref{quest2} is unavoidable assuming the Exponential Time Hypothesis, a feature that is shared by many problems of the same flavour.
Key towards providing a positive answer to \zcref{quest2} is a full understanding of the structural properties of annotated graphs of low bidimensionality, which is one of the main motivations behind our work.
Indeed, in case the answer to \zcref{quest2} is ``yes'', the functions $f$ and $g$ are likely to depend on the function from \zcref{thm:BidimensionalityIntro}, further emphasising the need for constructive and good bounds.
\smallskip

Apart from its focal role in understanding structural parameters for annotated graphs, bidimensionality has also been observed to be a key player in recent extensions of Courcelle's Theorem \cite{Courcelle1990Monadic}.
Courcelle's Theorem states that the model checking problem for MSO$_2$-formulas -- naively speaking, formulas where one is allowed to quantify over vertex sets and edge sets -- is fixed-parameter tractable on graphs of bounded treewidth.
In a vast generalisation of this result, Sau, Stamoulis, and Thilikos \cite{SauGT2025Parameterizing} have shown that in the space of minor-closed graph classes one may extend the applicability of Courcelle's Theorem beyond the scope of graphs of bounded treewidth by restricting the quantifications allowed in the formulas to vertex sets of bounded bidimensionality and certain disjoint-paths queries.
Strengthenings of this ``meta-algorithmic'' result have been proven for annotated graphs of bounded torso treewidth and annotated graphs excluding a red clique minor by Protopapas, Thilikos, and Wiederrecht \cite{ProtopapasTW2025Colorful}.

\paragraph{Layered parameters and the exclusion of ``apex-$\mathcal{X}$'' graphs.}
As manifested in \zcref{thm:ApexMinorIntro}, structure theory for excluding red minors in annotated graphs has direct implications for the exclusion of certain types of graphs $H$ in the setting of graph minors.
Let $\mathcal{X}$ be a minor-closed graph class.
We say that a graph $H$ is \emph{apex-$\mathcal{X}$} if there exists a vertex $v\in V(H)$ such that $H-v \in \mathcal{X}$.
In this terminology, an apex graph may also be called an apex-planar graph.

A \emph{layering} of a graph $G$ is a sequence $\langle L_i\rangle_{i\in\mathbb{N}}$ such that $\bigcup_{i\in\mathbb{N}}L_i=V(G)$ and for every edge $uv\in E(G)$ there is $i\in\mathbb{N}$ such that $u,v \in L_i \cup L_{i+1}$.

A layered parameter is a generalisation of a decomposition-based graph parameter evaluated over all possible decompositions and layerings for a given graph.
For example, the \emph{layered treewidth} of a graph $G$ is the minimum value of $  \max_{i\in\mathbb{N}} \max_{t\in V(T)} |L_i \cap \beta(t)|$ taken over all layerings $\langle L_i\rangle_{i\in\mathbb{N}}$ and all tree-decompositions $(T,\beta)$ of $G$.

Recently, layered graph decompositions have emerged as a powerful tool for solving a number of long-standing open problems even beyond the scope of minor-closed graph classes and including key results in the theory of the \textsl{clustered chromatic number} \cite{DujmovicMW2017Layered,DujmovicF2018Stack,DujmovicEJMW2020MinorClosed,ScottW2020BetterBounds,DujmovicMY2021TwoResults,DujmovicEMWW2022Clustered,BonamyBEGLPS2023Asymptotic,LiuW2024Clustered,DallardMMY2025Layered,HodorLMR2025Quickly}.
In many cases, the existence of such layered decompositions of small width is implied by the absence of some apex-$\mathcal{X}$ graph for a carefully chosen class $\mathcal{X}$.
Indeed, Dujmovi{\'c}, Morin, and Wood \cite{DujmovicMW2017Layered} showed that a minor-closed graph class has bounded layered treewidth if and only if it excludes some apex-planar graph and their proof makes use of the structure theorem of Dvo\v{r}{\'a}k and Thomas.
Hence, \zcref{thm:ApexMinorIntro} has direct implications for the bounds found by Dujmovi{\'c} et al.\@ \cite{DujmovicMW2017Layered}.
To obtain a better understanding for the structure of excluding an apex-$\mathcal{X}$ graph as a minor, it is often easier to understand the structure of excluding an annotated graph $(H,R)$ with $H\in\mathcal{X}$ as a red minor \cite{DujmovicEJMW2020MinorClosed,HodorLMR2025Quickly,ClausHJM2026Excluding}.
Our main theorem shows that, whenever $\mathcal{X}$ is a class of planar graphs, all involved bounds are polynomial.

\paragraph{Towards colorful minors of $q$-colorful graphs.}
The original structure theorem for annotated graphs of bounded bidimensionality due to Protopapas, Thilikos, and Wiederrecht \cite{ProtopapasTW2025Colorful} had one additional powerful feature which our polynomial version is lacking:
Instead of considering only one colour, Protopapas et al.\@ allowed for the presence of up to $q$ colours and considered the situation where they exclude a $(k\times k)$-grid in which every vertex carries all $q$ colours.
As a result of this generality, their bound is of the form $2^{k^{\mathbf{O}(1)}2^{2^{\mathbf{O}(q)}}}$.
The main reason for this super-exponential dependency on $q$ is the lack of tools which are able to ``homogenise'' a given flat wall with respect to $q$ distinct colours within polynomial bounds.
It should be mentioned that very recently Gorsky, Seweryn, and Wiederrecht \cite{GorskySW2026Price} introduced a technique that provides precisely such a homogenisation procedure.
However, as Gorsky et al.\@ explain in the conclusion of their paper, a variant of \zcref{thm:BidimensionalityIntro} for $q$ colours realising polynomial bounds is still out of reach for now.
The reason is that homogenising flat walls is not enough.
As explained in \zcref{subsec:ProofOverview}, a crucial step is the homogenisation of so-called ``transactions'' -- a notion from deep within the proof of the Graph Minor Structure Theorem (see \cite{RobertsonS1990Graph,KawarabayashiTW2021Quickly,GorskySW2025Polynomial}).
A tool for efficiently dealing with the homogenisation of transactions is still missing at the time of writing.

\subsection{Overview of our proof}\label{subsec:ProofOverview}
The proof of \zcref{thm:BidimensionalityIntro} is divided into four main steps as follows.
\begin{description}
    \item[Step 1:] A variant of the Flat Wall Theorem confining all red vertices to the outside of the flat wall, which for purpose of presentation we call the \textsl{Red Flat Wall Theorem}.
    \item[Step 2:] A version of the Society Classification Theorem ensuring that all new additions to the weak near embedding are free of red vertices, which we call the \textsl{Red Society Classification Theorem}.
    \item[Step 3:] The local version of \zcref{thm:BidimensionalityIntro} with respect to a large wall, which we call the \textsl{Red Local Structure Theorem}.
    \item[Step 4:] The globalisation of the outcome of \textbf{Step 3}, yielding \zcref{thm:BidimensionalityIntro}.
\end{description}
The first three steps combined make up the proof of the local version of \zcref{thm:BidimensionalityIntro}, that is \zcref{thm:strongest_localstructure}, while \textbf{Step 4} utilises the red local structure theorem in order to construct the tree-decomposition mentioned in \zcref{thm:BidimensionalityIntro}.
The proof of \zcref{thm:strongest_localstructure} inductively generates a near embedding.
The proof uses the outcome of the red flat wall theorem to enter the base case of the induction and then repeatedly applies the outcome of the red society classification theorem in order to either conclude or increase the Euler-genus of the underlying surface in the near embedding.

\paragraph{Step 1: A flat wall.}
The celebrated Flat Wall Theorem of Robertson and Seymour \cite{RobertsonS1995Graph} states that given any large enough wall $W$ in a graph $G$, either there exists a large clique-minor in $G$ whose model is highly connected to $W$, or there exists a weak near embedding of $G$ centred at $W$ in the sphere with a single vortex, without any bound on the adhesion of the path decomposition of the vortex.
The outcome of this first step is, that in the absence of a big red grid, one may ensure that all red vertices are confined into the vortex and the apex set.

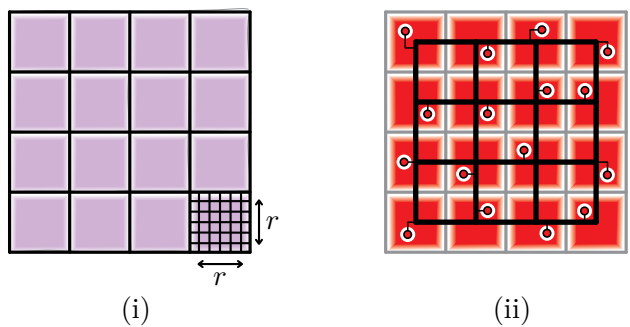
\begin{figure}[ht]
    \centering
    \begin{tikzpicture}

        \pgfdeclarelayer{background}
		\pgfdeclarelayer{foreground}
			
		\pgfsetlayers{background,main,foreground}
			
        \begin{pgfonlayer}{main}
        \node (C) [v:ghost] {};

        \node (Mpos) [v:ghost,position=180:50mm from C] {};

        \node(M) [v:ghost,position=0:0mm from Mpos] {
            \begin{tikzpicture}

                \pgfdeclarelayer{background}
		          \pgfdeclarelayer{foreground}
			
		          \pgfsetlayers{background,main,foreground}

                \begin{pgfonlayer}{background}
                    \pgftext{\includegraphics[width=4cm]{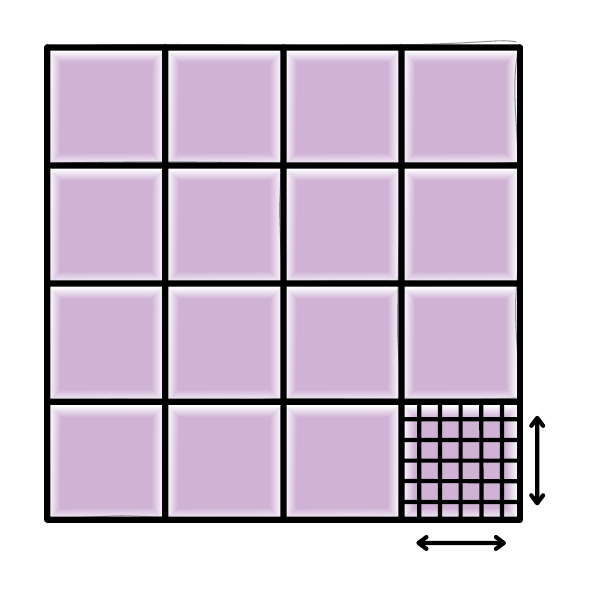}} at (C.center);
                \end{pgfonlayer}{background}
			
                \begin{pgfonlayer}{main}
                    \node (C) [v:ghost] {};
            
                \end{pgfonlayer}{main}
        
                \begin{pgfonlayer}{foreground}
                \end{pgfonlayer}{foreground}

            \end{tikzpicture}
        };

        \node (Rpos) [v:ghost,position=0:0mm from C] {};

        \node(R) [v:ghost,position=0:0mm from Rpos] {
            \begin{tikzpicture}

                \pgfdeclarelayer{background}
		          \pgfdeclarelayer{foreground}
			
		          \pgfsetlayers{background,main,foreground}

                \begin{pgfonlayer}{background}
                    \pgftext{\includegraphics[width=4cm]{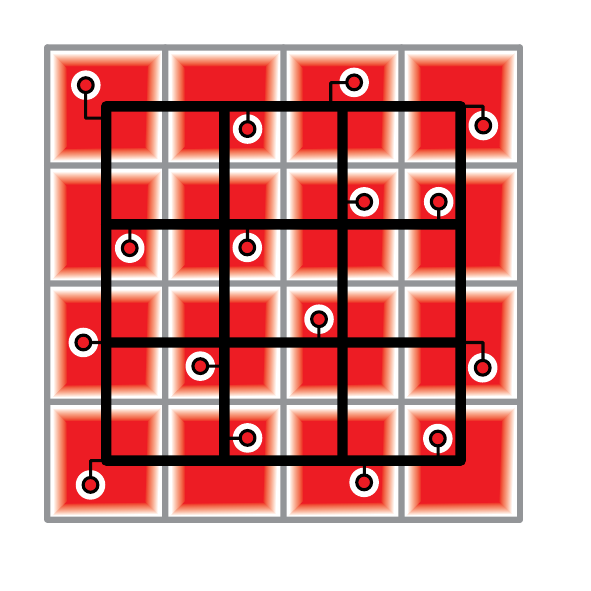}} at (C.center);
                \end{pgfonlayer}{background}
			
                \begin{pgfonlayer}{main}
                    \node (C) [v:ghost] {};
            
                \end{pgfonlayer}{main}
        
                \begin{pgfonlayer}{foreground}
                \end{pgfonlayer}{foreground}

            \end{tikzpicture}
        };

        \node (Rlabel) [v:ghost,position=270:2.3cm from R] {(ii)};

        \node (Mlabel) [v:ghost,position=180:5cm from Rlabel] {(i)};

        \node (g1) [v:ghost,position=0:11.2mm from Mlabel] {};
        \node (g2) [v:ghost,position=90:11.8mm from g1] {};

        \node (l1) [v:ghost,position=0:7mm from g2] {$r$};
        \node (l2) [v:ghost,position=270:7.5mm from g2] {$r$};

        \end{pgfonlayer}{main}
        
        \begin{pgfonlayer}{foreground}
        \end{pgfonlayer}{foreground}

    \end{tikzpicture}
    \caption{Diagrams of (i) a division of a $(5+4r)$-wall to either find a an $r$-subwall without any red vertices, or a red $(4\times 4)$-grid minor and (ii) a $(4 \times 4)$-grid minor in the second outcome.}
    \label{fig:BlankFlatWallIntro}
\end{figure}

This part is relatively straightforward.
We start with a large wall $W_0$ and apply a polynomial variant of the Flat Wall Theorem -- we use the one of Gorsky et al.\@ \cite{GorskySW2025Polynomial} -- to obtain either a large clique or a small apex set $A_1$ together with a big flat subwall $W_1$ of $W_0$.
A result of Protopapas, Thilikos, and Wiederrecht \cite{ProtopapasTW2025Colorful} allows us to dismiss the case where we find the clique.
Otherwise we subdivide $W_1$ into a large number of pairwise disjoint subwalls arranged in a grid-like shape.
If each of them contains a red vertex in its interior, we find a large red grid, otherwise one of them is the desired flat wall $W_2$ without any red vertices.
See \zcref{fig:BlankFlatWallIntro} for an illustration.

Towards deducing \zcref{thm:apexlocalstructure}, after the initial application of the Flat Wall Theorem, declare the neighbourhood of $A_1$ to be the set of red vertices and set up the numbers to obtain a red $(|A_1|k \times |A_1|k)$-grid as one possible outcome.
In case this red grid is found, the pigeonhole principle yields a $(k \times k)$-grid rooted in the neighbourhood of a single member of $A_1$.
Otherwise, $W_2$ must be flat without deleting a single apex vertex.

\paragraph{Step 2: Classifying a red society.}

The Society Classification Theorem -- first proven explicitly by Kawarabayashi, Thomas, and Wollan \cite{KawarabayashiTW2021Quickly} -- may be seen as an extension of the Flat Wall Theorem as follows:
    Given a vortex surrounded by a large number of concentric cycles in the embedded part -- called the \emph{nest} -- contained within a disk $\Delta$ of the surface, one may find one of four possible outcomes, see \zcref{fig:SocietyClassificationIntro} for an illustration:
    \begin{enumerate}
        \item A large clique-minor in $G$ whose model is highly connected to the nest.
        \item Subject to removing a small set of apices $A,$ either 
        \begin{enumerate}
        \item a large flat crosscap transaction in $\Delta$ that traverses the interior of the vortex, and moreover is orthogonal\footnote{A \emph{transaction} is a set of disjoint paths linking two disjoint boundary segments of $\Delta$. We say that a transaction escaping the embedded part of our graph is \emph{orthogonal} to a nest, if the intersection of each path in the transaction with each cycle in the nest consists of two paths -- one causes by escaping the embedded part and the other by entering it again (see \zcref{fig:SocietyClassificationIntro} as a reference).} to most of the starting nest,
        \item a large flat handle transaction in $\Delta$ that traverses the interior of the vortex and moreover is orthogonal to most of the starting nest, or
        \item a weak near embedding of the part of the graph drawn in $\Delta,$ centred at the nest, with a bounded number of vortices each with a bounded adhesion path decomposition.
        Moreover, each of these vortices is surrounded by a large nest which is linked back to the starting nest via many disjoint paths which are orthogonal to the starting nest.
        \end{enumerate}
    \end{enumerate}

The proof of the red society classification theorem (see \zcref{thm:blanksocietyclassification}) utilizes the society classification theorem above as a departure point and further refines each of its outcomes as outlined below.

\begin{figure}[ht]
    \centering
    \begin{tikzpicture}

        \pgfdeclarelayer{background}
		\pgfdeclarelayer{foreground}
			
		\pgfsetlayers{background,main,foreground}
			
        \begin{pgfonlayer}{main}
        \node (C) [v:ghost] {};

        \node (Lpos) [v:ghost,position=180:100mm from C] {};

        \node(L) [v:ghost,position=0:0mm from Lpos] {
            \begin{tikzpicture}

                \pgfdeclarelayer{background}
		          \pgfdeclarelayer{foreground}
			
		          \pgfsetlayers{background,main,foreground}

                \begin{pgfonlayer}{background}
                    \pgftext{\includegraphics[width=3.5cm]{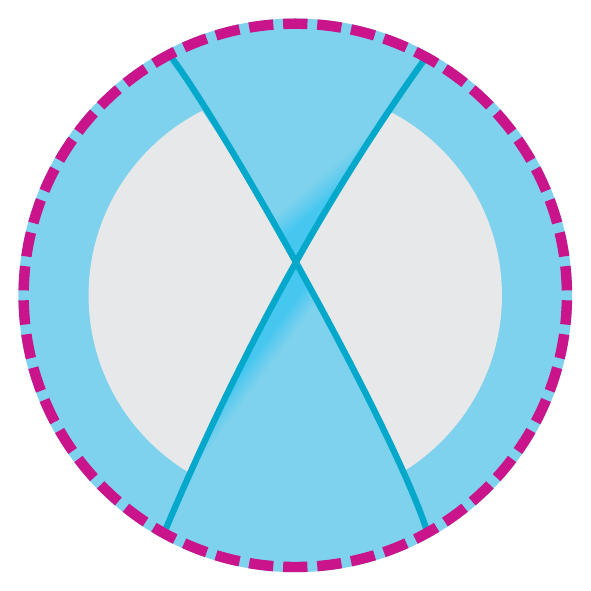}} at (C.center);
                \end{pgfonlayer}{background}
			
                \begin{pgfonlayer}{main}
                    \node (C) [v:ghost] {};
            
                \end{pgfonlayer}{main}
        
                \begin{pgfonlayer}{foreground}
                \end{pgfonlayer}{foreground}

            \end{tikzpicture}
        };

        \node (Mpos) [v:ghost,position=180:50mm from C] {};

        \node(M) [v:ghost,position=0:0mm from Mpos] {
            \begin{tikzpicture}

                \pgfdeclarelayer{background}
		          \pgfdeclarelayer{foreground}
			
		          \pgfsetlayers{background,main,foreground}

                \begin{pgfonlayer}{background}
                    \pgftext{\includegraphics[width=3.5cm]{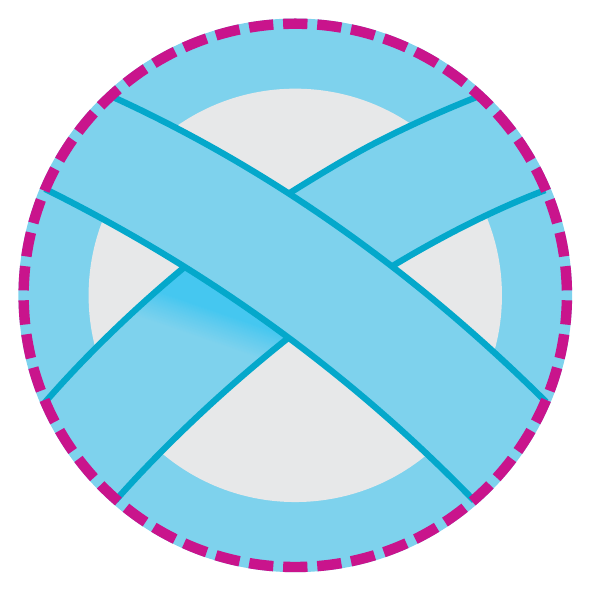}} at (C.center);
                \end{pgfonlayer}{background}
			
                \begin{pgfonlayer}{main}
                    \node (C) [v:ghost] {};
            
                \end{pgfonlayer}{main}
        
                \begin{pgfonlayer}{foreground}
                \end{pgfonlayer}{foreground}

            \end{tikzpicture}
        };

        \node (Rpos) [v:ghost,position=0:0mm from C] {};

        \node(R) [v:ghost,position=0:0mm from Rpos] {
            \begin{tikzpicture}

                \pgfdeclarelayer{background}
		          \pgfdeclarelayer{foreground}
			
		          \pgfsetlayers{background,main,foreground}

                \begin{pgfonlayer}{background}
                    \pgftext{\includegraphics[width=3.5cm]{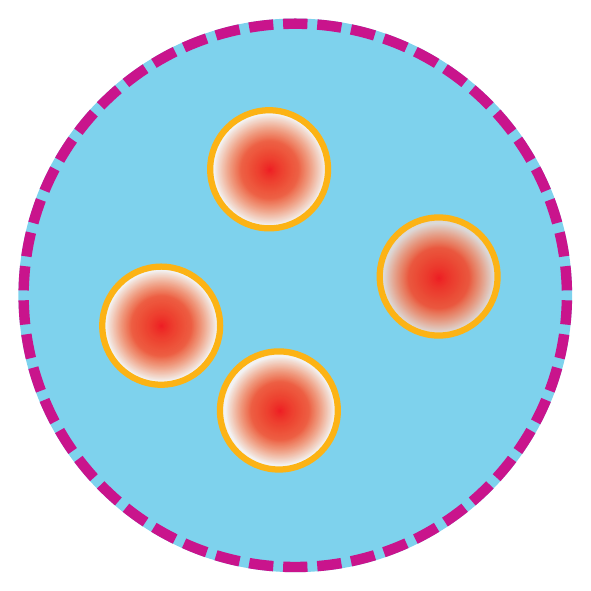}} at (C.center);
                \end{pgfonlayer}{background}
			
                \begin{pgfonlayer}{main}
                    \node (C) [v:ghost] {};
            
                \end{pgfonlayer}{main}
        
                \begin{pgfonlayer}{foreground}
                \end{pgfonlayer}{foreground}

            \end{tikzpicture}
        };

        \node (Rlabel) [v:ghost,position=270:2cm from R] {(iii)};

        \node (Mlabel) [v:ghost,position=180:5cm from Rlabel] {(ii)};

        \node (Llabel) [v:ghost,position=180:5cm from Mlabel] {(i)}; 
            
        \end{pgfonlayer}{main}
        
        \begin{pgfonlayer}{foreground}
        \end{pgfonlayer}{foreground}

    \end{tikzpicture}
    \caption{Diagrams of (i) a crosscap transaction, (ii) a handle transaction, and (iii) a weak near embedding in a disk with a small number of vortices.}
    \label{fig:SocietyClassificationIntro}
\end{figure}
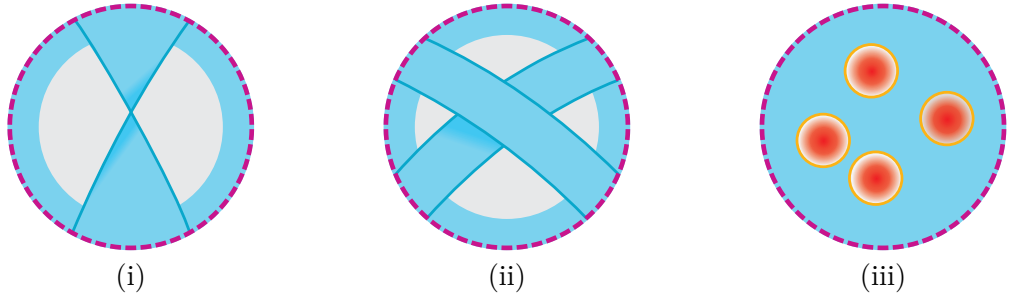

Similarly to before, case (i), where a clique minor is given, is delegated to the result of \cite{ProtopapasTW2025Colorful}.

To explain cases (ii.a) and (ii.b), we briefly recall the notion of a \textsl{flat transaction}.
The goal of this outcome eventually is to augment the weak near embedding using structure extracted from the interior of the vortex.
We say that a transaction is \emph{monotone} if the order of its endpoints on one of the two segments of $\Delta$ its endpoints are placed on is mirrored by the order of its endpoints on the other segment.
We define its strip of a monotone transaction as the union of those components of the graph drawn in $\Delta$ that remain after deleting the transaction and the boundary of $\Delta,$ and that have neighbours on an interior path of the transaction.
A transaction is \emph{flat (under $A$)} if its strip admits a weak near embedding without vortices (after deleting a set of vertices $A$).

As the given flat transaction is orthogonal to most of the nest, arguments analogous to those used in \textbf{Step 1} yield a flat subtransaction such that all regions between consecutive paths in the weak near embedding of the strip are homogeneous, meaning all contain a red vertex or none does.
In the former case we obtain a large red grid; otherwise, the transaction is entirely free of red vertices.
See \zcref{fig:CleaningTransaction} for an illustration of the situation.

\begin{figure}[ht]
    \centering
    \begin{tikzpicture}

        \pgfdeclarelayer{background}
		\pgfdeclarelayer{foreground}
			
		\pgfsetlayers{background,main,foreground}
			
        \begin{pgfonlayer}{main}
        \node (C) [v:ghost] {};

        \node (Mpos) [v:ghost,position=0:0mm from C] {};

        \node(M) [v:ghost,position=0:0mm from Mpos] {
            \begin{tikzpicture}

                \pgfdeclarelayer{background}
		          \pgfdeclarelayer{foreground}
			
		          \pgfsetlayers{background,main,foreground}

                \begin{pgfonlayer}{background}
                    \pgftext{\includegraphics[width=14cm]{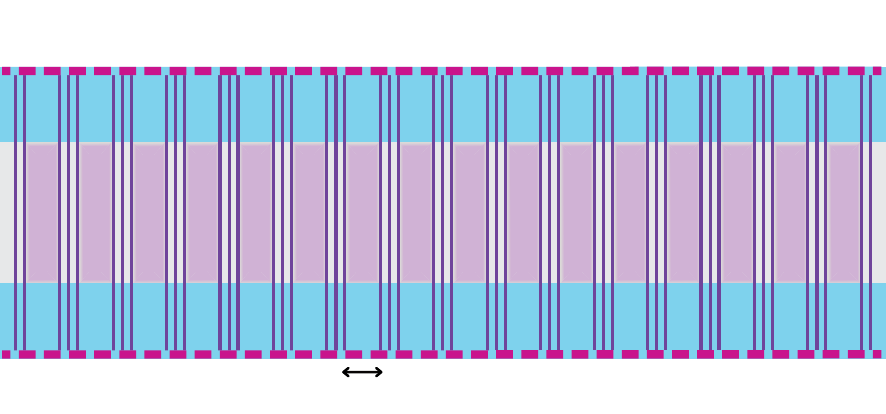}} at (C.center);
                \end{pgfonlayer}{background}
			
                \begin{pgfonlayer}{main}
                    \node (C) [v:ghost] {};
            
                \end{pgfonlayer}{main}
        
                \begin{pgfonlayer}{foreground}
                \end{pgfonlayer}{foreground}

            \end{tikzpicture}
        };

        \node (Rpos) [v:ghost,position=270:60mm from C] {};

        \node(R) [v:ghost,position=0:0mm from Rpos] {
            \begin{tikzpicture}

                \pgfdeclarelayer{background}
		          \pgfdeclarelayer{foreground}
			
		          \pgfsetlayers{background,main,foreground}

                \begin{pgfonlayer}{background}
                    \pgftext{\includegraphics[width=14cm]{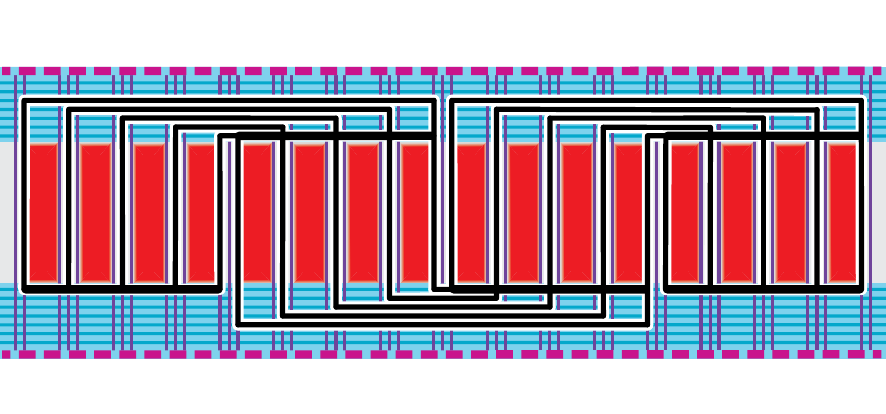}} at (C.center);
                \end{pgfonlayer}{background}
			
                \begin{pgfonlayer}{main}
                    \node (C) [v:ghost] {};
            
                \end{pgfonlayer}{main}
        
                \begin{pgfonlayer}{foreground}
                \end{pgfonlayer}{foreground}

            \end{tikzpicture}
        };

        \node (Rlabel) [v:ghost,position=180:75mm from R] {(ii)};

        \node (Mlabel) [v:ghost,position=180:75mm from M] {(i)};

        \node (g1) [v:ghost,position=270:28mm from Mlabel] {};
        \node (g2) [v:ghost,position=0:62.3mm from g1] {};

        \node (l1) [v:ghost,position=0:0mm from g2] {$r$};

        \end{pgfonlayer}{main}
        
        \begin{pgfonlayer}{foreground}
        \end{pgfonlayer}{foreground}

    \end{tikzpicture}
    \caption{Diagrams of (i) a strip of a transaction divided into $16$ pairwise disjoint substrips, each containing a transaction of order $r$ and (ii) a $(4 \times 4)$-grid minor in the case where all strips in (i) contain a red vertex.}
    \label{fig:CleaningTransaction}
\end{figure}
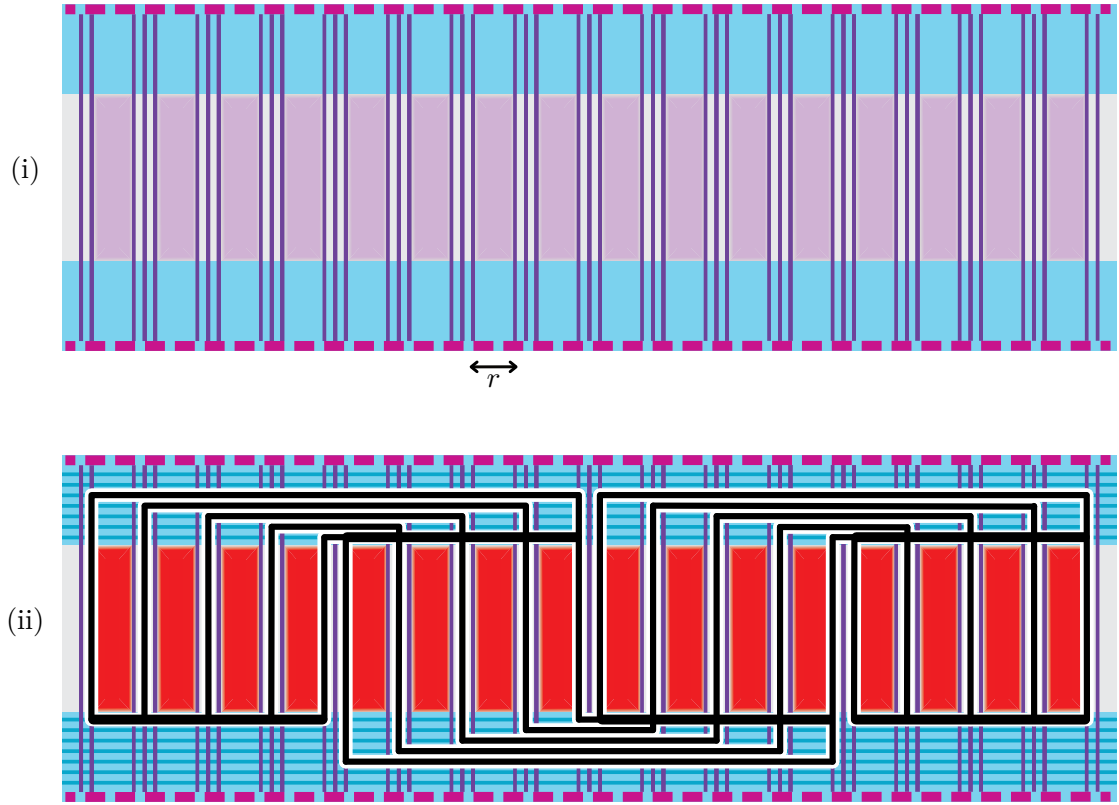

Outcome (ii.c) of the society classification theorem constitutes the main technical challenge.
In existing proofs of the GMST (e.g. \cite{RobertsonS1990Graph, RobertsonS2003Grapha, KawarabayashiTW2021Quickly, GorskySW2025Polynomial}), refining a weak near embedding by repeatedly splitting vortices requires sacrificing part of the nest, and termination is ensured by using so-called \textsl{crooked transactions} \cite{RobertsonS1990Graphb}.
For technical reasons, this tool is unavailable here.
However, the existence of a weak near embedding as in (ii.c) allows us to avoid sacrificing cycles.

Using similar techniques as for the proof of the red flat wall theorem on the wall-like structure from (ii.c), we either find a large red grid or obtain a large nest together with many radially traversing paths orthogonal to it, such that every vortex and every red vertex is enclosed by the nest.
We then view the interior of the nest as a single vortex.
Either this vortex admits a path decomposition of bounded adhesion, or there exists a large transaction traversing the nest.
A sequence of lemmas that follow shows that, such a transaction can always be chosen to traverse the nest interior, be orthogonal to the nest, and avoid all red vertices, similarly to cases (ii.a) and (ii.b).
This relies on adapting a technique originating in \cite{ThilikosW2024Killing} and further developed in \cite{PaulPTW2024Obstructionsa, PaulPTW2025Local}.
Crucially, this allows us to discard one side of a nest when splitting it, provided that side contains no red vertex or vortex, without reducing the nest size.
It is noteworthy that this case never shows up in the actual proof of the Graph Minor Structure Theorem.

To ensure polynomial bounds, we must also control how new nests connect to previous ones. Naively splitting radial paths at each step would lead to exponential loss.
Instead, we inductively maintain a \textsl{nest tree}, a hierarchical tree-like structure in which each nest is connected to its two children nests by sets of radial paths extracted from the splitting transactions, rather than from the parent nest’s own radial paths.
See \zcref{fig:NestTree} for an illustration.
This guarantees that no loss accumulates.

Finally, termination is ensured using the nest tree itself.
If the number of leaves in the nest tree exceeds a certain quadratic threshold (in terms of the target grid size), a simple argument using Menger's Theorem and the rich structure of the nest tree yields a large red grid minor.
Otherwise, we obtain only a bounded number of leaves for our nest tree, in total containing all original vortices and red vertices.
Each of these leaves can now be seen as vortices admitting a path decomposition of bounded adhesion, as no large transaction traversing them exists.
Again using Menger's Theorem and the structure of the nest tree we can link the nests in the leaves back to (part of) the original nest as desired.

\paragraph{Step 3: The local structure theorem.}

Using the red society classification theorem, we now inductively prove the red local structure theorem with respect to a large wall.
The statement essentially reads as follows: Given an annotated graph $(G, R)$ with a large wall $W,$ one may find one of the four possible outcomes:
\begin{enumerate}
\item A small set of vertices $S$ such that the component of $G - S$ that contains the majority of $W$ is free of red vertices.
\item A large red clique-minor in $G$ whose model is highly connected to $W$.
\item A large red grid-minor in $G$ whose model is highly connected to $W.$
\item A weak near embedding of $G$ centred at $W$ with a bounded number of vortices each with a bounded adhesion path decomposition such that all red vertices are confined in the interior of the vortices.
\end{enumerate}

\begin{figure}[ht]
    \centering
    \begin{tikzpicture}

        \pgfdeclarelayer{background}
		\pgfdeclarelayer{foreground}
			
		\pgfsetlayers{background,main,foreground}

        \begin{pgfonlayer}{background}
            \pgftext{\includegraphics[width=8cm]{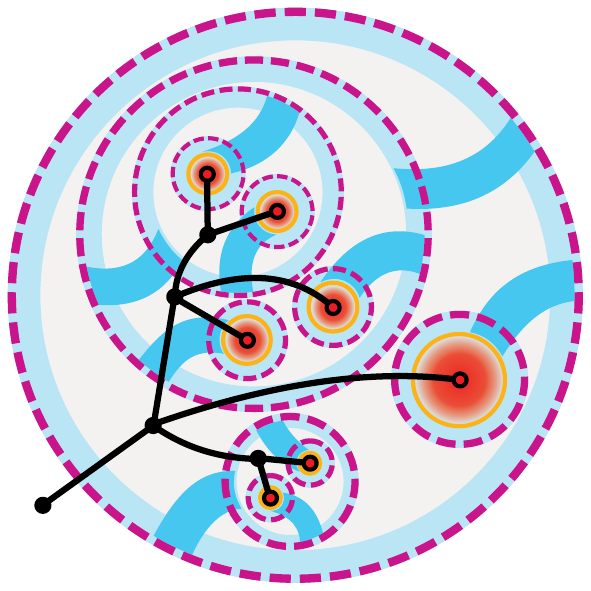}} at (C.center);
        \end{pgfonlayer}{background}
			
        \begin{pgfonlayer}{main}
        \node (C) [v:ghost] {};
            
        \end{pgfonlayer}{main}
        
        \begin{pgfonlayer}{foreground}
        \end{pgfonlayer}{foreground}

    \end{tikzpicture}
    \caption{A diagram of the tree-like structure of nests and vortices.}
    \label{fig:NestTree}
\end{figure}

The proof of the red local structure theorem proceeds by induction, with the red flat wall theorem as the base case. In the base case, the red flat wall theorem yields either a large red grid, in which case we are done, or a weak near-embedding with a single vortex that contains all red vertices in its interior.
This embedding serves as the starting point of the refinement procedure.

Using the wall infrastructure provided by the flat wall theorem, we first construct the required nest and then iteratively apply the red society classification theorem to the current weak near-embedding.
At each step of this process, if the outcome is a red clique or a red grid, we are done.

If instead a flat and homogeneous crosscap or handle transaction is produced, we proceed as follows.
If the transaction is red, we again obtain a red grid.
Otherwise, we apply tools from \cite{KawarabayashiTW2021Quickly} to extend the working surface by attaching a crosscap or handle and routing most of the transaction through it.
This yields a single new vortex cell for further refinement.
The new vortex is disjoint from the added crosscap or handle, contains all red vertices in its interior, and contains a sufficiently large nest which is defined by combining a part of the crosscap or handle transaction with the original nest, allowing the induction to continue.

In the remaining case, the red society classification theorem produces a weak near-embedding in which the vortex is replaced by a bounded number of vortices, each admitting a path decomposition of bounded adhesion and together containing all red vertices in their interiors.
By stitching this embedding along the boundary of the vortex to the already embedded part of the graph, we obtain the final outcome and conclude the proof.

\paragraph{Step 4: Local to global.}
Finally, we are able to use the local structure theorem of \textbf{Step 3} that in the absence of a large red grid, finds the desired tree-decomposition of \zcref{thm:BidimensionalityIntro}.

As we have previously discussed this part follows in a straightforward manner, by employing a fairly standard technique originating from \cite{RobertsonS1991Graph} that allows to turn the red local structure theorem into the desired global theorem based on tree-decompositions.

\section{Preliminaries}\label{sec:preliminaries}

In this section we present a collection of concepts and definitions from graph theory and graph minor structure theory essential to the results of this work.
Most of these are derived directly from \cite{KawarabayashiTW2018New,KawarabayashiTW2021Quickly,GorskySW2025Polynomial}, with only minor aesthetic changes applied.
The only definitions original to this work contained in this section are several variants of standard definitions in the current realm of graph minor structure theory in which we want no red vertices to appear or red vertices to be omnipresent.

\subsection{Basics}
First, we introduce notation for commonly used concepts.
By $\mathbb{N}$ we denote the set of non-negative integers.
Given any two integers $a,b\in\mathbb{N},$ we write $[a,b]$ for the set $\{z\in\mathbb{N} ~\!\colon\!~ a\leq z\leq b\}.$
Notice that the set $[a,b]$ is empty whenever $a>b.$
For any positive integer $c$ we set $[c]\coloneqq [1,c].$

\paragraph{Paths and linkages.}

A \emph{linkage} $\mathcal{L}$ in a graph $G$ is a set of pairwise vertex-disjoint paths.
We say that a path $P$ in $G$ is \emph{internally disjoint} from a set $X \subseteq V(G)$ if $V(P) \cap X$ does not contain any vertex of $P$ that is not an endpoint vertex.
Given a graph $G$ and two subsets $A, B \subseteq V(G)$, an \emph{$A$-path} in $G$ is a path with both endpoints in $A$ and internally disjoint from $A$, and an \emph{$A$-$B$-path} is a path with one endpoint in $A$, the other in $B$, and internally disjoint from $A \cup B$.
An \emph{$A$-$B$-linkage} in $G$ is a linkage consisting of $A$-$B$ paths.
If $H$ is a subgraph of $G$, an \emph{$H$-path} is a $V(H)$-path of length at least one with no edge in $E(H)$.

\paragraph{Surfaces.}
By a \emph{surface} we mean a compact $2$-dimensional manifold with or without boundary.

Given a pair $(\mathsf{h}, \mathsf{c}) \in \mathbb{N} \times [0,2]$ we define $\Sigma^{(\mathsf{h}, \mathsf{c})}$ to be the surface without boundary created from the sphere by adding $\mathsf{h}$ handles and $\mathsf{c}$ crosscaps (see \cite{MoharT2001Graphs} for more details).
If $\mathsf{c} = 0$ the surface $\Sigma^{(\mathsf{h}, \mathsf{c})}$ is an \emph{orientable} surface, otherwise it is called \emph{non-orientable}.
By Dyck's theorem \cite{Dyck1888Beitraege,FrancisW1999Conways}, two crosscaps are equivalent to a handle in the presence of a third crosscap.
Thus the notation $\Sigma^{(\mathsf{h}, \mathsf{c})}$ is sufficient to capture all two-dimensional surfaces without boundary.
We let the \emph{genus} of $\Sigma$ be $2\mathsf{h} + \mathsf{c}$, where $\Sigma^{\mathsf{(\mathsf{h},\mathsf{c})}}$ is a surface to which $\Sigma$ is isomorphic.

\subsection{Tools more specific to graph minor structure theory}

We continue by briefly introducing some key concepts for handling the general structure of $H$-minor-free graphs.
The definitions we introduce here are based on the framework introduced by \cite{KawarabayashiTW2021Quickly} and later adopted and reiterated by \cite{ThilikosW2024Excluding, PaulPTW2024Obstructionsa, PaulPTW2025Local, GorskySW2025Polynomial}.

\paragraph{Separations and tangles.}
Let $G$ be a graph and $k$ be a positive integer.
We denote by $\mathcal{S}_k(G)$ the collection of all separations $(A,B)$ of order less than $k$ in $G$.

An \emph{orientation} of $\mathcal{S}_k(G)$ is a set $\mathcal{O}$ such that for all $(A,B)\in\mathcal{S}_k(G)$ exactly one of $(A,B)$ and $(B,A)$ belongs to $\mathcal{O}$. 
A \emph{tangle} of order $k$ in $G$ is an orientation $\mathcal{T}$ of $\mathcal{S}_k(G)$ such that for all $(A_1,B_1),(A_2,B_2),(A_3,B_3)\in\mathcal{T}$, it holds that $G[A_1]\cup G[A_2]\cup G[A_3]\neq G$.
If $\mathcal{T}$ is a tangle and $(A,B)\in\mathcal{T}$ we call $A$ the \emph{small side} and $B$ the \emph{big side} of $(A,B)$.

Let $G$ be a graph and $\mathcal{T}$ and $\mathcal{D}$ be tangles of $G$.
We say that $\mathcal{D}$ is a \emph{truncation} of $\mathcal{T}$ if $\mathcal{D}\subseteq\mathcal{T}$.
\medskip

Let $G$ and $H$ be graphs as well as $\mathcal{T}$ be a tangle in $G$.
We say that a minor-model $\mu$ of $H$ in $G$ is \emph{controlled} by $\mathcal{T}$ if there does not exist a separation $(A,B)\in\mathcal{T}$ of order less than $|V(H)|$ and an $x \in V(H)$ such that $\mu(x)\subseteq A\setminus B$.

\paragraph{Meshes.}
Let $n,m$ be integers with $n,m\geq 2$.
A \emph{$(n\times m)$-mesh} is a graph $M$ which is the union of paths $M=P_1\cup\cdots\cup P_n\cup Q_1\cup \cdots \cup Q_m$ where
    \begin{itemize}
        \item $P_1,\cdots,P_n$ are pairwise vertex-disjoint, and $Q_1,\cdots,Q_m$ are pairwise vertex-disjoint.
        \item for every $i\in [n]$ and $j\in [m]$, the intersection $P_i\cap Q_j$ induces a path,
        \item each $P_i$ is a  $V(Q_1)$-$V(Q_m)$-path intersecting the paths $Q_1,\cdots Q_m$ in the given order, and each $Q_j$ is a $V(P_1)$-$V(P_m)$-path intersecting the paths $P_1,\cdots, P_h$ in the given order. 
    \end{itemize}
We say that the paths $P_1,\cdots,P_n$ are the \emph{horizontal paths}, and the paths $Q_1,\cdots,Q_m$ are the \emph{vertical paths}.
The union $P_{1} \cup P_{n} \cup Q_{1} \cup Q_{m}$ is a cycle called the \emph{perimeter} of $M$.
The unique cycle in the union $P_{i} \cup P_{i+1} \cup Q_{j} \cup Q_{j+1}$, where $i \in [n - 1]$ and $j \in [m - 1]$, is called a \emph{brick} of $M$.
A mesh $M'$ is a \emph{submesh} of a mesh $M$ if every horizontal (vertical) path of $M'$ is a subpath of a horizontal (vertical) path $M$, respectively.
We write \emph{$n$-mesh} as a shorthand for an $(n \times n)$-mesh.

Let $r \in \mathbb{N}$ with $r\geq 3$, let $G$ be a graph, and $M$ be an $r$-mesh in $G$.
Let $\mathcal{T}_M$ be the orientation of $\mathcal{S}_r$ such that for every $(A,B)\in\mathcal{T}_M$, the set $B\setminus A$ contains the vertex set of both a horizontal and a vertical path of $M$, we call $B$ the \emph{$M$-majority side} of $(A,B)$.
Then $\mathcal{T}_M$ is the tangle \emph{induced} by $M$.
If $\mathcal{T}$ is a tangle in $G$, we say that $\mathcal{T}$ \emph{controls} the mesh $M$ if $\mathcal{T}_M$ is a truncation of $\mathcal{T}$.

\paragraph{Paintings in surfaces.}
A \emph{painting} in a surface $\Sigma$ is a pair $\Gamma = (U,N)$, where $N \subseteq U \subseteq \Sigma$, $N$ is finite, $U \setminus N$ has a finite number of arcwise-connected components, called \emph{cells} of $\Gamma$, and for every cell $c$, the closure $\overline{c}$ is a closed disk where $N_\Gamma(c) \coloneqq \overline{c} \cap N \subseteq \mathsf{bd}(\overline{c})$.
If $|N_\Gamma(c)| \geq 4$, the cell $c$ is called a \emph{vortex}.
We further let $N(\Gamma) \coloneqq N$, let $U(\Gamma) \coloneqq U$, and let $C(\Gamma)$ be the set of all cells of $\Gamma$.
\medskip

Any given painting $\Gamma = (U,N)$ defines a hypergraph with $N$ as its vertices and the set of closures of the cells of $\Gamma$ as its edges.
Accordingly, we call $N$ the \emph{nodes} of $\Gamma$.

\paragraph{$\Sigma$-renditions.}
Let $G$ be a graph and $\Sigma$ be a surface.
A \emph{$\Sigma$-rendition} of $G$ is a triple $\rho = (\Gamma, \sigma, \pi)$, where
\begin{itemize}
    \item $\Gamma$ is a painting in $\Sigma$,
    \item for each cell $c \in C(\Gamma)$, $\sigma(c)$ is a subgraph of $G$, and
    \item $\pi \colon N(\Gamma) \to V(G)$ is an injection,
\end{itemize}
such that
\begin{description}
    \item[R1] $G = \bigcup_{c \in C(\Gamma)}\sigma(c)$,
    \item[R2] for all distinct $c,c' \in C(\Gamma)$, the graphs $\sigma(c)$ and $\sigma(c')$ are edge-disjoint,
    \item[R3] $\pi(N_\Gamma(c)) \subseteq V(\sigma(c))$ for every cell $c \in C(\Gamma)$, and
    \item[R4] for every cell $c \in C(\Gamma)$, we have $V(\sigma(c) \cap \bigcup_{c' \in C(\Gamma) \setminus \{ c \}} (\sigma(c'))) \subseteq \pi(M_\Gamma(c))$.
\end{description}
We write $N(\rho)$ for the set $N(\Gamma)$, let $N_\rho(c) = N_\Gamma(c)$ for all $c \in C(\Gamma)$, and similarly, we lift the set of cells from $C(\Gamma)$ to $C(\rho)$.
If it is clear from the context which $\rho$ is meant, we will sometimes simply write $N(c)$ instead of $N_\rho(c)$, and if the $\Sigma$-rendition $\rho$ for $G$ is understood from the context, we usually identify the sets $\pi(N(\rho))$ and $N(\rho)$ along $\pi$ for ease of notation.

\paragraph{Blank renditions.}
Let $\rho$ be a $\Sigma$-rendition of an annotated graph $(G,R)$.
If $$\pi(N(\rho)) \cup \bigcup \{ V(\sigma(c)) \colon c \in C(\rho) \text{ and } c \text{ is not a vortex}\}$$ is disjoint from $R$, we call $\rho$ a \emph{blank rendition (of $(G,R)$)}.

\paragraph{Societies.}
Let $\Omega$ be a cyclic ordering of the elements of some set which we denote by $V(\Omega)$.
A \emph{society} is a pair $(G,\Omega)$, where $G$ is a graph and $\Omega$ is a cyclic ordering with $V(\Omega)\subseteq V(G)$.
For a given set $S \subseteq V(\Omega)$ a vertex $s \in S$ is an \emph{endpoint} of $S$ if there exists a vertex $t \in V(\Omega) \setminus S$ that immediately precedes or succeeds $s$ in $\Omega$.
We call $S$ a \emph{segment} of $\Omega$ if $S$ has two or less endpoints.

Let $(G,\Omega)$ be a society and let $\Sigma$ be a surface with one boundary component $B$ homeomorphic to the unit circle.
A \emph{rendition} of $(G,\Omega)$ in $\Sigma$ is a $\Sigma$-rendition $\rho$ of $G$ such that the image under $\pi_{\rho}$ of $N(\rho) \cap B$ is $V(\Omega)$ and $\Omega$ is one of the two cyclic orderings of $V(\Omega)$ defined by the way the points of $\pi_{\rho}(V(\Omega))$ are arranged in the boundary $B$.

\paragraph{Traces of paths and cycles.}
Let $\rho$ be a $\Sigma$-rendition of a graph $G$.
For every cell $c \in C(\rho)$ with $|N_\rho(c)| = 2$, we select one of the components of $\mathsf{bd}(c) - N_\rho(c)$.
This selection will be called a \emph{tie-breaker in $\rho$}, and we assume that every rendition comes equipped with a tie-breaker.

Let $G$ be a graph and $\rho$ be a $\Sigma$-rendition of $G$.
Let $Q$ be a cycle or path in $G$ that uses no edge of $\sigma(c)$ for every vortex $c \in C(\rho)$.
We say that $Q$ is \emph{grounded} if it uses edges of $\sigma(c_1)$ and $\sigma(c_2)$ for two distinct cells $c_1, c_2 \in C(\rho)$, or $Q$ is a path with both endpoints in $N(\rho)$.
If $Q$ is grounded we define the \emph{trace} of $Q$ as follows.
Let $P_1,\dots,P_k$ be distinct maximal subpaths of $Q$ such that $P_i$ is a subgraph of $\sigma(c)$ for some cell $c$.
Fix $i \in [k]$.
The maximality of $P_i$ implies that its endpoints are $\pi(n_1)$ and $\pi(n_2)$ for distinct nodes $n_1,n_2 \in N(\rho)$.
If $|N_\rho(c)| = 2$, let $L_i$ be the component of $\mathsf{bd}(c) - \{ n_1,n_2 \}$ selected by the tie-breaker, and if $|N_\rho(c)| = 3$, let $L_i$ be the component of $\mathsf{bd}(c) - \{ n_1,n_2 \}$ that is disjoint from $N_\rho(c)$.
We define $L_i'$ by pushing $L_i$ slightly so that it is disjoint from all cells in $C(\rho)$, while maintaining that the resulting curves intersect only at a common endpoint.
The \emph{trace} of $Q$ is defined to be $\bigcup_{i\in[k]} L_i'$.
If $Q$ is a cycle, its trace thus the homeomorphic image of the unit circle, and otherwise, it is an arc in $\Sigma$ with both endpoints in $N(\rho)$.

\paragraph{Aligned disks and grounded subgraphs.}
Let $G$ be a graph and let $\rho = (\Gamma, \sigma, \pi)$ be a $\Sigma$-rendition of $G$. 
We say that a 2-connected subgraph $H$ of $G$ is \emph{grounded (in $\rho$)} if every cycle in $H$ is grounded and no vertex of $H$ is drawn by $\Gamma$ in a vortex of $\rho$.
A disk in $\Sigma$ is called \emph{$\rho$-aligned} if its boundary only intersects $\Gamma$ in nodes.
If $H$ is planar, we say that it is \emph{flat in $\rho$} if there exists a $\rho$-aligned disk $\Delta \subseteq \Sigma$ which contains all cells $c \in C(\rho)$ with $E(\sigma(c)) \cap E(H) \neq \emptyset$ and $\Delta$ does not contain any vortices of $\Gamma$.

For any $\rho$-aligned disk $\Delta$, we call the subgraph of $G$ that is drawn by $\Gamma$ onto $\Delta$ the \emph{crop of $G$ by $\Delta$ (in $\rho$)}.
Furthermore, the \emph{restriction $\delta'$ of $\rho$ by $\Delta$} is defined as the $\Delta$-rendition that consists of the restriction of both $\Gamma$, $\sigma$, and $\pi$ to $\Delta$.

This allows to define a society associated to $\Delta$ as follows.
Let $V(\Omega_{\Delta})$ be the set of all vertices whose corresponding nodes are drawn in the boundary of $\Delta$ and let $\Omega_{\Delta}$ be the cyclic ordering of $V(\Omega_{\Delta})$ obtained by traversing along the boundary of $\Delta$ in the anticlockwise direction.
Now, let $G_{\Delta}$ be the crop of $G$ by $\Delta$.
We call the society $(G_{\Delta}, \Omega_{\Delta})$ the \emph{$\Delta$-society (in $\rho$)}.
If $\rho$ is clear from the context, we do not mention it.
We also call the restriction of $\rho$ by $\Delta$, the \emph{restriction of $\rho$ to $(G_{\Delta}, \Omega_{\Delta})$.}

Let $\mathcal{P}$ be an $X$-$Y$-linkage in $G$ such that $X \cap V(G_{\Delta}) \subseteq V(\Omega_{\Delta})$ and $Y \subseteq V(G_{\Delta})$ and assume that each path in $\mathcal{P}$ is grounded in $\rho$.
Then we define the \emph{$\Delta$-truncation (in  $\rho$)} of $\mathcal{P}$ to be the $V(\Omega_{\Delta})$-$Y$-linkage in $G_{\Delta}$ which consists of the minimal $V(\Omega_{\Delta})$-$Y$-subpaths of the paths in $\mathcal{P}.$

Let $\rho$ be a rendition of a society $(G, \Omega)$ in the disk $\Delta.$
Given a cycle $C \subseteq G$ that is grounded in $\rho$ we define the \emph{$C$-disk (in $\rho$)} as the unique $\rho$-aligned disk $\Delta' \subseteq \Delta$ bounded by the trace of $C$ in $\rho.$
We also use the terms \emph{$C$-society (in $\rho$)} to denote the $\Delta'$-society in $\rho$ and \emph{$C$-truncation (in $\rho$)} to denote the $\Delta'$-truncation in $\rho$ of an appropriately defined linkage in $G.$

\paragraph{Transactions and their types.}
Let $(G, \Omega)$ be a society. 
A \emph{transaction} in $(G, \Omega)$ is an $A$-$B$-linkage for disjoint segments $A, B$ of $\Omega$ consisting of $V(\Omega)$-paths. 
The inclusion-wise minimal segments $X$ and $Y$ of $\Omega$ for which $\mathcal{P}$ is an $X$-$Y$-linkage are called the \emph{end segments} of $\mathcal{P}$ in $(G, 
\Omega)$.

Let $\mathcal{P}$ be a transaction in a society $(G, \Omega)$.
Suppose that the members of $\mathcal{P}$ can be enumerated as $P_{1}, \ldots, P_{n}$ so that if $x_{i} \in X$ and $y_{i} \in Y$ denote the endpoints of $P_{i}$, then the vertices $x_{1}, \ldots, x_{n}$ appear in the segment $X$ in the listed order or the reverse one, and the vertices $y_{1}, \ldots, y_{n}$ appear in $Y$ in the listed order or the reverse one.
Then we say that $\mathcal{P}$ is \emph{monotone}, and if $P_{1}, \ldots, P_{n}$ are ordered as above, they are \emph{indexed naturally}.

Should the vertices $x_{1}, \ldots, x_{n}, y_{n}, \ldots y_{1}$ appear in $\Omega$ in the listed cyclic ordering or its reverse, we call $\mathcal{P}$ a \emph{planar transaction}, and if the vertices $x_{1}, \ldots x_{n}, y_{1}, \ldots y_{n}$ appear in $\Omega$ in the listed cyclic ordering or its reverse, we call $\mathcal{P}$ a \emph{crosscap transaction}.
The paths $P_{1}$ and $P_{2}$ are called the \emph{boundary paths} of $\mathcal{P}$.

\paragraph{Strips and their societies.}
Let $H$ be a subgraph of a graph $G$.
An \emph{$H$-bridge} in $G$ is a connected subgraph $B$ of $G$ such that $E(B) \cap E(H) = \emptyset$ and either $E(B)$ consists of a unique edge with both ends in $H$, or $B$ is constructed from a component $C$ of $G - V(H)$ and the non-empty set of edges $F \subseteq E(G)$ with one end in $V(C)$ and the other in $V(H)$, by taking the union of $C$, the endpoints of the edges in $F$, and $F$ itself.
The vertices in $V(B) \cap V(H)$ are called the \emph{attachments} of $B$.

We let $H$ denote the subgraph of $G$ obtained from the union of elements of a monotone transaction $\mathcal{P}$ that is indexed naturally as $P_{1}, \ldots, P_{n}$ by adding the elements of $V(\Omega)$ as isolated vertices.
Further, we define $H'$ as the subgraph of $H$ consisting of $\mathcal{P}$ and all vertices of $X \cup Y$.
Consider all $H$-bridges of $G$ with at least one attachment in $V(H') \setminus V(P_{1} \cup P_{n})$, and for each such $H$-bridge $B$ let $B'$ denote the graph obtained from $B$ by deleting all attachments that do not belong to $V(H')$.
We let $G_{1}$ denote the union of $H'$ and all graphs $B'$ as above.

The \emph{$\mathcal{P}$-strip society in $(G, \Omega)$} is defined as the society $(G_{1}, \Omega_{1})$, where $\Omega_{1}$ is the concatenation of the segment $X$ ordered from $x_{1}$ to $x_{n}$, and the segment $Y$ ordered from $y_{n}$ to $y_{1}$.
If the $\mathcal{P}$-strip society admits a vortex-free rendition in a disk, we call $\mathcal{P}$ a \emph{flat transaction}.
Further, if no edge of $G$ has an endpoint in $V(G_{1}) \setminus V(P_{1} \cup P_{n})$ and the other endpoint in $V(G) \setminus V(G_{1})$, then we call $\mathcal{P}$ \emph{isolated}.
Let $X', Y'$ be the two distinct segments of $\Omega$ that have one endpoint in $X$ and the other in $Y$.
Note that $V(\Omega) = X \cup Y \cup X' \cup Y'$.
We say that $\mathcal{P}$ is \emph{separating} if it is isolated and there exists no $X'$-$Y'$-path in $G - V(G_{1})$.

\paragraph{Handle transactions.}
A transaction $\mathcal{P}$ of order $2n$ in a society $(G, \Omega)$, for a positive integer $n$, is called a \emph{handle transaction} if $\mathcal{P}$ can be partitioned into two transactions $\mathcal{R}, \mathcal{Q}$ each of order $n$, such that both are planar, and $S^\mathcal{R}_1, S^\mathcal{Q}_1, S^\mathcal{R}_2, S^\mathcal{Q}_2$ are segments partitioning $\Omega$, with $\mathcal{X}$ being a $S^\mathcal{X}_1$-$S^\mathcal{X}_2$-linkage for both $\mathcal{X} \in \{ \mathcal{R} , \mathcal{Q} \}$, and the segments are found on $\Omega$ in the order they were listed above.
We call a handle transaction $\mathcal{P} = \mathcal{R} \cup \mathcal{Q}$ \emph{isolated in $G$}, respectively \emph{flat in $G$}, if both the $\mathcal{R}$-strip society and the $\mathcal{Q}$-strip society of $(G, \Omega)$ are isolated in $G$, respectively flat in $G$.

\paragraph{Cylindrical renditions.}
A rendition of a society $(G, \Omega)$ in the disk with a unique vortex $c_{0}$ is called a \emph{cylindrical rendition} of $(G, \Omega)$ \emph{around} $c_{0}$.

If $(G, \Omega)$ is a society with a cylindrical rendition $\rho$ around a vortex $c_{0}$ and $\mathcal{P}$ is a transaction in $(G, \Omega)$, we call $\mathcal{P}$ \emph{exposed} if for every path $p \in \mathcal{P}$ there exists an edge $e \in E(P) \cap \sigma(c_{0})$.

\paragraph{Nests and radial linkages.}
Let $\rho$ be a rendition of a society $(G, \Omega)$ in a disk $\Delta$.
A \emph{nest (in $\rho$)} is a set of disjoint cycles $\mathcal{C} = \{ C_{1}, \ldots, C_{s} \}$ in $G$ such that each of them is grounded in $\rho$, and if $\Delta_{i}$ is the $C_{i}$-disk for $i \in [s]$, then every vortex of $\rho$ is contained in $\Delta_{1}$ and $\Delta_{1} \subseteq \ldots \subseteq \Delta_{s} \subseteq \Delta$.
We call $C_{1}$ the \emph{inner cycle} of $\mathcal{C}$ and $C_{s}$ the \emph{outer cycle} of $\mathcal{C}$ respectively.
Moreover, we call a $V(\Omega)$-$V(C_{1})$-linkage $\mathcal{R}$ a \emph{radial linkage (in $\rho$) for $\mathcal{C}$} if all paths in $\mathcal{R}$ are grounded in $\rho$ and internally disjoint from $V(\Omega).$

If $(G, \Omega)$ is a society with a nest $\mathcal{C}$ in a rendition $\rho$ of $(G, \Omega)$ in a disk, we say that a radial linkage $\mathcal{R}$ for $\mathcal{C}$ is \emph{orthogonal to $\mathcal{C}$} if for all $C \in \mathcal{C}$ and all $R \in \mathcal{R}$ the graph $C \cap R$ is a path.
Similarly, we say that a transaction $\mathcal{P}$ in $(G, \Omega)$ is \emph{orthogonal to $\mathcal{C}$} if for all $C \in \mathcal{C}$ and all $P \in \mathcal{P}$ the graph $C \cap P$ consists of exactly two paths.

\paragraph{Inner and outer graphs of a cycle.}
Let $(G, \Omega)$ be a society with a $\Sigma$-rendition $\rho.$
Further, let $C$ be a grounded cycle whose trace bounds a disk $\Delta_C$ and the $\Delta_C$-society $(G', \Omega').$
We call $G' \cup C$ the \emph{inner graph of $C$ (in $\rho$)} and call $G'$ itself the \emph{proper inner graph of $C$ (in $\rho$).}
Let $B = \pi(N(\rho) \cap \mathsf{bd}(\Delta_C)).$
We define the \emph{proper outer graph of $C$ (in $\rho$)} as $G'' \coloneqq G[B \cup (V(G) \setminus V(G'))]$ and call $G'' \cup C$ the \emph{outer graph of $C$ (in $\rho$).}

\paragraph{Depth of vortices.}
Let $G$ be a graph and $\rho$ be a $\Sigma$-rendition of $G$ with a vortex cell $c_0.$
Notice that $c_0$ defines a society $(\sigma(c_0), \Omega_{c_0})$, where $V(\Omega_{c_0})$ is the set of vertices of $G$ corresponding $N_\rho(c_0).$
The ordering $\Omega_{c_{0}}$ is obtained by traversing along the boundary of the closure of $c_0$ in anti-clockwise direction.
We call $(\sigma(c_0),\Omega_{c_0})$ as obtained above the \emph{vortex society} of $c_0.$

We define the \emph{depth} of a society $(G, \Omega)$ as the maximum cardinality of a transaction in $(G, \Omega).$
The \emph{depth} of the vortex $c_{0}$ is thereby defined as the depth of its vortex society.

Given a $\Sigma$-rendition $\rho$ with vortices, we define the \emph{breadth of $\rho$} as the number of vortex cells of $\rho$ and the \emph{depth of $\rho$} as the maximum depth of its vortex societies.

\subsection{Surface walls}\label{sec:surfacwalls}

In our local structure theorem we want to ensure that any part of the surface we add is sufficiently represented by a lot of grid-like infrastructure.
This part is inspired by the work of Thilikos and Wiederrecht on excluding graphs of bounded genus \cite{ThilikosW2024Excluding} and essentially the same definitions are featured in \cite{PaulPTW2024Obstructionsa, PaulPTW2025Local, GorskySW2025Polynomial}.

\paragraph{Annulus walls.}
Let $m,n$ be positive integers.
The \emph{$(n\times m)$-annulus grid} is the graph obtained from the $(n\times m)$-grid by adding the edges $\{ \{(i,1),(i,n)\} ~\!\colon\!~\ i\in[n] \}$.
The \emph{elementary $(n\times m)$-annulus wall} is the graph obtained from the $(n\times 2m)$-annulus grid by deleting all edges in the following set
\begin{align*}
    \big\{  \{(i,j),(i+1,j) \}  ~\!\colon\!~ i\in[n-1],\text{ }j\in[2m]\text{, and }i\not\equiv j\mod 2 \big\}.
\end{align*}
An \emph{$(n \times n)$-annulus wall} is a subdivision of the elementary $(n \times n)$-annulus wall.
We also write \emph{$n$-annulus wall or grid} as an abbreviation for an $(n \times n)$-annulus wall or grid.
\medskip

One can also see an annulus $(n\times m)$-wall as the graph obtained by completing the horizontal paths of a wall to cycles instead of discarding the vertices of degree one.
This viewpoint will be very helpful in the following constructions.
An $n$-annulus wall contains $n$ cycles $C_1,\dots,C_n$, such that $C_i$ consists exactly of the vertices of the $i$th row of the original wall.
We refer to these cycles as the \emph{base cycles} of the $n$-annulus wall.

\paragraph{Cylindrical meshes.}
In some settings it will be easier to work with a version of meshes that is cylindrical, since it is easier to argue for their existence than for annulus walls or grids.

Let $m,n$ be positive integers, let $M$ be a graph, and let $C_1, \ldots, C_m$ be cycles and $P_1, \ldots , P_n$ be paths in $M$ such that the following holds for all $i \in [m]$ and $j \in [n]$:
\begin{itemize}
    \item $C_1, \ldots , C_m$ are pairwise vertex-disjoint, $P_1, \ldots, P_n$ are pairwise vertex-disjoint, and $M = C_1 \cup \cdots \cup C_m \cup P_1 \cup \cdots \cup P_n$.

    \item $C_i \cap P_j$ is a path, and if $i \in \{ 1, m \}$ or $j \in \{ 1, n \}$, then $C_i \cap P_j$ has exactly one vertex,

    \item when traversing $C_i$ starting from an endpoint of $P_1\cap C_i$, then either the paths $P_1, \ldots , P_n$ are encountered in the order listed or the next $P_j$ one encounters is $P_n$ and from here the paths are encountered in the order $P_n,\dots,P_1$, and

    \item $P_j$ has one end in $C_1$ and the other in $C_m$, and when traversing $P_j$ starting from its endpoint on $C_1$, the cycles $C_1, \ldots , C_m$ are encountered in the order listed.
\end{itemize}
If the above conditions hold for $M$, we call $M$ an \emph{$(n \times m)$-cylindrical mesh}.
The cycles $C_1, \ldots, C_m$ are called the \emph{concentric cycles}, or \emph{cycles}, of $M$ and the paths $P_1, \ldots , P_n$ are called the \emph{radial paths}, or \emph{rails}, of $M$.
According to this definition, annulus walls are cylindrical meshes.
We also call $(n \times n)$-cylindrical meshes \emph{$n$-cylindrical meshes}.

Let $n, m \geq 3$ be integers, $G$ be a graph, and $M$ be an $(n \times m)$-cylindrical mesh in $G$.
Similarly to meshes, let $\mathcal{T}_{M}$ be the orientation of $\mathcal{S}_{r}$, where $r = \min \{ n, m \}$, such that for every $(A, B) \in \mathcal{T}_{M}$, the set $B \setminus A$ contains the vertex set of both a concentric cycle and a radial path of $M$, we call $B$ the \emph{$M$-majority side of $(A, B)$.}
Then $\mathcal{T}_{M}$ is the tangle \emph{induced} by $M$.
If $\mathcal{T}$ is a tangle in $G$, we way that $\mathcal{T}$ controls the mesh $M$ if $\mathcal{T}_{M}$ is a truncation of $\mathcal{T}$.

\paragraph{Wall segments.}
Let $n$ be a positive integer.
An \emph{elementary $n$-wall-segment} is the graph $W_0$ obtained from the $(n\times 8n)$-grid by deleting all edges in the following set
\begin{align*}
    \big\{  \{(i,j),(i+1,j) \}  ~\!\colon\!~ i\in[n-1],\text{ }j\in[8n]\text{, and }i\not\equiv j\mod 2 \big\}.
\end{align*}
The vertices in $\{ (i,1)  ~\!\colon\!~ i\in[n] \}$ are said to be the \emph{left boundary} of the segment, while the vertices in $\{ (i,n)  ~\!\colon\!~ i\in[n] \}$ form the \emph{right boundary of the segment}.
Finally, we refer to the vertices in $\{ (1,i) ~\!\colon\!~ i=2j\text{, }j\in[1,8n] \}$ as the \emph{top boundary}.
\smallskip

An \emph{elementary $n$-handle-segment} is the graph obtained from the elementary $n$-wall-segment $W_0$ by adding the following edges, which we call \emph{handle edges},
\begin{align*}
    &\big\{ \{(1,2i),(1,6n+2-2i)\}  ~\!\colon\!~  i\in[1,n]  \big\}\\
    \cup \ &\big\{ \{(1,2i),(1,8n+2-2i)\}  ~\!\colon\!~ i\in[n+1,2n]  \big\}.
\end{align*}

An \emph{elementary $n$-crosscap-segment} is the graph obtained from the elementary $n$-wall-segment $W_0$ by adding the following edges, which we call \emph{crosscap edges},
\begin{align*}
    \big\{ \{(1,2i),(1,4n+2i)\}  ~\!\colon\!~  i\in[1,2n]  \big\}.
\end{align*}

An \emph{elementary $n$-vortex-segment} is the graph obtained from the disjoint union of two elementary $n$-wall segments $W_0$ and $W_1$ by making the $i$th top boundary vertex of $W_0$ adjacent to the $i$th top boundary vertex of $W_1$ for each $i \in [4n]$, and by making the $j$th left boundary vertex of $W_1$ adjacent to the $j$th right boundary vertex of $W_1$ for each $j \in [n]$.

We denote the $(n\times 4n)$-annulus wall defined on the vertex set of $W_1$ as above by $W$.
The base cycle $C_1,\dots C_n$ of $W$ are assumed to be ordered such that $C_n$ contains all vertices adjacent to $W_0$ and we call $\{ C_1,\ldots, C_s\}$ the \emph{nest} of the elementary $n$-vortex segment.
We refer to $C_n$ as the \emph{outer cycle} and to $C_1$ as the \emph{inner cycle} of the elementary $n$-vortex segment.
Finally notice that there exist $4n$ pairwise disjoint ``vertical'' paths which are orthogonal to both the horizontal paths of $W_0$ and the cycles $C_1,\dots, C_n$.
The family of these paths is called the \emph{rails} of the elementary $n$-vortex-segment.
\smallskip

In all four types of segments we refer to the elementary wall segment $W_0$ as the \emph{base}.
If we do not want to specify the \textit{type} of an elementary wall-, handle-, crosscap-, or vortex-segment, we simply refer to the graph as a \emph{elementary $n$-segment}, or \emph{elementary segment} if $n$ is not specified.

Let $n$ and $\ell$ be positive integers and let $S_1,\dots,S_\ell$ be elementary $n$-segments.
The \emph{cylindrical closure} of $S_1,\dots,S_{\ell}$ is the graph obtained by introducing, for every $i\in[\ell-1]$ and every $j\in[n]$ and edge between the $j$th vertex of the right boundary of $S_i$ and the $j$th vertex of the left boundary of $S_{i+1}$ together with edges between the $j$th vertex of the right boundary of $S_{\ell}$ and the $j$th vertex of the left boundary of $S_1$.

\paragraph{Surface-walls.}
The \emph{elementary extended $n$-surface wall} with \emph{$h$ handles, $c$ crosscaps, and $b$ vortices} is the graph obtained from the cylindrical closure of $n$-segments $S_1,\dots,S_{h+c+b+1}$ such that there are exactly one elementary $n$-wall-segment, $h$ elementary $n$-handle-segments, $c$ elementary $n$-crosscap-segments, and $b$ elementary $n$-vortex-segments among the $S_i$.
An \emph{extended $n$-surface-wall} is a subdivision of an elementary $n$-surface-wall.
We refer to the tuple $(h,c,b)$ as the \emph{signature} of the extended $n$-surface-wall with $h$ handles, $c$ crosscaps, and $b$ vortices.
An extended surface wall without vortices is also simply called a \emph{surface wall}.

Notice that every extended $n$-surface-wall with signature $(h,c,b)$ contains an $(n \times 4(h+c+b+1)n)$-annulus-wall consisting of $n$ cycles.
We refer to this wall as the \emph{base wall} of the extended $n$-surface-wall.
Let $C_1,\dots,C_n$ be the base cycles of the base wall.
We will usually assume that $C_1$ is the cycle that contains all top boundary vertices of all segments involved, while $C_n$ can be seen as the ``outermost'' cycle.
We refer to $C_n$ as the \emph{simple cycle} of the extended $n$-surface-wall.
Moreover, we refer to the tangle induced by the base wall as the tangle \emph{induced} by the extended $n$-surface wall.

\subsection{Excluding a red clique minor}\label{sec:excludeclique}

Any attempt at describing the structure of an annotated graph $(G, R)$ in the the regime of red minors needs to deal with the presence of a large clique minor in $G$.
This is in fact the starting point of the structural results in \cite{ProtopapasTW2025Colorful}.
The following tool is a simple instantiation of Theorem 3.1 in \cite{ProtopapasTW2025Colorful}, for the case in which there is a single coloured set of vertices in the graph.

\medskip
Let $t$ be a positive integer and $\mu$ be a $K_{t}$-minor-model in $G$.
Notice that for every separation $(A, B) \in \mathcal{S}_{t}$ there exists a unique $X \in \{ A, B \}$, $Y \in \{ A, B \} \setminus X$, such that $\mu(u) \subseteq X \setminus Y$ for some $v \in V(K_{t})$.
We call $X$ the \emph{$\mu$-majority side} of $(A, B)$.
Moreover, if we let $\mathcal{T}_{\mu}$ be the orientation of $\mathcal{S}_{t}$ obtained by taking all separations $(A, B) \in \mathcal{S}_{t}$ such that $B$ is the $\mu$-majority side of $(A, B)$, then $\mathcal{T}$ is a tangle.
We call $\mathcal{T}_{\mu}$ the tangle \emph{induced by $\mu$.}

\begin{proposition}[Protopapas, Thilikos, and Wiederrecht \cite{ProtopapasTW2025Colorful}]\label{prop:redClique}
    Let $t$ and $k$ be positive integers with $k \geq \lfloor \nicefrac{3t}{2} \rfloor + t$.
    Let $(G,R)$ be an annotated graph such that $G$ contains a minor-model $\mu$ of $K_k.$
    Then one of the following is true:
    \begin{enumerate}
    
        \item There exists a separation $(A,B) \in \mathcal{T}_{\mu}$ of order at most $t-1$ such that $(B \setminus A) \cap R = \emptyset,$ or
            
        \item there exists a red minor-model $\mu'$ of $K_t$ in $(G,R)$ such that $\mathcal{T}_{\mu'}$ is a truncation of $\mathcal{T}_{\mu}.$

    \end{enumerate}
    Furthermore, there exists an algorithm that takes as input $t$, $(G,R)$, and $\mu$ as above and finds one of the two outcomes in time $\mathbf{poly}(t) \cdot |E(G)|.$
\end{proposition}

\section{A flat wall theorem for red minors}\label{sec:flatwall}

In this section we prove a variant of the Flat Wall Theorem of Robertson and Seymour~\cite{RobertsonS1995Graph} for red minors in annotated graphs.
The techniques and theorems given here provide the basis for the rest of our efforts.

\subsection{The flat wall theorem}

Originally proven by Robertson and Seymour \cite{RobertsonS1995Graph}, the Flat Wall Theorem now has multiple proofs guaranteeing polynomial bounds, including a relatively simple proof~\cite{KawarabayashiTW2018New}, and a proof with close-to optimal bounds~\cite{Chuzhoy2015Improved}.
The statement below is the one found in \cite{GorskySW2025Polynomial}
\medskip

Let $n \ge 2$ be an integer.
Let $G$ be a graph, and let $M \subseteq G$ be an $n$-mesh.
We say that $M$ is a \emph{flat mesh} in $G$ if there exists a sphere-rendition $\rho$ of $G$ with a single vortex $c_0$ such that $M$ is flat in $\rho$ and the trace of the perimeter of $M$ in $\rho$ separates all vertices in $N(\rho) \cap V(M)$ from $c_0$.
We say that $\rho$ \emph{witnesses} the flatness of $M.$

\begin{figure}[ht]
    \centering
    \begin{tikzpicture}

        \pgfdeclarelayer{background}
		\pgfdeclarelayer{foreground}
			
		\pgfsetlayers{background,main,foreground}
			
        \begin{pgfonlayer}{main}
        \node (C) [v:ghost] {};

        \node (Mpos) [v:ghost,position=180:41mm from C] {};

        \node(M) [v:ghost,position=0:0mm from Mpos] {
            \begin{tikzpicture}

                \pgfdeclarelayer{background}
		          \pgfdeclarelayer{foreground}
			
		          \pgfsetlayers{background,main,foreground}

                \begin{pgfonlayer}{background}
                    \pgftext{\includegraphics[width=8.15cm]{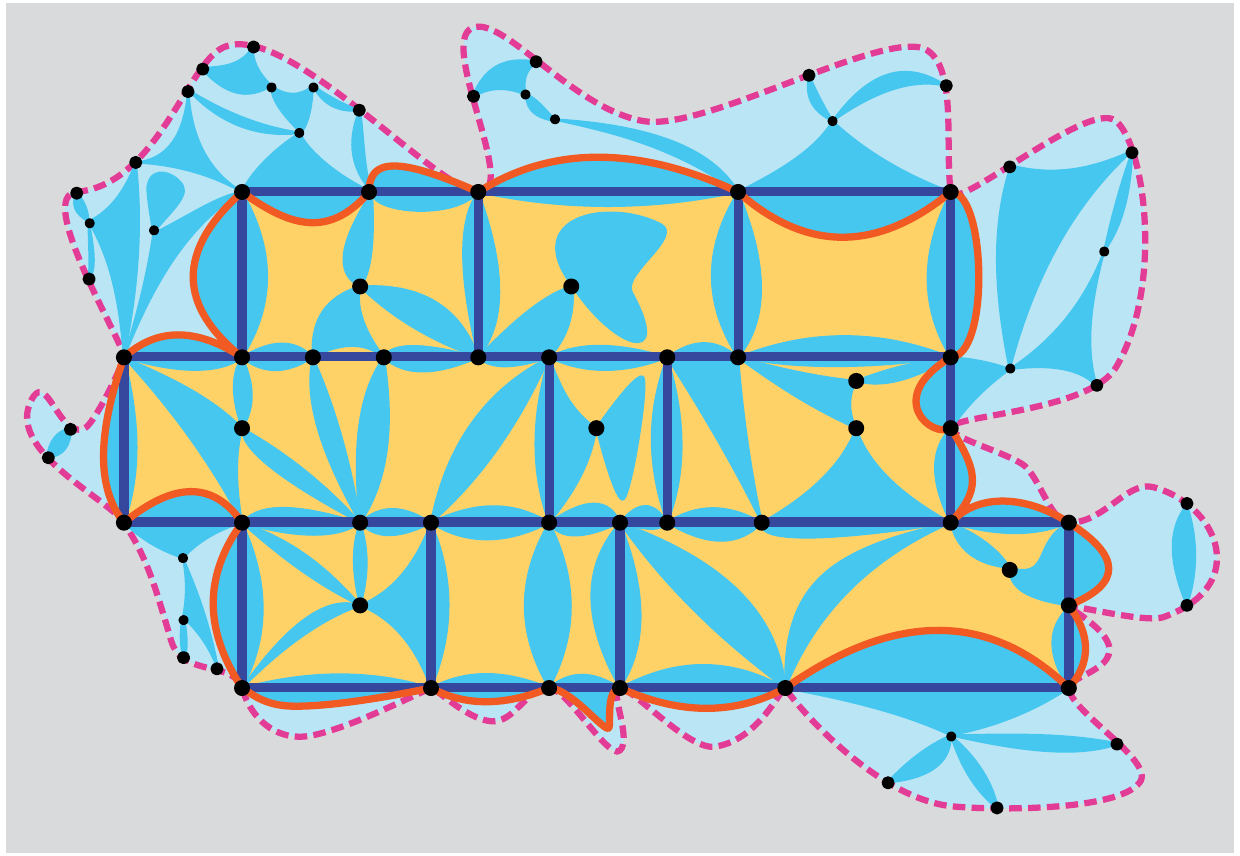}} at (C.center);
                \end{pgfonlayer}{background}
			
                \begin{pgfonlayer}{main}
                    \node (C) [v:ghost] {};
            
                \end{pgfonlayer}{main}
        
                \begin{pgfonlayer}{foreground}
                \end{pgfonlayer}{foreground}

            \end{tikzpicture}
        };

        \node (Rpos) [v:ghost,position=0:41mm from C] {};

        \node(R) [v:ghost,position=0:0mm from Rpos] {
            \begin{tikzpicture}

                \pgfdeclarelayer{background}
		          \pgfdeclarelayer{foreground}
			
		          \pgfsetlayers{background,main,foreground}

                \begin{pgfonlayer}{background}
                    \pgftext{\includegraphics[width=8.15cm]{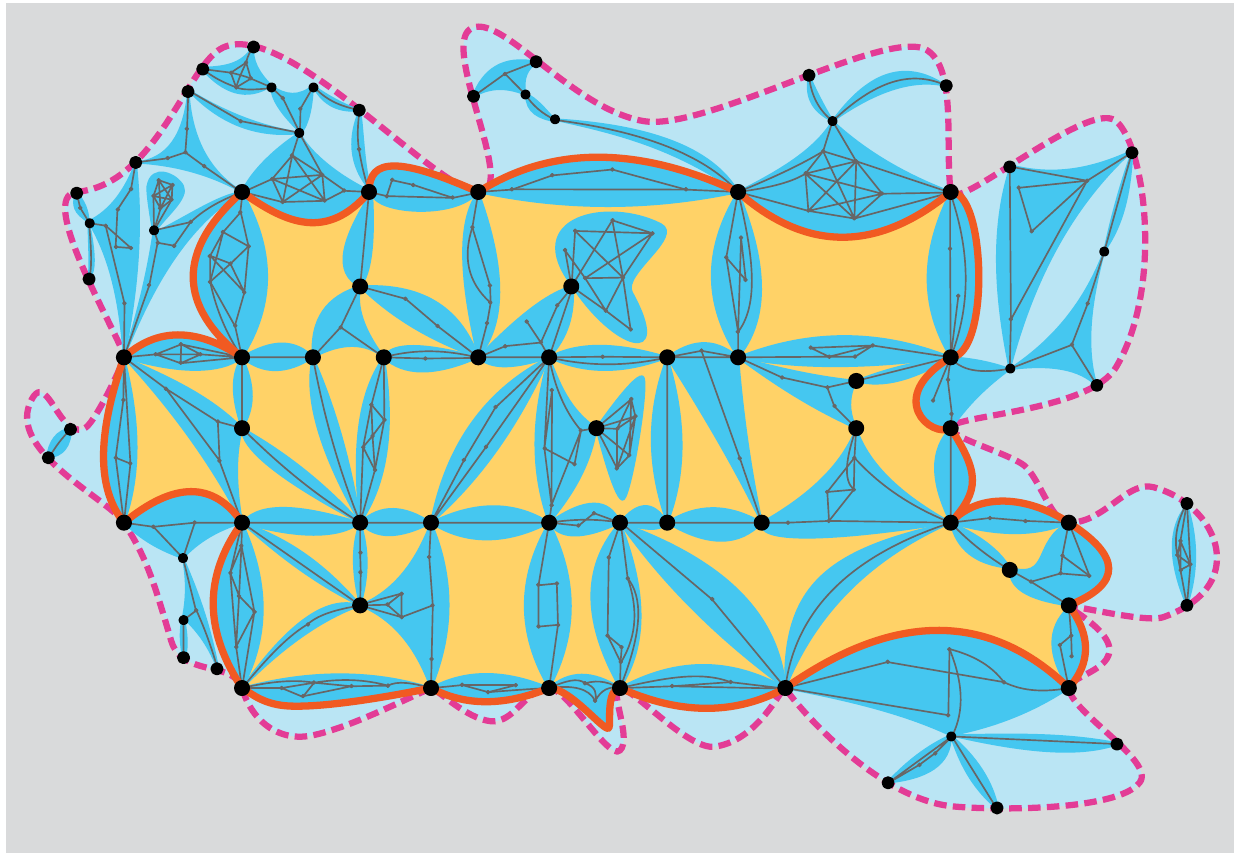}} at (C.center);
                \end{pgfonlayer}{background}
			
                \begin{pgfonlayer}{main}
                    \node (C) [v:ghost] {};
            
                \end{pgfonlayer}{main}
        
                \begin{pgfonlayer}{foreground}
                \end{pgfonlayer}{foreground}

            \end{tikzpicture}
        };

        \node (Rlabel) [v:ghost,position=270:31mm from R] {(ii)};

        \node (Mlabel) [v:ghost,position=270:31mm from M] {(i)};
            
        \end{pgfonlayer}{main}
        
        \begin{pgfonlayer}{foreground}
        \end{pgfonlayer}{foreground}

    \end{tikzpicture}
    \caption{An illustration of a flat mesh $M$ together with a rendition $\rho$ witnessing its flatness. In (i) the minor model of $M$ is indicated by the thick dark blue lines and (ii) depicts the rendition $\rho$, including the contents of the cells. In both pictures, the pink dashed line indicates the boundary of $c_0$, with the interior of $c_0$ being marked in grey. The black dots represent the nodes of $\rho$ and the blue blobs against the light blue, respectively the yellow background, depict the cells of $\rho$. The orange line marks the trace $T_P$ of the perimeter of $M$. The yellow area -- together with the cells separated from $c_0$ by $T_P$ -- describes the disc bounded by $T_P$ that captures the interior of $M$.}
    \label{fig:FlatMesh}
\end{figure}
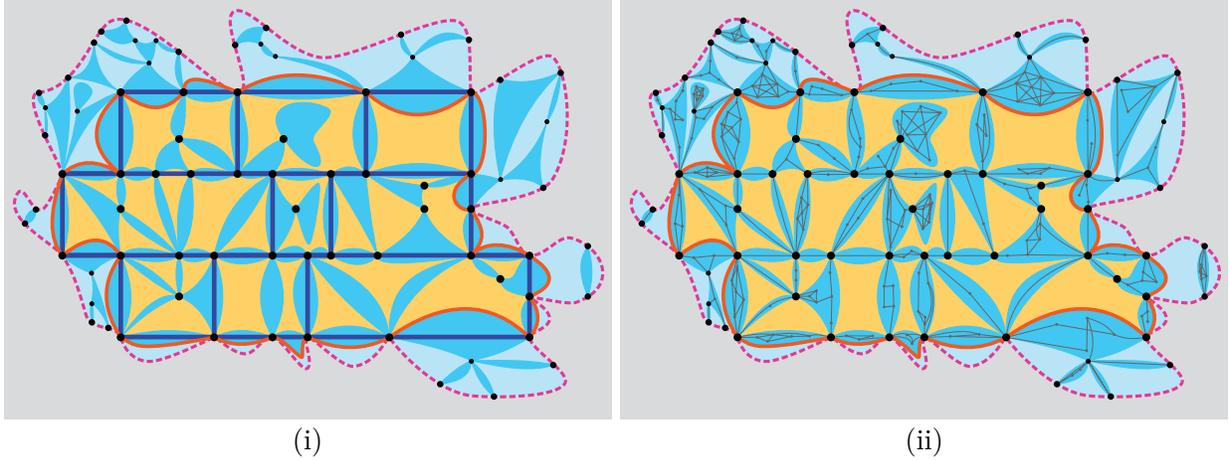

\begin{proposition}[Gorsky, Seweryn, and Wiederrecht \cite{GorskySW2025Polynomial}]\label{prop:FlatMesh}
    For all integers $t \ge 5$, $n' \ge 2$, there exist integers \(k = k(t) = t^2-t\) and \(n = n(t, n') = 7t^2n'+535t^3+44t^2 \in\mathbf{O}(t^2(n'+t))\) such that the following holds.
    
    Let \(G\) be a graph with an \(n\)-mesh \(M \subseteq G\). Then either
    \begin{itemize}
        \item there exists a $K_{t}$-minor-model $\mu$ in $G$ such that $\mathcal{T}_{\mu}$ is a truncation of $\mathcal{T}_{M}$, or
        \item there exist an \(n'\)-submesh \(M' \subseteq M\), and a set \(Z \subseteq V(G) \setminus V(M')\) with \(|Z| \le k\) such that \(M'\) is a flat mesh in \(G - Z\).
    \end{itemize}
    Furthermore, there exists a \(\mathbf{O}(t^2|E(G)|)\)-time algorithm which finds either \(\mu\) or \((M', Z)\) as above together with a sphere-rendition \(\rho\) of \(G - Z\) witnessing the flatness of $M'$.
\end{proposition}

\subsection{Proof of the annotated variant}

Let $(G, R)$ be an annotated graph with an $n$-mesh $M$ where $n \geq 2$, and $\rho$ be a $\Sigma$-rendition where $\Sigma$ is the sphere such that $M$ is grounded in $\rho$.
Let $B$ be a brick of $M$.
Then let $\Delta_{B} \subseteq \Sigma$ be the disk bounded by the trace of $B$ in $\rho$ that avoids the perimeter of $M$.
Let $H_{B} \subseteq G$ be defined as the union of $\sigma(c)$ for all cells $c$ of $\rho$ such that either $\sigma(c)$ contains an edge of $B$ or is contained in $\Delta_{B}$.
We call $H_{B}$ the \emph{subgraph of $G$ induced by $B$ (in $\rho$).}

We say that $M$ is \emph{red (in $\rho$)} if $V(H_{B}) \cap B$ contains a vertex of $R$ for every brick $B$ of $M$ and analogously, we call $M$ \emph{blank (in $\rho$)} if $V(H_{B}) \cap B$ contains a vertex of $R$ for no brick $B$ of $M$.
In either case we say that $M$ is \emph{homogeneous (in $\rho$).}
In case $M$ is a flat mesh in $G$ then whenever we say that it is red, blank, or homogeneous in $G$, we implicitly mean with respect to the sphere-rendition of $G$ that witnesses the flatness of $M$.

The following lemma shows how from a flat mesh in an annotated graph we can find a submesh that is homogeneous with a rendition witnessing the flatness in a respectful way.

\begin{lemma}\label{lemma:flatMeshToHomegeneousMesh} For every integer $r \geq 2$ the following holds.
Let $(G, R)$ be an annotated graph with a flat $(r + 2)^{2}$-mesh $M \subseteq G$ and $\rho$ be a sphere-rendition of $G$ with a single vortex $c_{0}$ witnessing the flatness of $M$.

Then, there exists an $r$-submesh $M' \subseteq M$ and a $\rho$-aligned disk $\Delta$ in the sphere that contains $c_{0}$ and avoids $M'$ such that
\begin{itemize}
\item $M'$ is flat and homogeneous in $G$ witnessed by the sphere-rendition $\rho'$, where
\item $\rho'$ is obtained from $\rho$ by removing all cells contained in $\Delta$ and replacing them with a single vortex $c'_{0} = \Delta$.
Also, if $M'$ is blank then $\rho'$ is blank.
\end{itemize}
Moreover, there exists an algorithm that $M'$ and $\rho'$ in time $\mathbf{poly}(r) \cdot |E(G)|$.
\end{lemma}
\begin{proof}
Let us denote the vertical paths of $M$ by $P_{1}, \ldots P_{(r + 2)^{2}}$ and the horizontal paths of $M$ by $Q_{1}, \ldots, Q_{(r + 2)^{2}}$.
Next, for each $i \in [r + 2]$ let $\alpha_{i} = i + (i - 1)(r + 2)$.
Let $M_{\alpha}$ be the $(r + 2)$-mesh contained in the union of the vertical paths $P_{\alpha_{i}}$ and $Q_{\alpha_{i}}$ for all $i \in [r + 2]$.
Moreover, observe that for every brick $B$ of $M_{\alpha}$, the subgraph $H_{B}$ of $G$ induced by $B$ contains an $(r + 2)$-mesh $M_{B}$ whose perimeter is $B$.

Now there are two cases to examine: Either $V(H_{B}) \cap R \neq \emptyset$ for all bricks $B$ of $M_{\alpha}$, or there exists a brick $B'$ of $M_{\alpha}$ such that $V(H_{B}) \cap R = \emptyset$.

In the first case, we let $M'$ be the $r$-submesh of $M_{\alpha}$ that avoids its perimeter which by assumption is red.
In this case we may simply conclude with $\rho$ being the witnessing sphere-rendition as well.

In the second case, we let $M'$ be the $r$-submesh of $M_{B'}$ that avoids $B'$ which by assumption is blank.
It remains to argue that there is a blank sphere-rendition $\rho'$ of $G$ witnessing the flatness of $M'$ with the required properties.
This follows easily by taking the sphere-rendition $\rho$ of $G$ witnessing the flatness of $M$ and altering it by making the disk bounded by the trace of $B'$ in $\rho$ that avoids $M'$ into its unique vortex $c'_{0}$.
The fact that $\rho$ is also blank follows from the fact that $M_{B'}$ is blank.
\end{proof}

We proceed with the main the variant of the flat wall theorem we will use throughout the paper.

\begin{theorem}\label{thm:FlatRedMesh}
    There exists a function $\mathsf{rfw}_{\ref{thm:FlatRedMesh}} \colon \mathbb{N}^{2} \to \mathbb{N}$ such that for all integers $t \geq 5$ and $r \geq 2$ the following holds.

    Let $(G, R)$ be an annotated graph with an $\mathsf{rfw}_{\ref{thm:FlatRedMesh}}(t, r)$-mesh $M \subseteq G$.
    Then there exists either
    \begin{itemize}
        \item a separation $(A,B) \in \mathcal{T}_M$ of order at most $t-1$ such that $(B \setminus A) \cap R = \emptyset$,
        \item a red $K_{t}$-minor-model $\mu$ in $(G,R)$ such that $\mathcal{T}_{\mu}$ is a truncation of $\mathcal{T}_{M}$, or
        \item an $r$-submesh $M' \subseteq M$ and a set $Z \subseteq V(G) \setminus V(M')$ with $|Z| < 9t^{2}$ such that $M'$ is a homogeneous, flat mesh in $G - Z$.
        Moreover, in case that $M'$ is blank, there is a blank sphere-rendition $\rho$ of $(G - Z, R)$ witnessing the flatness of $M'$.
    \end{itemize}
    Furthermore, $\mathsf{rfw}_{\ref{thm:FlatRedMesh}}(t, r) \in \mathbf{O}(t^{2}(r^{2} +  t))$ and there exists an algorithm that either find $\mu$ or $M', Z$, and $\rho$ as above in time $\mathbf{poly}(t + r) \cdot |E(G)|$.
\end{theorem}
\begin{proof}
We begin by setting up the function $\mathsf{rfw}_{\ref{thm:FlatRedMesh}}$ to be
\begin{align*}
\mathsf{rfw}_{\ref{thm:FlatRedMesh}}(t, r) \ \coloneqq \ 63t^{2}(r + 2)^{2} + 14445t^{3} + 396t^{2}.
\end{align*}
This allows us to apply \zcref{prop:FlatMesh} with $3t$ and $(r + 2)^{2}$ which gives us either a model $\mu$ of $K_{3t}$ in $G$ such that $\mathcal{T}_{\mu} \subseteq \mathcal{T}_{M}$, or an $(r + 2)^{2}$-submesh $M'' \subseteq M$ together with a set $Z \subseteq V(G) \setminus V(M'')$ with $|Z| < 9t^{2}$ such that $M'$ is flat in $G - Z$.
The proof is split in two cases.

\textbf{Case 1:} In case we get a model $\mu$ of $K_{3t}$ in $G$ such that $\mathcal{T}_{\mu} \subseteq \mathcal{T}_{M}$, we call upon \zcref{prop:redClique} and obtain, either a red minor-model $\mu'$ of $K_{t}$ in $(G, R)$ such that $\mathcal{T}_{\mu'} \subseteq \mathcal{T}_{\mu} \subseteq \mathcal{T}_{M}$ and conclude, or a separation $(A, B) \in \mathcal{T}_{\mu} \subseteq \mathcal{T}_{M}$ of order at most $t - 1$ such that $(B \setminus A) \cap R = \emptyset$, and also conclude.

\textbf{Case 2:} This leaves us with the second outcome of \zcref{prop:FlatMesh}.
In this case we may simply call \zcref{lemma:flatMeshToHomegeneousMesh} with $G - Z$ and $M'$ and also conclude.
\end{proof}

\subsection{Red, flat meshes and red grid minors}

Here we present a couple of useful tools that will come in handy later.
The first says that, if we are given a flat mesh of even order whose bricks sandwiched between its middle horizontal paths, all contain a red vertex, then we can in fact find a fully red, flat mesh.
First some definitions.

\medskip
Let $M$ be an $(n \times m)$-mesh, $P_{1}, \ldots, P_{n}$ be a top to bottom ordering of its horizontal paths, and $Q_{1}, \ldots, Q_{m}$ be a left to right ordering of its vertical paths.
Given $(i, j) \in [n - 1] \times [m - 1]$, we say that the brick of $M$ defined as the unique cycle in $P_{i} \cup P_{i + 1} \cup Q_{j} \cup Q_{j + 1}$ is the \emph{$(i, j)$-brick} of $M$.

\begin{lemma}\label{lemma:MiddleRedMeshToFullyRed} For every integer $r \geq 3$ the following holds.
Let $(G, R)$ be an annotated graph and $\rho$ be a rendition of a society $(G, \Omega)$ in the disk with an $(2(r + 1) \times r(r - 1))$-mesh $M \subseteq G$ that is grounded in $\rho$.
Further, assume that for every $j \in [r(r - 1) - 1]$ the subgraph of $G$ induced by the $(r, j)$-brick of $M$ contains a vertex of $R$.
Then, there exists an $r$-mesh $M' \subseteq G$ that is grounded and red in $\rho$ and such that $\mathcal{T}_{M'}$ is a truncation of $\mathcal{T}_{M}$.
\end{lemma}
\begin{proof} Let $P_{1}, \ldots, P_{2(r + 1)}$ be a top to bottom ordering of the horizontal paths of $M$ and $Q_{1}, \ldots, Q_{r(r - 1)}$ be a left to right ordering of the vertical paths of $M.$
Also, for $i \in [2r + 1]$ and $j \in [r(r - 1) - 1]$, let $B_{i}^{j}$ denote the $(i, j)$-brick of $M$ and for $i \in [r - 1]$, let $$\mathcal{B}_{i} = \{ B_{r}^{(i - 1)r + 1}, \ldots, B_{r}^{ir - 1} \}.$$

Let $T_{1}$ be defined as the union of bricks in $\mathcal{B}_{1}$ union $P_{r + 1} \cap \bigcup_{j \in [r]} Q_{j}$, and $T_{r - 1}$ be defined as the union of bricks in $\mathcal{B}_{r - 1}$ union $R \cap \bigcup_{j \in [(r - 2)r + 1, r(r - 1)]} Q_{j}$, minus all edges of the bricks in $\mathcal{B}_{r - 1}$ which are edges of $R$ but not edges of any vertical path, where $R = P_{r + 2}$ if $r$ is odd and $R = P_{r + 1}$ otherwise.

Next, for each $i \in [2, r - 2]$, let $T_{i}$ be defined as the union of bricks in $\mathcal{B}_{i}$ union $(P_{r + 1} \cup P_{r}) \cap \bigcup_{j \in [(i - 1)r + 1, ir]} Q_{j}$, minus all edges of the bricks in $\mathcal{B}_{i}$ which are edges of $R$ but not edges of any vertical path, where $R =  P_{r + 1}$ if $i$ is odd and $R = P_{r + 2}$ otherwise.

Note, that all previously defined $T_{i}$'s are pairwise disjoint.
Moreover, for each $i \in [2, r - 2]$, $T_{i}$ contains precisely $2r$ degree-$1$ vertices, $r$ of them, say $t^{i}_{1}, \ldots, t^{i}_{r}$ ordered from left to right, are vertices of $P_{r}$, and we call them the \emph{top endpoints of $T_{i}$}, and $r$ of them say $b^{i}_{1}, \ldots, b^{i}_{r}$ ordered from left to right, are vertices of $P_{r + 1}$, and we call them the \emph{bottom endpoints of $T_{i}$}.
Additionally, there are precisely $r'$ degree-$1$ vertices of $T_{1}$, say $x_{1}, \ldots, x_{r'}$ ordered from left to right, and precisely $r$ degree-$1$ vertices of $T_{r - 1}$,
say $y_{1}, \ldots, y_{r'}$ ordered from left to right.

Our goal is to find a collection of linkages in $M$, internally disjoint from all $T_{i}$'s that join the $x_{i}$'s in left to right order to the $b^{2}_{j}$'s in the reverse order, then, for each $i \in [2, r - 2]$, join the $t^{i}_{j}$'s in left to right order to the $t^{i + 1}_{j}$'s in the reverse order, and join the $b^{i + 1}_{j}$'s in left to right order to the $b^{i + 2}_{j}$'s in reverse order, and so on, until we reach the $y_{i}$'s.

It follows that $M$ provides the necessary infrastructure for this routing.
See \zcref{fig:MiddleRedMeshToFullyRed} for an illustration on how to find these paths.
In fact, by construction, if we take the union of all $T_{i}$'s with the linkages we find, we obtain an $r$-mesh $M'$ such that moreover, for every brick $B$ of $M$, the subgraph of $G$ induced by $B$ contains the subgraph of $G$ induced by precisely one brick among those contained in a family $\mathcal{B}_{i}$ above, and as a result $M'$ is red in $\rho$.
Also clearly, $\mathcal{T}_{M'}$ is a truncation of $\mathcal{T}_{M}$.
\end{proof}

\begin{figure}[ht]
\centering
\includegraphics{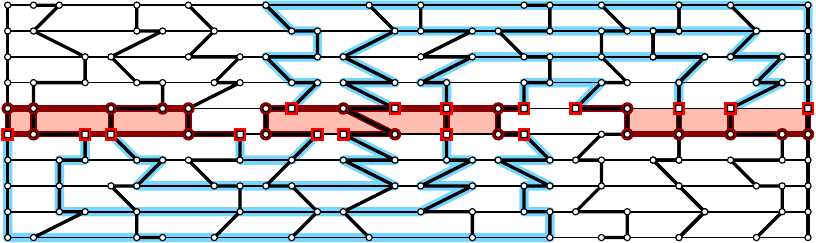}
\caption{\label{fig:MiddleRedMeshToFullyRed} An illustration for the proof of \zcref{lemma:MiddleRedMeshToFullyRed} for $r = 4.$}
\end{figure}

The second lemma extracts a red grid minor from a red, flat mesh.

\begin{lemma}\label{lemma:RedMeshToRedGrid} For every integer $k \geq 2$ the following holds.
Let $(G, R)$ be an annotated graph and $\rho$ be a rendition of a society $(G, \Omega)$ in the disk with a $(3k - 1)$-mesh $M \subseteq G$ that is grounded and red in $\rho.$
Then, $(G, R)$ contains $\mathsf{R}_{k}$ as a red minor.
\end{lemma}
\begin{proof}
Let us denote the vertical paths of $M$ by $P_{1}, \ldots, P_{3k - 1}$ and the horizontal paths of $M$ by $Q_{1}, \ldots, Q_{3k - 1}$.
Let $M'$ denote the $2k$-submesh of $M$ obtained from the union of the $P_{i}$'s and $Q_{i}$'s with indices in the set
$$\{ 3m + 1, 3m + 2 \mid m \in [0, k - 1] \}.$$
Moreover, for every $i, j \in [0, k-1]$, denote by $B_{ij}$ the brick of $M$ that is the unique cycle in $P_{3i + 1} \cup P_{3i + 2} \cup Q_{3j + 1} \cup Q_{3j + 2}$ and let $\mathcal{B}_{M'} = \{ B_{ij} \mid i,j \in [0, k - 1] \}$.

By definition, for every $i, j \in [0, k-1]$, it follows that the subgraph $H_{B_{ij}}$ of $G$ induced by $B_{ij}$ is connected, contains $B_{ij}$, and since $M$ is red in $\rho$, $V(H_{B_{ij}}) \cap R \neq \emptyset$.
Moreover, for distinct $B, B' \in \mathcal{B}_{M}$, $H_{B}$ is disjoint from $H_{B'}$, and we ensure this by skipping over every 3rd horizontal and vertical path in the definition of $M'$.

Now, it is easy to see that $M'$ contains a $\mathsf{R}_{k}$-minor-model where the branch set of each vertex contains precisely one of each $H_{B}$, $B \in \mathcal{B}_{M}.$
\end{proof}

\section{A society classification theorem for red minors}\label{sec:societyclassification}

In this section we prove a variant of the society classification theorem for red minors with polynomial bounds.
This is the most technical and involved contribution of our work.

\paragraph{The society classification theorem with polynomial bounds.}
We start from the society classification theorem for graphs.
The original comes from \cite{KawarabayashiTW2021Quickly}.
Here we utilize the polynomial version due to~\cite{GorskySW2025Polynomial}.

\begin{proposition}[Gorsky, Seweryn, and Wiederrecht \cite{GorskySW2025Polynomial}]\label{thm:societyclassification}
    There exist polynomial functions $\mathsf{apex}^\mathsf{genus}_{\ref{thm:societyclassification}}, \mathsf{loss}_{\ref{thm:societyclassification}} \colon \mathbb{N} \rightarrow \mathbb{N}$, $\mathsf{nest}_{\ref{thm:societyclassification}}, \mathsf{cost}_{\ref{thm:societyclassification}} \colon \mathbb{N}^2 \rightarrow \mathbb{N}$, and $ \mathsf{apex}^\mathsf{fin}_{\ref{thm:societyclassification}}$,  $\mathsf{depth}_{\ref{thm:societyclassification}} \colon \mathbb{N}^3\to\mathbb{N}$, such that for all integers $t, k, p \geq 1$ the following holds.

    Let $s \geq \mathsf{nest}_{\ref{thm:societyclassification}}(t, k)$ be an integer.
    Let $(G,\Omega)$ be a society and $\rho$ be a cylindrical rendition of $(G, \Omega)$ in a disk $\Delta$ with a cozy nest $\mathcal{C} = \{ C_1, \ldots , C_s \}$ around the vortex $c_{0}$ and a radial linkage $\mathcal{R}$ for $\mathcal{C}$ of order $s$ that is orthogonal to $\mathcal{C}$.
    Further, let $M$ be the $s$-cylindrical mesh contained in $\bigcup (\mathcal{C} \cup \mathcal{R})$, and $(G', \Omega')$ be the $C_{s - \mathsf{cost}_{\ref{thm:societyclassification}}(t,k)}$-society in $\rho.$
    
    Then $G'$ contains a set $A \subseteq V(G')$ such that one of the following exists:
    \begin{enumerate}
        \item A $K_t$-minor-model in $G$ controlled by $M$.

        \item A flat, isolated crosscap transaction $\mathcal{P}$ of order $p$ in $(G'-A,\Omega')$, with $|A| \leq \mathsf{apex}^\mathsf{genus}_{\ref{thm:societyclassification}}(t)$, and a nest $\mathcal{C}'$ in $\rho$ of order $s - (\mathsf{loss}_{\ref{thm:societyclassification}}(t)+\mathsf{cost}_{\ref{thm:societyclassification}}(t,k))$ around $c_0$ to which $\mathcal{P}$ is orthogonal.

        \item A flat, isolated handle transaction $\mathcal{P}$ of order $p$ in $(G'-A,\Omega')$, with $|A| \leq \mathsf{apex}^\mathsf{genus}_{\ref{thm:societyclassification}}(t)$, and a nest $\mathcal{C}'$ in $\rho$ of order $s - (\mathsf{loss}_{\ref{thm:societyclassification}}(t)+\mathsf{cost}_{\ref{thm:societyclassification}}(t,k))$ around $c_0$ to which $\mathcal{P}$ is orthogonal.

        \item A rendition $\rho'$ of $(G - A, \Omega)$ in $\Delta$ with breadth $b \in [\nicefrac{1}{2}(t-3)(t-4)-1]$ and depth at most $\mathsf{depth}_{\ref{thm:societyclassification}}(t,k,p)$, $|A| \leq \mathsf{apex}^\mathsf{fin}_{\ref{thm:societyclassification}}(t,k,p)$, and an extended $k$-surface-wall $D$ with signature $(0,0,b)$, such that $D$ is grounded in $\rho'$, the base cycles of $D$ are the cycles $C_{s-\mathsf{cost}_{\ref{thm:societyclassification}}(t,k)-1-k},\dots,C_{s-\mathsf{cost}_{\ref{thm:societyclassification}}(t,k)-1}$, and there exists a bijection between the vortices $v$ of $\rho'$ and the vortex segments $S_v$ of $D$, where $v$ is the unique vortex contained in the disk $\Delta_{S_v}$ defined by the trace of the inner cycle of the nest of $S_v$, and $\Delta_{S_v}$ is chosen to avoid the trace of the simple cycle of $D$.
    \end{enumerate}
    
    In particular, the set $A$, the $K_t$-minor-model, the transaction $\mathcal{P}$, the rendition $\rho'$, and the extended surface-wall $D$ can each be found in time $\mathbf{poly}( t + s + p + k ) \cdot |E(G)||V(G)|^2$.
\end{proposition}

In \cite{GorskySW2025Polynomial}, the following estimates for the functions in \zcref{thm:societyclassification} are given.
\begin{align*}
\mathsf{nest}_{\ref{thm:societyclassification}}(t,k) &\in \mathbf{O}((t+k)^9),\\
\mathsf{apex}^\mathsf{genus}_{\ref{thm:societyclassification}}(t) &\in \mathbf{O}(t^8), \ \ \mathsf{loss}_{\ref{thm:societyclassification}}(t) \in \mathbf{O}(t^3), \ \ \mathsf{cost}_{\ref{thm:societyclassification}}(t,k) \in \mathbf{O}(t+k),\\
\mathsf{apex}^\mathsf{fin}_{\ref{thm:societyclassification}}(t,k,p) &\in \mathbf{O}((t+k+p)^{37}), \ \ \text{and} \ \ \mathsf{depth}_{\ref{thm:societyclassification}}(t,k,p) \in \mathbf{O}((t+k+p)^{51}).
\end{align*}

\paragraph{The annotated variant.}
The goal of this section is to prove the following version of the society classification theorem for annotated graphs.

\begin{restatable}{theorem}{blanksocietyclassification}\label{thm:blanksocietyclassification}
There exist polynomial functions $\mathsf{apex}^\mathsf{genus}_{\ref{thm:blanksocietyclassification}}, \mathsf{loss}_{\ref{thm:blanksocietyclassification}} \colon \mathbb{N} \rightarrow \mathbb{N}$, $\mathsf{nest}_{\ref{thm:blanksocietyclassification}}, \mathsf{cost}_{\ref{thm:blanksocietyclassification}} \colon \mathbb{N}^3 \rightarrow \mathbb{N}$, and $ \mathsf{apex}^\mathsf{fin}_{\ref{thm:blanksocietyclassification}}$, $\mathsf{depth}_{\ref{thm:blanksocietyclassification}} \colon \mathbb{N}^4 \to \mathbb{N}$, such that for all integers $r, k \geq 3$ and $t, p \geq 1$ the following holds.

Let $s \geq \mathsf{nest}_{\ref{thm:blanksocietyclassification}}(r, t, k).$
Let $(G, R)$ be an annotated graph and $\rho$ be a blank cylindrical rendition of a society $(G, \Omega)$ in a disk $\Delta$ with a cozy nest $\mathcal{C} = \{ C_{1}, \ldots, C_{s} \}$ around the vortex $c_{0}$ and a radial linkage $\mathcal{R}$ for $\mathcal{C}$ of order $s$ that is orthogonal to $\mathcal{C}$.
Further, let $M$ be the $s$-cylindrical mesh contained in $\bigcup (\mathcal{C} \cup \mathcal{R})$, and $(G', \Omega')$ be the $C_{s - \mathsf{cost}_{\ref{thm:blanksocietyclassification}}}(r, t, k)$-society in $\rho$.

Then $G'$ contains a set $A \subseteq V(G')$ such the one of the following exists:
\begin{enumerate}

\item A separation $(X, Y) \in \mathcal{T}_{M}$ of order at most $t - 1$ such that $(Y \setminus X) \cap R = \emptyset$.

\item A red $K_{t}$-minor-model in $(G, R)$ controlled by $M$.

\item A homogeneous, isolated crosscap transaction $\mathcal{P}$ of order $p$ in $(G' - A, \Omega')$, with $|A| \leq \mathsf{apex}^\mathsf{genus}_{\ref{thm:blanksocietyclassification}}(t)$, and a nest $\mathcal{C}'$ in $\rho$ of order $s - (\mathsf{loss}_{\ref{thm:blanksocietyclassification}}(t) + \mathsf{cost}_{\ref{thm:blanksocietyclassification}}(r, t, k))$ to which $\mathcal{P}$ is orthogonal.

\item A homogeneous, isolated handle transaction $\mathcal{P}$ of order $p$ in $(G' - A, \Omega')$, with $|A| \leq \mathsf{apex}^\mathsf{genus}_{\ref{thm:blanksocietyclassification}}(t)$, and a nest $\mathcal{C}'$ in $\rho$ of order $s - (\mathsf{loss}_{\ref{thm:blanksocietyclassification}}(t)+\mathsf{cost}_{\ref{thm:blanksocietyclassification}}(r, t, k))$ to which $\mathcal{P}$ is orthogonal.

\item A blank rendition $\rho'$ of $(G - A, \Omega)$ in $\Delta$ such either
\begin{itemize}
\item there exists an $r$-mesh $M' \subseteq G - A$ that is grounded and red in $\rho'$ and such that $\mathcal{T}_{M'}$ is a truncation of $\mathcal{T}_{M}$, or
\item $\rho'$ has breadth $b$ at most $\nicefrac{3}{2}(t - 1)(3t - 4) + r(r - 1) - 3$ and depth at most $\mathsf{depth}_{\ref{thm:blanksocietyclassification}}(r, t, k, p)$, $|A| \leq \mathsf{apex}^{\mathsf{fin}}_{\ref{thm:blanksocietyclassification}}(r, t, k, p)$, and an extended $k$-surface-wall $D$ with signature $(0, 0, b)$, such that $D$ is grounded in $\rho'$, the base cycles of $D$ are the cycles $C_{s - \mathsf{cost}_{\ref{thm:blanksocietyclassification}}(r, t, k)-1-k}, \ldots, C_{s - \mathsf{cost}_{\ref{thm:blanksocietyclassification}}(r, t, k)-1}$, and there exists a bijection between the vertices $v$ of $\rho'$ and the vortex segments $S_{v}$ of $D$, where $v$ is the unique vortex contained in $\Delta_{S_{v}}$ defined by the trace of the inner cycle of the nest of $S_{v}$, and $\Delta_{S_{v}}$ is chosen to avoid the trace of the simple cycle of $D$.
\end{itemize}

\end{enumerate}
In particular, each of these outcomes can be found in time $\mathbf{poly}(r + t + s + p + k) \cdot |E(G)|^{3}$.
\end{restatable}

Towards this, in the upcoming subsections we develop a series of tools that carefully combine ideas from the proof of \zcref{thm:societyclassification} along with new techniques to deal with the red vertices, culminating in the proof of \zcref{thm:blanksocietyclassification} in \zcref{sec:ProofOfSociety}.

At this stage, let us provide estimates for the function involved in \zcref{thm:blanksocietyclassification}.
\begin{align*}
\mathsf{nest}_{\ref{thm:blanksocietyclassification}}(r, t, k) &\in \mathbf{O}((r + t + k)^{144}),\\
\mathsf{apex}^{\mathsf{genus}}_{\ref{thm:blanksocietyclassification}}(t) &\in \mathbf{O}(t^{8}), \ \ \mathsf{loss}_{\ref{thm:blanksocietyclassification}}(t) \in \mathbf{O}(t^{3}), \ \ \mathsf{cost}_{\ref{thm:blanksocietyclassification}}(r, t, k) \in \mathbf{O}((r + t + k)^{16}),\\
\mathsf{apex}^{\mathsf{fin}}_{\ref{thm:blanksocietyclassification}}(r, t, k, p) &\in \mathbf{O}((r + t + k + p)^{592}), \ \ \text{and} \ \ \mathsf{depth}_{\ref{thm:blanksocietyclassification}}(r, t, k, p) \in \mathbf{O}((r + t + k + p)^{820}).
\end{align*}

\subsection{Dealing with flat crosscap or handle transactions}

To prove \zcref{thm:blanksocietyclassification} we eventually have to deal with outcomes ii) and iii) if \zcref{thm:societyclassification}.
Our goal here is to first show that, similar to \zcref{sec:flatwall}, we may homogenize flat crosscap or handle transactions with respect to the red set.
Then, assuming the homogenized transaction is not blank, we show that using the mesh infrastructure provided by the union of the nest and the transaction, we may always find a large red grid minor in the graph.
This allows us to always conclude with a blank transaction.

\paragraph{Homogeneous transactions.}
Let $(G, R)$ be an annotated graph and $(G, \Omega)$ be a society.
Also, let $\mathcal{P} = \{ P_{1}, \ldots, P_{n} \}$ be a naturally indexed, flat transaction in $(G, \Omega)$ of order $n \geq 2$ and $(G_{1}, \Omega_{1})$ be the $\mathcal{P}$-strip society in $(G, \Omega)$.

Let $X, Y$ be the end segments of $\mathcal{P}$ in $(G, \Omega)$ such that $\Omega_{1}$ is the concatenation of $X$ and $Y$.
Moreover, for all $i \in [n - 1]$, let $X_{i} \subseteq X$ and $Y_{i} \subseteq Y$ be the end segments of $\{ P_{i}, P_{i+1} \}$ in $(G, \Omega)$.

Next, let $H \subseteq G$ be obtained from the union of elements of $\mathcal{P}$ by adding the elements of $V(\Omega)$ as isolated vertices.
Given $i \in [n-1]$, consider any $H$-bridge of $G$ with at least one attachment in $$\big(V(X_{i}) \cup V(Y_{i}) \cup V(P_{i} \cup P_{i + 1})\big) \setminus V(P_{1} \cup P_{n}),$$ and for each such $H$-bridge $B$ let $B'$ denote the graph obtained from $B$ by deleting all attachments that do not belong to $V(\bigcup \mathcal{P}) \cup V(X) \cup V(Y)$.
We define the $i$-th \emph{strip} of $\mathcal{P}$ to be the graph $H_{i} \subseteq G_{1}$ obtained from the union of $P_{i} \cup P_{i+1}$ and all graphs $B'$ as above, after adding all vertices of $V(X_{i})$ and $V(Y_{i})$ as isolated vertices.
We also say that $H_{1}$ and $H_{n-1}$ are the \emph{boundary strips} of $\mathcal{P}$.

Note that since $\mathcal{P}$ is flat, and therefore $(G_{1}, \Omega_{1})$ admits a vortex-free rendition in the disk, it follows that for every $H$-bridge $B$ of $G$ with at least one attachment in $V(\bigcup \{ P_{2}, \ldots, P_{n - 1} \}) \cup V(X) \cup V(Y)$, $B'$ defined as above is a subgraph of possibly both the $i$-th and $(i+1)$-th strip of $\mathcal{P}$ for some $i \in [n-1]$, and is otherwise disjoint from any other strip.
Also note that $G_{1} = \bigcup_{i \in [n -1]} H_{i}$.

A flat transaction is called \emph{red} if all its non-boundary strips contain a vertex of $R$ or \emph{blank} if no strip does.
In either case we call it \emph{homogeneous}.

\begin{lemma}\label{lemma:homoTransactions} For all integers $p, q \geq 2$ the following holds.
Let $(G, R)$ be an annotated graph, $(G, \Omega)$ be a society, and $\rho$ be a rendition of $(G, \Omega)$ in the disk.
Then for every flat transaction $\mathcal{P}$ in $(G, \Omega)$ of order $qp$ there exists either a blank transaction $\mathcal{Q}_{1} \subseteq \mathcal{P}$ of order $p$ or a red transaction $\mathcal{Q}_{2} \subseteq \mathcal{P}$ of order $q$.

Moreover, there exists an algorithm that finds $\mathcal{Q}$ in time $\mathbf{O}(qp \cdot |E(G)|)$.
\end{lemma}
\begin{proof} Let $P_{1}, \ldots, P_{qp}$ be an ordering of $\mathcal{P}$ so that they are indexed naturally.
For all integers $i < j \in [qp - 1]$, let $H_{i}^{j} = \bigcup_{l \in [i, j - 1]} H_{l}$, where $H_{i}$ denotes the $i$-th strip of $\mathcal{P}$.

For $i \in [q]$, let $\mathcal{P}_{i} = \{ P_{(i-1)\cdot p + 1}, \ldots, P_{i \cdot p}\}$ and let $\mathcal{P}'$ be defined by collecting the first path from every $\mathcal{P}_{i}$, $i \in [p]$.
Now we are in one of two cases.
Either there exists $i \in [q]$, such that $H_{(i - 1) \cdot p + 1}^{i \cdot p - 1}$ contains no vertex of $R$, or all do.
In the first case, we argue that $\mathcal{P}_{i}$ is the desired blank transaction, while in the second case we argue that $\mathcal{P}'$ is the desired red transaction.

It follows that any subtransaction $Q \subseteq \mathcal{P}$ is flat.
Let $(G_{1}, \Omega_{1})$ be the $\mathcal{P}$-strip society in $(G, \Omega)$ and $(G_{2}, \Omega_{2})$ be the $\mathcal{Q}$-strip society in $(G, \Omega)$.
To complete our proof it suffices to show that $G_{2} \subseteq G_{1}$ and that for any two consecutive paths $P_{i}, P_{j}$ in $\mathcal{Q}$ such that both $P_{i}$ and $P_{j}$ are non-boundary paths of $\mathcal{Q}$, $H'_{i} = H_{i}^{j}$, where $H'_{i}$ is the strip of $\mathcal{Q}$ that corresponds to $P_{i}$.

The fact that $G_{2} \subseteq G_{1}$ follows by definition of strip societies and this already concludes the proof if we are in the first case above.
For the second property, let $X, Y$ be the end segments of $\mathcal{P}$ in $(G, \Omega)$ so that $\Omega_{1}$ is the concatenation of $X$.
Moreover, for all $i \in [n-1]$, let $X_{i} \subseteq X$ and $Y_{i} \subseteq Y$ be the end segments of $\{ P_{i}, P_{i+1} \}$ in $(G, \Omega)$.

Let $H$ denote the subgraph of $G$ obtained from the union of elements of $\mathcal{P}$, by adding the elements of $V(\Omega)$ as isolated vertices.
First note that strips, by definition, cannot contain isolated vertices which are not vertices of $V(\Omega)$.
Hence consider an edge $e \in E(H'_{i})$.
If $e$ is an edge of a path in $\{ P_{i}, \ldots, P_{j} \}$, then $e \in E(H_{i}^{j})$.
Moreover, by flatness of $\mathcal{Q}$, $e$ cannot be an edge of a path in $\mathcal{P} \setminus \{ P_{i}, \ldots, P_{j} \}$ either.
This implies, combined with the flatness of $\mathcal{P}$, that there is $l \in [i, j-1]$ and $B'$, where $B$ is an $H$-bridge of $G$ with at least one attachment in $(V(X_{l}) \cup V(Y_{l}) \cup V(P_{l} \cup P_{l+1})) \setminus V(P_{1} \cup P_{n})$, and $B'$ is obtained from $B$ after deleting attachments that do not belong to $V(\bigcup \mathcal{P}) \cup V(X) \cup V(Y)$, such that $e \in E(B')$.
However by definition of strips for $\mathcal{P}$, $e \in E(H_{i})$, and therefore $e \in E(H_{i}^{j})$.
In a similar manner we may prove that $H_{i}^{j} \subseteq H'_{i}$ as well.
This concludes the proof for the latter case.
\end{proof}

Note that in case of a flat handle transaction $\mathcal{P} = \mathcal{R} \cup \mathcal{Q}$ in $G$, where $\mathcal{R}$ and $\mathcal{Q}$ are its constituent transactions, we say that $\mathcal{P}$ is \emph{homogeneous} if both $\mathcal{P}$ and $\mathcal{R}$ are homogeneous in $G$.

\paragraph{Red, flat mesh or blank transaction.}

We now guarantee that, given a cylindrical rendition of a society in the disk with a nest around a vortex and a large flat, red transaction that is orthogonal to the nest, we may always use the mesh infrastructure provided by the union of the nest and the transaction to find a mesh that is fully red.
This will allow us to clean up the crosscap or handle transactions implied by~\zcref{thm:blanksocietyclassification} for the proof of our local structure theorem, \zcref{thm:strongest_localstructure}.

\begin{lemma}\label{lemma:RedTransactionGivesRedMesh} For every integer $r \geq 2$ the following holds.
Let $(G, R)$ be an annotated graph and $\rho$ be a blank cylindrical rendition of a society $(G, \Omega)$ in a disk $\Delta$ around a vortex $c_{0}$ with a nest $\mathcal{C}$ of order $r + 2.$
If $\mathcal{P}$ is a red, exposed transaction in $(G, \Omega)$ of order $3r(r - 1) - 2$ that is orthogonal to $\mathcal{C}$ and $\rho'$ is a vortex-free rendition of the $\mathcal{P}$-strip society in $(G, \Omega),$ then there exists an $r$-mesh $M \subseteq G$ that is grounded and red in $\rho'$ and such that $\mathcal{T}_{M}$ is a truncation of the tangle of a mesh whose horizontal paths are subpaths of distinct cycles from $\mathcal{C}.$

Moreover, there exists an algorithm that finds $M$ in time $\mathbf{poly}(r) \cdot |E(G)|$.
\end{lemma}
\begin{proof} Let $r' = r(r - 1).$
Let $\mathcal{P} = \{ P_{1}, \ldots, P_{3r' - 2} \}$ be an ordering of $\mathcal{P}$ that gives a natural indexing for $\mathcal{P},$ $\mathcal{C} = \{ C_{1}, \ldots, C_{r + 2} \},$ and $\rho'$ be a vortex-free rendition in the disk $\Delta$ of the $\mathcal{P}$-strip society $(G', \Omega')$ in $(G, \Omega).$
Also, let $\mathcal{P}' = \{ P'_{1}, \ldots, P'_{3r' - 2} \}$ where each $P'_{i}$ is defined as the unique maximal subpath of $P_{i}$ with both endpoints in $V(C_{1})$ and for $i \in [3r' - 1],$ let $X_{i}$ and $Y_{i}$ denote the two disjoint $V(P_{i})$-$V(P_{i + 1})$-subpaths of $C_{1}$ that are disjoint from all other paths in $\mathcal{P}.$

Let $\mathcal{Q} = \{ P_{3m + 1} \mid m \in \{ 0, \ldots, r' - 1 \} \}.$
Clearly $|\mathcal{Q}| = r'.$
Observe that since $\mathcal{Q}$ is orthogonal to $\mathcal{C}$ and $|\mathcal{C}| = r + 2,$ we can find a $(2(r + 1) \times r')$-mesh $M \subseteq \bigcup (\mathcal{C} \setminus \{ C_{1} \}) \cup \bigcup \mathcal{Q}$ such that the unique subpath of every path $Q \in \mathcal{Q}$ contained in the intersection of $Q$ with the crop of $G$ by the $C_{r + 2}$-disk in $\rho,$ that intersects all cycles in $\mathcal{C},$ is a vertical path of $M.$
Clearly, $M \subseteq G'$ and grounded in $\rho'.$

We now claim that $M$ satisfies the conditions to apply \zcref{lemma:MiddleRedMeshToFullyRed} that will give us the desired mesh.
What we have to show is that for every $m \in \{ 0, \ldots, r' - 2 \},$ the subgraph $H_{B_{m}}$ of $G_{1}$ induced by the brick $B_{m}$ of $M$ in $\rho',$ where $B_{m}$ is the brick contained in the union of $P_{3m + 1} \cup P_{3(m + 1) + 1} \cup C_{2},$ contains a vertex of $R.$

For this, it suffices to show the following.
Let $m' = 3m + 2$ and let $H_{m'}$ be the $m'$-th strip of $\mathcal{P}.$
Since $\mathcal{P}$ is red, $V(H_{m'}) \cap R \neq \emptyset.$
Further let $U_{m'} \subseteq V(H_{n'})$ be defined as the vertex set of the union $P'_{m'} \cup P'_{m' + 1} \cup X_{m'} \cup Y_{m'}.$
The first thing to observe is that the existence of $\rho$ and the fact that it is blank, implies that every $(V(H_{m'}) \cap R)$-$V(M)$-path in $G$ that is disjoint from $P'_{m'} \cup P'_{m' + 1}$ must intersect $C_{1}$ precisely at a vertex of $X_{i}$ or $Y_{i}.$
This implies that any connected component $C_{m'}$ in the graph $H_{m'} - U_{m'}$ which contains a vertex of $R$ is disjoint from the connected component of $H_{m'} - U_{m'}$ which contains $B_{m},$ which in turn implies that $C_{m'} \subseteq H_{B_{m}}.$
Indeed, if this were not the case, then the existence of $\rho'$ would imply that there is a vertex $v \in V(C_{m'}) \setminus V(H_{B_{m}})$ which $B_{m}$ separates from $U_{m'}$ which is clearly absurd.

Moreover, by construction, it easily follows that $\mathcal{T}_{M}$ is the truncation of $\mathcal{T}_{M'}$, where $M'$ is any $r$-submesh of $M$ whose horizontal paths are subpaths of distinct cycles from $\mathcal{C}.$
We may now call \zcref{lemma:MiddleRedMeshToFullyRed} to obtain an $r$-mesh $M'' \subseteq G$ that is grounded and red in $\rho'$ such that $\mathcal{T}_{M''} \subseteq \mathcal{T}_{M} \subseteq \mathcal{T}_{M'}$ as desired.
\end{proof}
 
\subsection{Towards a blank rendition}

To prove \zcref{thm:blanksocietyclassification} we eventually have to deal with outcome iv) of \zcref{thm:societyclassification}.
The task here is to show that, starting from a bounded breadth and depth rendition of a society in a disk, one can either find a large flat, red mesh in the graph, or refine this rendition into a bounded breadth and depth rendition which is blank, i.\@e.\@, red vertices are featured only within the vortices of this rendition.
In addition this has to be performed modestly in order to preserve polynomial functions.
This is the most technically challenging part of our proof.

\subsubsection{Nest trees}

In what follows we define an intermediate structure that will aid us in keeping track of the necessary structure to be subsequently refined in multiple steps over a series of lemmas that follow.

\medskip
Let $(G, \Omega)$ be a society with a rendition $\rho$ in the disk.
Let $\mathcal{C}$ be a nest and $\mathcal{R}$ be a radial linkage for $\mathcal{C}$.
Further, let $\mathfrak{C}$ be a non-empty set of pairwise disjoint nests in $\rho$ containing $\mathcal{C}$ and $\mathfrak{R}$ be a (possibly empty) set of linkages of $G$.

Further let $(T, r)$ be a rooted subcubic tree and $\phi_{1} \colon V(T) \to \mathfrak{C}$, $\phi_{2} \colon E(T) \to \mathfrak{R}$ be bijections satisfying the following properties.

Let $s_{0} \geq 1$ be an integer.
First assume that for every non-leaf node $t \in V(T)$, $|\phi_{1}(t)| = s_{0}$, and if $\phi_{1}(t) = \{ C^{t}_{1}, \ldots, C^{t}_{s_{0}} \}$, let $\Delta^{\mathsf{out}}_{t}$ denote the $C^{t}_{s_{0}}$-disk in $\rho$ and $\Delta^{\mathsf{in}}_{t}$ denote the $C^{t}_{1}$-disk in $\rho$.

Moreover, assume that for every leaf node $l \in V(T)$, $|\phi_{1}(l)| \geq s_{0} + 2$, and if $\phi_{1}(l) = \{ C^{l}_{1}, \ldots, C^{l}_{s} \}$, let $\Delta^{\mathsf{out}}_{l}$ denote the $C^{l}_{s}$-disk in $\rho$, $\Delta^{\mathsf{soc}}_{t}$ denote the $C^{l}_{s - s_{0}}$, and $\Delta^{\mathsf{in}}_{t}$ denote the $C^{l}_{1}$-disk in $\rho$.

Further assume that,
\begin{description}
    \item[T1] $\phi_{1}(r) = \mathcal{C}$ and $\mathcal{R}_{r}$ is the $\Delta^{\mathsf{out}}_{r}$-truncation in $\rho$ of $\mathcal{R}$.
    Then $\mathcal{R}_{r}$ is an orthogonal radial linkage in the restriction of $\rho$ by $\Delta^{\mathsf{out}}_{r}$ for $\mathcal{C}$.
    \item[T2] for every edge $tt' \in E(T)$ with $t$ being the parent of $t'$ in $(T, r)$, $\Delta^{\mathsf{out}}_{t'} \subseteq \Delta^{\mathsf{in}}_{t}$,
    \item[T3] if $t, t' \in V(T)$ are \emph{siblings}, i.e., they have the same parent in $(T, r)$, then $\Delta^{\mathsf{out}}_{t} \cap \Delta^{\mathsf{out}}_{t'} = \emptyset$.
    \item[T4] Moreover, for every edge $tt' \in E(T)$ with $t$ being the parent of $t'$ in $(T, r)$, $\phi_{2}(tt')$ is a radial linkage in the restriction of $\rho$ by $\Delta^{\mathsf{out}}_{t}$ for $\phi_{1}(t) \cup \phi_{1}(t')$ that is orthogonal to $\phi_{1}(t) \cup \phi_{1}(t')$.
    Also, if $\mathcal{R}_{t'}$ is the $\Delta^{\mathsf{out}}_{t'}$-truncation in $\rho$ of $\phi_{2}(tt')$, then $\mathcal{R}_{t'}$ is a radial linkage in the restriction of $\rho$ by $\Delta^{\mathsf{out}}_{t'}$ for $\phi_{1}(t')$ that is orthogonal to $\phi_{1}(t')$.
    \item[T5] Further, if $t_{1}, t_{2} \in V(T)$ are siblings with parent $t \in V(T)$ in $(T, r)$, then every path in $\phi_{2}(tt_{1})$ is disjoint from every path in $\phi_{2}(tt_{2})$.
    In particular, $\phi_{2}(tt_{1})$ is disjoint from the crop of $G$ by $\Delta^{\mathsf{out}}_{t_{2}}$ and $\phi_{2}(tt_{2})$ is disjoint from the crop of $G$ by $\Delta^{\mathsf{out}}_{t_{1}}$.
    \item[T6] Additionally, every vortex cell $c_{0}$ of $\rho$ is contained in $\Delta^{\mathsf{in}}_{l}$, for some leaf node $l \in V(T)$.
\end{description}

We call the tuple $\mathfrak{T} = (\mathcal{C}, \mathcal{R}, T, r, \phi_{1}, \phi_{2})$ a \emph{nest tree (in $\rho$)} of $(G, \Omega)$.

For an integer $t \geq 1$, we say that $\mathfrak{T}$ has \emph{linkage order} $t$ if $|\mathcal{R}| = t$ and for every edge $e \in V(T)$, $|\phi_{2}(e)| = t$.
Moreover, for an integer $s \geq 1$, assume that for every leaf node $l \in V(T)$, $|\phi_{1}(l)| = s + s_{0} + 1$.
We refer to $s$ as the \emph{cycle order} of $\mathfrak{T}$, while to $s_{0}$ as the \emph{reserve}.
We also refer to the number of leaf nodes in $T$ as the \emph{number of leaves} of $\mathfrak{T}$.

We refer to nodes of $T$ as \emph{nodes} of $\mathfrak{T}$.
We call $\mathcal{C}$ the \emph{root nest} of $\mathfrak{T}$ and $\mathcal{R}$ the \emph{root radial linkage} of $\mathfrak{T}$.
For every node $t \in V(T)$ we call $\mathcal{R}_{t}$ the \emph{truncated radial linkage for $t$ (in $\mathfrak{T}$)} and the corresponding non-truncated version, the \emph{radial linkage for $t$ (in $\mathfrak{T}$).}
For every leaf node $l \in V(T)$ we call the set $\{ C^{t}_{s+2}, \ldots C^{t}_{s + s_{0} + 1} \}$ the \emph{nest in reserve for $l$ (in $\mathfrak{T}$).}
Further, we call $C^{l}_{s + 1}$ the \emph{boundary cycle for $l$ (in $\mathfrak{T}$)} and $\phi_{1}(l)$ minus both the nest in reserve and the boundary cycle for $l$, the \emph{leaf nest for $l$ (in $\mathfrak{T}$).}

We also refer to the $\Delta^{\mathsf{soc}}_{l}$-society $(G_{l}, \Omega_{l})$, for some leaf node $l \in V(T)$, as a \emph{leaf society} of $\mathfrak{T}$, to the leaf nest for $l$ as the \emph{nest of $(G_{l}, \Omega_{l})$,} and to the $\Delta^{\mathsf{soc}}_{l}$-truncation in $\rho$ of $\mathcal{R}_{l}$ as the \emph{radial linkage of $(G_{l}, \Omega_{l})$.}

Given a set $R \subseteq V(G)$, we call a nest tree $\mathfrak{T}$ \emph{$R$-consistent} if every vertex of $R$ is drawn in the interior of $\Delta^{\mathsf{in}}_{l}$, for some leaf node $l \in V(T)$, and for every leaf node $l \in V(T)$, if $\Delta^{\mathsf{in}}_{l}$ contains no vortex cell of $\rho$, then there is at least one vertex of $R$ drawn in the interior of $\Delta^{\mathsf{in}}_{l}$.

Finally, given a second nest tree $\mathfrak{T}' = (\mathcal{C}', \mathcal{R}', T', r', \phi'_{1}, \phi'_{2})$ in $(G, \Omega)$, we say that $\mathfrak{T}$ \emph{respects} $\mathfrak{T}$ if the endpoints of $\mathcal{R}'$ on $V(\Omega)$ are a subset of the endpoints of $\mathcal{R}$ on $V(\Omega)$ and $\mathcal{C'} \subseteq \mathcal{C}$.

\subsubsection{Developing the tools}

\paragraph{Flat transactions in a fixed rendition.}
Outcome iv) of \zcref{thm:societyclassification} allows to assume that we are always dealing with societies for which we already know that there exists a rendition in a disk with bounded breadth and depth.
As it will become apparent later, in the absence of a large flat, red mesh, outcome iv) of \zcref{thm:societyclassification} combined with \zcref{thm:FlatRedMesh}, will allows us to start from a society with a rendition in the disk of bounded breadth and depth and an initial nest tree with a single leaf.

Therefore, we may always assume to be working with such renditions which immediately allow us to deduce that any large enough transaction contains a still large subtransaction which is planar and resides in a vortex-free part of our rendition.
\medskip

Let $(G, \Omega)$ be a society and $\rho$ be a rendition of $(G, \Omega)$ in the disk.
Moreover, let $\mathcal{P} = \{ P_{1}, \ldots, P_{n} \}$ be a naturally indexed planar transaction of order $n \geq 2$ in $(G, \Omega)$.
Also, assuming that both $P_{1}$ and $P_{n}$ are grounded in $\rho$, let $T_{1}$ and $T_{n}$ be the trace of $P_{1}$ and $P_{n}$ in $\rho$ respectively. 
Then, there exists a unique open disk in $\Delta \setminus (T_{1} \cup T_{n})$ whose closure $\Delta_{\mathcal{P}}$ is bounded by both $T_{1}$ and $T_{2}$.
We call $\Delta_{\mathcal{P}}$ the \emph{container (in $\rho$)} of $\mathcal{P}$.

We say that $\mathcal{P}$ is \emph{$\rho$-flat} if the restriction $\rho_{\mathcal{P}}$ of $\rho$ by $\Delta_{\mathcal{P}}$ is vortex-free.

Note that $\rho$-flatness should not be confused with the flat transactions that were previously used.
This definition of $\rho$-flatness is used only in relation to a fixed rendition $\rho$.

\begin{lemma}\label{lemma:flatTransaction} For all integers $b \geq 0$ and $p, d \geq 2$ the following holds.
If $\rho$ is a rendition of $(G, \Omega)$ in the disk with breadth $b$ and depth $d$, then every transaction in $(G, \Omega)$ of order $(b + 1)(2bd + p)$ contains a transaction of order $p$ that is $\rho$-flat.
\end{lemma}

\begin{proof}
Let $\mathcal{P}$ be a transaction of order $(b+1)(2bd + p)$ in $(G, \Omega)$, $I, J$ be the end segments of $\mathcal{P}$ in $(G, \Omega)$, and consider an ordering $P_1, \ldots, P_{\ell}$ of the paths in $\mathcal{P}$, ordered with respect to the occurrence of the endpoint $u_i$ of $P_i$ in $I$.

Suppose there exists a transaction $\mathcal{P}' \subseteq \mathcal{P}$ such that there exists a path $P$ in $\mathcal{P}'$ which crosses all other paths in $\mathcal{P}'$.
If $|\mathcal{P}'| > bd$, then there must exist a transaction $\mathcal{P}'' \subseteq \mathcal{P}'$ of order at least $d+1$ and a vortex $c$ of $\rho$ such that $\mathcal{P}''$ induces a transaction of order at least $d+1$ in the vortex society of $c$.
Since $c$ is of depth at most $d$ this is impossible and thus $|\mathcal{P}'| \leq bd$.

We partition $I$ into $b+1$ consecutive segments $I_1, I_2, \ldots, I_{b+1}$ each containing at least $2bd + p$ endpoints of members of $\mathcal{P}$.
For each $i \in [b+1]$ there exists a family $\mathcal{P}_i$ containing $p$ consecutive paths such that $I_i$ has at least $bd$ other endpoints on either side.
Notice that, by the observation above, it is not possible that a path from $\mathcal{P}_i$ crosses a path from $\mathcal{P}$ which does not have an endpoint in $I_i$.
In particular, this means that for any choice of $i \neq j \in [b+1]$ there does not exist a vortex of $\rho$ which contains an edge of a path in $\mathcal{P}_i$ and an edge of a path of $\mathcal{P}_j$.
Therefore, there must exist $i \in [b+1]$ such that no path in $\mathcal{P}_i$ has an edge or a vertex that belongs to a vortex of $\rho$.
It follows that $\mathcal{P}_i$ is $\rho$-flat.
\end{proof}

\paragraph{Exposed transactions in non-cylindrical renditions.}
Given a rendition of a society with a nest, in order to make any progress with our constructions, we require a large part of a large transaction in the society to traverse the interior of the inner cycle of the nest.
\medskip

Let $(G, \Omega)$ be a society and $\rho$ be a rendition $\rho$ of $(G, \Omega)$ in the disk with a nest $\mathcal{C} = \{ C_{1}, \ldots \}$.
A transaction $\mathcal{P}$ in $(G, \Omega)$ is said to be \emph{$\rho$-exposed} if the crop of $G$ by the $C_{1}$-disk in $\rho$ contains at least one edge of every path in $\mathcal{P}$ which is not an edge of $C_{1}$.

Please note that the definition of exposed we use here differs from the one we previously used for cylindrical renditions.
The reason behind this is that for the proofs that follow we do not require the presence of an actual vortex cell for a transaction to be exposed.
For technical reasons, it simply suffices that the transaction crosses through the disk of the inner cycle of the nest.

\paragraph{Tighter nest trees.} As we have discussed, our target is to iteratively refine the initial single leaf nest tree.
There are two ways in which this refinement works.
Given a nest tree, we are either able to utilize a large exposed transaction we find in the leaf society of one of the leaf nests in order to split it into two new nests which will be placed as children to the first, which always increases the number of leaf segments by one, or in case this fails, we are able to find a ``tighter'' leaf nest that reduces the part of the graph that belongs to the corresponding leaf society.
This second trick will is in fact what allows us to find exposed transactions in the first place.
This trick originates from a technique used in~\cite{ThilikosW2024Killing}.

\medskip
Let $\rho$ be a rendition of a society $(G, \Omega)$ in the disk and $\mathfrak{T} = (\mathcal{C}, \mathcal{R}, T, r, \phi_{1}, \phi_{2})$ be a nest tree of $(G, \Omega)$.
We say that a nest tree $\mathfrak{T}' = (\mathcal{C}', \mathcal{R}', T, r, \phi'_{1}, \phi'_{2})$ of $(G, \Omega)$ is \emph{tighter (in $\rho$)} than $\mathfrak{T}$ if $\mathfrak{T}$ and $\mathfrak{T}'$ have the same cycle and linkage order, the same reserve, as well as the same number of leaves.
Moreover, they are identical, except for a unique leaf node $l \in V(T)$, for which the following holds.

Let $\{ C_{1}, \ldots, C_{s + 1} \}$ be the leaf nest for $l$ including the boundary cycle for $l$ and $\{ C'_{1}, \ldots, C'_{s+1} \}$ be the leaf nest for $l'$ including the boundary cycle for $l'.$
Then, there exists $i \in [s+1]$ such that, if $H_{i}$ is the crop of $G$ by the $C_{i}$-disk in $\rho$ and $H'_{i}$ is the crop of $G$ by the $C'_{i}$-disk in $\rho$, then either $H'_{i} - V(C'_{i}) \subsetneq H_{i} - V(C_{i})$ or $E(H'_{i}) \subsetneq E(H_{i}).$
In case $i = s + 1$ above, we say that $\mathfrak{T}'$ is \emph{strictly} tighter than $\mathfrak{T}.$
Note that $\mathfrak{T'}$ clearly respects $\mathfrak{T}.$

Finding a tighter nest tree is essentially accomplished by either slightly pushing one of the cycles of a leaf nest closer to the untamed area or by extending the nest into the untamed area.
To achieve this, we introduce the following supporting definitions.

\medskip
Let $(G, \Omega)$ be a society and $\rho$ be a rendition of $(G, \Omega)$ in the disk $\Delta$ and let $\Delta_{0}$ be a $\rho$-aligned disk in $\Delta$.
Also, let $T$ be a closed curve in $\Delta$, or a curve with both endpoints in the boundary of $\Delta$.
Then, if one of the two regions of $\Delta - T$ is a disk whose closure $\Delta'$ contains $\Delta_{0}$, we call $\Delta'$ the \emph{$\Delta_{0}$-disk} of $T$.
For a grounded subgraph $H$ of $G$ containing a cycle that possesses a $\Delta_{0}$-disk, the inclusion-wise minimal $\Delta_{0}$-disk of any cycle in $H$ is the \emph{$\Delta_{0}$-disk} of $H$.

If $T$ is the trace of a grounded cycle or a grounded $V(\Omega)$-path $C$ and $T$ has a $\Delta_{0}$-disk $\Delta'$, then we say that $\Delta'$ is the $\Delta_{0}$-disk of $C$ (not to be confused with the $C$-disk in $\rho$).

Now, given a grounded $C$-path $P$, we say that $P$ \emph{sticks out towards $\Delta_{0}$ (in $\rho$)} if the $\Delta_{0}$-disk of $P \cup C$ is not the $\Delta_{0}$-disk of $C$ and otherwise we say that $P$ \emph{sticks out away from $\Delta_{0}$ (in $\rho$).}

With this we are ready to prove the following lemma.

\begin{lemma}\label{lemma:exposedOrTighten} For all integers $s, p \geq 1$ the following holds.
Let $(G, Z)$ be an annotated graph and $\rho$ be a rendition of a society $(G, \Omega)$ with a $Z$-consistent nest tree $\mathfrak{T}$ of cycle order $s$.
Moreover, let $(G', \Omega')$ be a leaf society of $\mathfrak{T}$ and $\rho'$ be the restriction of $\rho$ to $(G', \Omega')$.
If $\mathcal{P}$ is a $\rho'$-flat transaction in $(G', \Omega')$ of order $2s + p + 2$, then there exists either
\begin{itemize}
\item a $\rho'$-exposed, $\rho'$-flat transaction $\mathcal{Q} \subseteq \mathcal{P}$ of order $p$, or
\item a tighter $Z$-consistent nest tree $\mathfrak{T}'$ of $(G, \Omega)$.
\end{itemize}
Moreover, there exists an algorithm that finds one of the two outcomes in time $\mathbf{poly}(s + p) \cdot |E(G)|$.
\end{lemma}
\begin{proof} Let $l$ be the leaf node of $\mathfrak{T}$ corresponding to $(G', \Omega')$.
Let $\mathcal{C} = \{ C_{1}, \ldots, C_{s+1} \}$ be the leaf nest for $l$ including the boundary cycle for $l$ and $\mathcal{R}$ be the radial linkage for $l$.
Moreover, let $\Delta$ be the $C_{s+1}$-disk in $\rho$ in which the rendition $\rho'$ lives in.

If $\mathcal{P}$ contains a $\rho'$-exposed transaction of order $p$ we are immediately done.
Moreover, we can check in linear time for the existence of such a transaction by checking for each path in $\mathcal{P}$ individually if it is $\rho'$-exposed.
Hence, we may assume that there exists a linkage $\mathcal{Q} \subseteq \mathcal{P}$ of order $2(s + 1) + 1$ such that \textsl{no} path of $\mathcal{Q}$ is $\rho'$-exposed.

Let $\Delta^{*}$ denote the $C_{1}$-disk in $\rho$.
It follows that no path in $\mathcal{Q}$ intersects the interior of $\Delta^{*}$.
Now, for each $Q \in \mathcal{Q}$, let us call the disk defined as the closure of the region of $\Delta$ minus the $\Delta^{*}$-disk of $Q$, the \emph{small side} of $Q$.
Given two members of $\mathcal{Q}$, then either the small side of one is contained in the small side of the other, or their small sides are disjoint.
It is straightforward to see that saying two members are \emph{equivalent} if their small sides intersect, defines an equivalence relation on $\mathcal{Q}$.
Moreover, there are exactly two equivalence classes, one of which, call it $\mathcal{Q}'$, contains at least $s + 2$ members.
Since $|\mathcal{C}| = s + 1$ there must exist some $i \in [2, s + 1]$ and a subpath $L$ of a path in $\mathcal{Q}'$ such that $L$ is a $C_{i}$-path that sticks out towards $\Delta^{*}$ and $V(L) \cap V(C_{j}) = \emptyset$, $j \in [s + 1] \setminus \{ i \}$.

Notice that $C_i \cup L$ contains a unique cycle $C'$ different from $C_i$ whose trace separates $\Delta^*$ from the nodes corresponding to $V(\Omega)$.
Moreover, the $C'$-disk $\Delta''$ in $\rho$ is properly contained in the $C_{i}$-disk in $\rho$.
In particular, there exists an edge of $C_i$ which is not an edge of $C'$ and an edge of $C'$ which is not an edge of $C_i$.
Therefore, the crop of $G$ by $\Delta''$ after deleting the vertices of $C'$ is properly contained in the crop of $G$ by $\Delta'$ after deleting the vertices of $C_i$.

We would now like to define the new nest $\mathcal{C}'$ by replacing the edited cycle with $C'$.
First, we are going to change the cycle $C'$ once more.
Consider $C''$ to be the cycle obtained from $C'$ after iteratively replacing a subpath of $C'$, in the same way as above, with any $C'$-subpath of a path in $\mathcal{R}$ that sticks out towards $\Delta^{*}$.
This step guarantees that the drawing of any $C''$-subpath of a path in $\mathcal{R}$ is now a path that sticks out away from $\Delta^{*}$.
Call this property *.

The next and final step is to edit the paths in the radial linkage $\mathcal{R}$ in order to obtain a new radial linkage $\mathcal{R}'$ that is orthogonal to the new nest $\mathcal{C}''$ obtained by replacing $C'$ with $C''$.
First, observe that the paths in $\mathcal{R}$ may not be orthogonal to $\mathcal{C}''$ only because they are not orthogonal to $C''$.
Now, for each path $R \in \mathcal{R}$, let $R'$ be any $C''$-subpath of $R$ and $C''_{1}$ be the unique subpath of $C''$ with the same endpoints as $R'$ such that the $(R' \cup C''_{1})$-disk $\Delta_{R'}$ in $\rho$ avoids $\Delta^{*}$.
By property *, $\Delta_{R'}$ does not intersect the interior of the $\Delta^{*}$-disk of $C''$.
In fact, $\Delta_{R'}$ cannot intersect the drawing of any other path in $\mathcal{R}$, as this would imply the existence of a $C''$-subpath of a path in $\mathcal{R}$ that sticks out towards $\Delta^{*}$ which cannot exist.
Therefore, we may safely update $R$ by replacing $R'$ by $C''_{1}$.
By repeating this procedure until we no longer find any such subpath $R'$, we conclude with a radial linkage $\mathcal{R}'$ orthogonal to $\mathcal{C''}$.

In fact, since all changes to $\mathcal{R}$ happen locally within $\Delta$, by replacing $\mathcal{R}$ with $\mathcal{R}'$ and the nest $\mathcal{C}$ with the nest $\mathcal{C}''$ in $\mathfrak{T}$, we obtain a tighter nest tree $\mathfrak{T}'$.
If moreover $\mathfrak{T}$ is $Z$-consistent, then $\mathfrak{T}'$ is clearly $Z$-consistent as the cycle $C_{1}$ does not change.
\end{proof}

\paragraph{Orthogonalizing transactions and radial linkages.}

What follows is a lemma that allows us to orthogonalize a transaction with respect to a nest in a given rendition of a society.
Before that we need to introduce the notion of a \textsl{cozy nest} in a society.

\medskip
Let $(G, \Omega)$ be a society and $\rho$ be a rendition of $(G, \Omega)$ in the disk $\Delta$ with a nest $\mathcal{C} = \{ C_{1}, \ldots, C_{s} \}$.
Let $\Delta_{1}$ be the $C_{1}$-disk in $\rho$.

We say that $\mathcal{C}$ is \emph{cozy} if for every $i \in [s]$ and every grounded $C_{i}$-path $P$ that sticks out away from $\Delta_{1}$ we have $V(P) \cap V(\Omega) \neq \emptyset$ or there exists a $j \in [s] \setminus \{ i \}$ such that $V(P) \cap V(C_{j}) \neq \emptyset$.

The following lemma states that we may always turn a nest into a cozy nest.

\begin{proposition}[\cite{GorskySW2025Polynomial}, Lemma 8.4]\label{prop:cozyNest}
Let $s \geq 1$ be an integer, $(G, \Omega)$ be a society, and $\rho$ be a rendition of $(G, \Omega)$ in the disk with a nest $\mathcal{C} = \{ C_{1}, \ldots, C_{s} \}$.
Then there exists a cozy nest $\mathcal{C}'$ of order $s$ in $(G, \Omega)$ such that the union of $C_{1}$ and the outer graph of $C_{1}$ in $\rho$ contains $\bigcup \mathcal{C}'.$

Moreover, there exists an algorithm that finds $\mathcal{C'}$ in time $\mathbf{O}(s \cdot |E(G)|^{2}).$
\end{proposition}

Next, whenever we swap a nest for a cozy nest we need to make sure to update any orthogonal radial linkage to the previous nest to an orthogonal radial linkage in the new nest.

\medskip
Let $(G,\Omega)$ be a society, let $\rho$ be a $\Sigma$-rendition of $(G,\Omega)$ in a disk $\Delta$, and let $\mathcal{P}$ and $\mathcal{Q}$ be two linkages of the same order.
If for each path $Q \in \mathcal{Q}$ there exists a path $P \in \mathcal{P}$ with the same endpoints as $Q$, we say that $\mathcal{P}$ and $\mathcal{Q}$ are \emph{end-identical}.

The next lemma will allow us to do just that.

\begin{proposition}[Gorsky, Seweryn, and Wiederrecht \cite{GorskySW2025Polynomial}]\label{prop:radialtoorthogonal}
Let $s,r$ with $r \leq s$.
Let $(G,\Omega)$ be a society with a $\Sigma$-rendition $\rho$ with a cozy nest $\mathcal{C}$ of order $s$ and a radial linkage $\mathcal{R}$ of order $r$ for $\mathcal{C}$.

Then there exists a radial linkage $\mathcal{R}'$ of order $r$ for $\mathcal{C}$ that is orthogonal to $\mathcal{C}$ and end-identical to $\mathcal{R}$.

Moreover, there exists an algorithm running in $\mathbf{O}(r \cdot |E(G)|)$-time that finds $\mathcal{R}'$.
\end{proposition}

Now we can present the lemma that does the orthogonalization.

\begin{proposition}[Choi et al.\ \cite{ChoiGKMW2025OddCyclePackingtreewidth} (see also Lemma 4.16 in \cite{PaulPTW2024Obstructionsa})]\label{prop:planar_orthogonal_transaction}
Let $s,p$ be positive integers.
Let $(G, \Omega)$ be a society and $\rho$ be rendition of $(G, \Omega)$ in the disk with a cozy nest $\mathcal{C} = \{ C_1, \ldots , C_s \},$ and let $H$ be the outer graph of $C_1$ in $\rho.$
Further, let $\mathcal{P}$ be an exposed transaction of order $s(p+1)+1$ in $(G,\Omega)$ with the end segments $X_1, X_2,$ such that there exists a vortex-free rendition $\rho'$ of $(H \cup \bigcup \mathcal{P}, \Omega)$ that agrees with $\rho$ on $H.$

Then there exists a transaction $\mathcal{P}'$ of order $p$ such that
\begin{enumerate}
    \item $\mathcal{P}'$ is orthogonal to $\mathcal{C}$ and exposed in $\rho$,
    
    \item $\mathcal{P}'$ connects vertices of $X_1 \cap V(\mathcal{P})$ to vertices of $X_2 \cap V(\mathcal{P})$, and

    \item the intersection of $\bigcup\mathcal{P}'$ with the inner graph of $C_1$ in $\rho$ is fully contained in $C_1\cup\bigcup\mathcal{P}$.
\end{enumerate}
Moreover, there exists an algorithm that finds $\mathcal{P}'$ in time $\mathbf{O}( p \cdot |E(G)|)$.
\end{proposition}

We also need the following lemma that allows us to connect connect two radial linkages in a given society with a nest.

\begin{proposition}[Gorsky, Seweryn, and Wiederrecht \cite{GorskySW2025Polynomial}]\label{prop:connected_linkages}
Let $s,r,k,\ell$ be positive integers with $s \geq r+3$.
Let $(G,\Omega)$ be a society with a $\Sigma$-rendition $\rho$ and a nest $\mathcal{C} = \{ C_1, \ldots , C_s \}$.
Moreover, let $\mathcal{L}$ and $\mathcal{R}$ each be radial linkages of order $r$ in $(G,\Omega)$ such that both are orthogonal to $\mathcal{C}$ and let $I = [\ell,k] \subseteq [2,s]$ be an interval with $|I| = r+2$.

Then there exists a radial linkage $\mathcal{P}$ of order $r$ in $(G,\Omega)$ such that
\begin{enumerate}
    \item $\mathcal{P}$ is orthogonal to $\{ C_i ~\!\colon\!~ i \in [s] \setminus I \}$ with endpoints on $C_1$,
    
    \item $H_{\ell}\cap \bigcup\mathcal{P}$ is a subgraph of $H_{\ell}\cap \mathcal{L}$, where $H_{\ell}$ is the inner graph of $C_{\ell}$ in $\rho$.
    In particular, the endpoints of $\mathcal{P}$ on $V(C_1)$ coincide with the endpoints of $\mathcal{L}$ on $V(C_1)$, and
    
    \item $H_k \cap \bigcup \mathcal{P}$ is a subgraph of $H_k \cap \mathcal{R}$, where $H_k$ is the outer graph of $C_k$ in $\rho$.
    In particular, the endpoints of $\mathcal{P}$ on $V(\Omega)$ coincide with the endpoints of $\mathcal{R}$ on $V(\Omega)$.
\end{enumerate}
Moreover, there exists an algorithm that finds $\mathcal{P}$ in time $\mathbf{O}(r|E(G)|)$.
\end{proposition}

We need one more tool that allows us to orthogonalize a radial linkage with respect to some nest in a respectful way.
This is essentially done by iterating the ideas in the proof of \zcref{lemma:exposedOrTighten}.

\begin{lemma}\label{lemma:orthogonalRadial} For all integers $s, r \geq 1$ the following holds.
Let $\rho$ be a rendition of a society $(G, \Omega)$ with a nest $\mathcal{C}$ of order $s$ and a radial linkage $\mathcal{R}$ of order $r$ for $\mathcal{C}$.
Then there exists a nest $\mathcal{C}'$ of order $s$ and a radial linkage $\mathcal{R}'$ of order $r$ for $\mathcal{C}'$ that is orthogonal to $\mathcal{C}'$ and with the same endpoints on $V(\Omega)$ as $\mathcal{R}$.
Further, the disk of the outer cycle of $\mathcal{C}'$ in $\rho$ is contained in the disk of the outer cycle of $\mathcal{C}$ in $\rho$, and if $Z \subseteq V(G)$ such that $\mathcal{C}$ is $Z$-consistent, then $\mathcal{C}'$ is $Z$-consistent.

Moreover, there exists an algorithm that finds $\mathcal{C}'$ and $\mathcal{R}'$ in time $\mathbf{poly}(s + r) \cdot |E(G)|^{2}.$
\end{lemma}
\begin{proof} Let $\mathcal{C} = \{ C_{1}, \ldots, C_{s} \}$.
Moreover, for each $i \in [s]$, let $\Delta_{i}$ be the $C_{i}$-disk in $\rho$.

We define the desired nest $\mathcal{C}' = \{ C'_{1}, \ldots, C'_{s} \}$ by iteratively applying the following rule: As long as there exists a path $R \in \mathcal{R}$, some $i \in [s]$, and a $C_{i}$-subpath $R' \subseteq R$ that sticks out towards $\Delta_{1}$ satisfying $V(R') \cap V(C_{j}) = \emptyset$, for all $j \in [s] \setminus \{ i \}$, we update $C_{i}$ to the unique cycle $C'$ contained in $C_{i} \cup R'$ different from $C_{i}$, whose trace in $\rho$ separates $\Delta_{1}$ from the nodes of $\rho$ corresponding to $V(\Omega)$.
Now observe that the $C'$-disk in $\rho$ is properly contained in $\Delta_{i}$.
Also notice that $C'_{1} = C_{1}$ as since $\mathcal{R}$ is a radial linkage, there are no $C_{1}$-subpaths of any path in $\mathcal{R}$.
In fact this implies that $\mathcal{C}'$ remains $Z$-consistent.
This step can be performed in time $\mathbf{poly}(s + r) \cdot |E(G)|^{2}$.

Now, for $i \in [s]$, let $\Delta'_{i}$ be the $C'_{i}$-disk in $\rho$ accordingly.
It follows from the previous step that any $C'_{i}$-subpath $R'$ of any path $R \in \mathcal{R}$ for $i \in [2, s]$, must either intersect $C'_{i-1}$ or it cannot be drawn $\Delta'_{i}$.
From now on we call such a subpath $R'$ of a path $R \in \mathcal{R}$ a, \emph{regression} of $R$.

Let the paths in $\mathcal{R} = \{ R_1, \ldots, R_r \}$ be ordered according to the occurrence of their endpoints on $\Omega$.
Moreover, assume there exists an $i \in [r]$ and a $j \in [2, s]$ such that $R_i$ has a regression $R'$ with both endpoints on $C'_j$.
Notice that there exists a subpath $L_{R'}$ of $C'_j$ that shares its endpoints with $R'$ which is internally disjoint from $R_i$.
This is because otherwise, we could find a subpath of $R_i$ which witnesses that we are in the previous case.
Let $\Delta_{R'}$ be the $(R' \cup L_{R'})$-disk in $\rho$ which avoids $\Delta_{1}$.
We assume that $R'$ is chosen maximally with the property of being a regression.

Suppose there exists some path $Q$, possibly $Q = R$, in $\mathcal{R}$ whose drawing intersects the interior of $\Delta_{R'}$ in a node or arc distinct from those of $R'$.
Then, this intersection of $Q$ with $\Delta_{R'}$ must belong to a maximal regression $Q'$ of $Q$ and there must exist $j' \in [2, s]$, $j' < j$, such that $Q'$ is a $C_{j'}$-path.
Hence, there must exist some maximal regression $P'$ of some path $P' \in \mathcal{R}$ such that no other part of $\mathcal{R}$ intersects the corresponding disk $\Delta_P'$.
Thus, by replacing $P'$ with $L_{P'}$, we obtain a new radial linkage with the same endpoints in $V(\Omega)$ as $\mathcal{R}$ but with strictly less regressions.
This means that, after at most $E(\mathcal{R})$ many such steps, we must have found a radial linkage $\mathcal{R}'$ which is orthogonal to $\mathcal{C}'$.
This step can be performed in time $\mathbf{poly}(r) \cdot |E(G)|$.
\end{proof}

\paragraph{Blank transactions in a fixed rendition.}
We now know that we may always tighten a given nest tree, or find within a leaf society, a large flat planar transaction which is also exposed.
Moreover, subject to making the nest of the leaf society cozy, we can use \zcref{prop:planar_orthogonal_transaction} to make this transaction orthogonal to the nest.
Similarly to \zcref{lemma:RedTransactionGivesRedMesh}, we now want to use the mesh infrastructure provided by this transaction union the nest, to either find a large red mesh, or conclude with a large subtransaction that avoids all red vertices.

\medskip
Let $(G, \Omega)$ be a society and $\rho$ be a rendition of $(G, \Omega)$ in the disk with a nest $\mathcal{C} = \{ C_{1}, \ldots, C_{s} \}$.
In a slight abuse of terminology, given a set $R \subseteq V(G)$, we say that $\mathcal{C}$ is \emph{$R$-consistent (in $\rho$)} if every vertex of $R$ is a vertex of the crop of $G$ by the $C_{1}$-disk in $\rho$.

Moreover, let $\mathcal{P}$ be a $\rho$-flat transaction in $(G, \Omega)$ of order at least $2$.
We say that $\mathcal{P}$ is \emph{$R$-blank (in $\rho$)} if $\sigma(c) \cap R = \emptyset$, for every cell $c \in C(\rho)$ that is either contained in the $\rho$-container of $\mathcal{P}$ or for which $\sigma(c)$ contains an edge of a boundary path of $\mathcal{P}$.

\begin{lemma}\label{lemma:blankTrans} For all integers $r \geq 3$ and $p \geq 2$ the following holds.
Let $(G, R)$ be an annotated graph and $\rho$ be a rendition of a society $(G, \Omega)$ in the disk with an $R$-consistent nest $\mathcal{C}$ of order $r + 1.$
If $\mathcal{P}$ is a $\rho$-flat, $\rho$-exposed transaction in $(G, \Omega)$ of order $r(r - 1)p$ that is orthogonal to $\mathcal{C}$, then there exists either
\begin{itemize}
\item an $r$-mesh $M \subseteq G$ grounded and red in $\rho$ such that $\mathcal{T}_{M}$ is a truncation of the tangle induced by a mesh whose horizontal paths are subpaths of distinct cycles from $\mathcal{C}$, or
\item an $R$-blank transaction $\mathcal{Q} \subseteq \mathcal{P}$ of order $p$.
\end{itemize}
Moreover, there exists an algorithm that finds one of the two outcomes in time $\mathbf{poly}(k + p) \cdot |E(G)|$.
\end{lemma}
\begin{proof} Let $P_{1} \ldots, P_{r(r-1)p}$ be an ordering of $\mathcal{P}$ so that they are indexed naturally and let $\mathcal{C} = \{ C_{1}, \ldots, C_{r + 1} \}$.
Moreover, let $\Delta^{*}$ be the $C_{1}$-disk in $\rho$.
Also, let $P'_{i}$ be a $C_{1}$-subpath of $P_{i} \in \mathcal{P}$ that is grounded in $\rho$ and $T'_{i}$ be the trace of $P'_{i}$ in $\rho$.
Note that since $\mathcal{P}$ is $\rho$-exposed and orthogonal to $\mathcal{C}$, each $P'_{i}$ is uniquely defined.

Now, for all integers $i < j \in [r(r - 1)]$, there is a unique disk in $\Delta^{*} \setminus (T'_{i} \cup T'_{j})$, whose closure $\Delta_{i}^{j}$ is a $\rho$-aligned disk bounded by both $T'_{i}$ and $T'_{j}$.
Let $H_{i}^{j}$ denote the crop of $G$ by $\Delta_{i}^{j}$ union $\sigma(c)$, for any cell $c \in C(\rho)$ such that $E(\sigma(c)) \cap E(P'_{i} \cup P'_{j}) \neq \emptyset$.
Since $\mathcal{C}$ is $R$-consistent, observe that what we are looking for is a transaction $\mathcal{Q}$ such that, if $i$ is the minimum index among all paths in $\mathcal{Q}$ while $j$ is the maximum, then $H_{i}^{j} \cap R = \emptyset$.

For $i \in [r(r - 1)]$, let $\mathcal{P}_{i} = \{ P_{(i - 1) \cdot p + 1}, \ldots, P_{i \cdot p} \}$, where each $|\mathcal{P}_{i}| = p$.
Now we are in one of two cases.
Either there exists $i \in [r(r - 1)]$, such that $H_{(i - 1)\cdot p + 1}^{i \cdot p}$ contains no vertex of $R$, or all do.
In the former case we are done with $\mathcal{P}_{i}$ being the desired transaction.
So we may assume we are in the latter.

Let $\mathcal{Q}$ be defined by collecting the first path from every $\mathcal{P}_{i},$ $i \in [r(r - 1)]$.
Now observe that since $\mathcal{Q}$ is orthogonal to $\mathcal{C}$ and $|\mathcal{C}| = r + 1$, $|\mathcal{Q}| = r(r - 1)$, there exists a $(2(r+1) \times r(r - 1))$-mesh $M \subseteq \bigcup \mathcal C \cup \bigcup \mathcal{Q}$ such that $\bigcup \mathcal{Q} \subseteq M$.
Moreover, since for every $i \in [r(r - 1) - 1]$, $H_{(i - 1)\cdot p + 1}^{i \cdot p}$ contains a red vertex, we infer that the subgraph of $G$ induced by the $(r + 1, i)$-brick of $M$ contains a vertex of $R$.
Then, by calling \zcref{lemma:MiddleRedMeshToFullyRed}, we may conclude with an $r$-mesh $M' \subseteq G$ such that $\mathcal{T}_{M'}$ is a truncation of the tangle induced by an $r$-submesh $M''$ of $M$ whose horizontal paths are subpaths of distinct cycles from $\mathcal{C}$.
\end{proof}

\subsection{Refining nest trees}

\paragraph{Splitting leaf nests.}
We now have all the tools at our disposal to start splitting the nests of leaf societies for a given nest tree.
We first introduce some additional definitions to ease writing.
\medskip

Let $\rho$ be a rendition of a society $(G, \Omega)$ in the disk with a nest $\mathcal{C}$.
Moreover, let $\Delta$ be the disk of the inner cycle of $\mathcal{C}$ in $\rho$ and $\mathcal{P}$ be a $\rho$-flat, $\rho$-exposed transaction in $(G, \Omega)$ that is orthogonal to $\mathcal{C}$.

Now, notice that since $\mathcal{P}$ is $\rho$-exposed and orthogonal to $\mathcal{C}$, if we remove from $\Delta$ the interior of the $\rho$-container of $\mathcal{P}$, we obtain precisely two $\rho$-aligned disks $\Delta^{*}_{1}$ and $\Delta^{*}_{2}$ which contain precisely the cells of $\rho$ contained in $\Delta$ and not in the $\rho$-container of $\mathcal{P}$.
We call $\Delta^{*}_{1}$ and $\Delta^{*}_{2}$ the \emph{residual vortices} of $\rho$ and $\mathcal{P}$.

\begin{lemma}\label{lemma:SplitLeafs} For all integers $s_{0}, s, t, \ell \geq 1$ with $s_{0} \geq t$, let $p = 2t + 2s_{0} + 2s + 6$.
Then the following holds.

Let $(G, R)$ be an annotated graph and $\rho$ be a rendition of a society $(G, \Omega)$ in the disk with an $R$-consistent nest tree $\mathfrak{T}$ of cycle order $2s_{0} + s + 2$, linkage order $t$, reserve $s_{0}$, and $\ell$ leaves.
Also, let $(G', \Omega')$ be a leaf society of $\mathfrak{T}$, $\rho'$ be the restriction of $\rho$ to $(G', \Omega')$, and $\mathcal{C}$ be the nest of $(G', \Omega')$ in $\rho'$.

Further, let $\mathcal{P}$ be a $\rho'$-flat, $\rho'$-exposed transaction in $(G', \Omega')$ of order $p$ that is $R$-blank and orthogonal to $\mathcal{C}$ and let $\Delta^{*}_{1}$ and $\Delta^{*}_{2}$ be the residual vortices of $\rho'$ and $\mathcal{P}$.

Now, assume that either $\rho'$ has a vortex, or that $V(G') \cap R \neq \emptyset$.
Then,
\begin{enumerate}
\item if for each $i \in [2]$, either $\Delta^{*}_{i}$ contains a vortex, or there is a vertex of $R$ in the crop of $G'$ by $\Delta^{*}_{i}$, then there exists an $R$-consistent nest tree $\mathfrak{T}'$ of $(G, \Omega)$ of cycle order $s$, linkage order $t$, reserve $s_{0}$, and $\ell + 1$ leaves, that respects $\mathfrak{T}$.
\item Otherwise, there exists a strictly tighter $R$-consistent nest tree of $(G, \Omega)$.
\end{enumerate}
Moreover, there exists an algorithm that finds one of the two outcomes in time $\mathbf{poly}(t + s + s_{0}) \cdot |V(G)||E(G)|$.
\end{lemma}
\begin{proof}

Let $l$ be the leaf node of $\mathfrak{T}$ corresponding to $(G', \Omega')$.
Also, assume that $\mathcal{C} = \{ C_{1}, \ldots, C_{2s_{0} + s + 2} \}$, let $\mathcal{R}$ be the radial linkage for $l$ in $\mathfrak{T}$, and let $\Delta$ be the disk of the boundary cycle for $l$ in which the rendition $\rho'$ lives in.

\textbf{Step 1: Setup.} Consider an ordering $P_{1}, \ldots, P_{2t + 2s_{0} + 2s + 4}$ of $\mathcal{P}$ so that it is indexed naturally.
With this, we split $\mathcal{P}$ in $4$ subtransactions to utilize in the construction steps that follow.
\begin{align*}
\mathcal{T}_{1} &\coloneqq \{ P_{2}, \ldots, P_{t + 1} \}\\
\mathcal{S}_{1} &\coloneqq \{ P_{t + 2}, \ldots, P_{t + s_{0} + s + 3} \}\\
\mathcal{S}_{2} &\coloneqq \{ P_{t + s_{0} + s + 4}, \ldots, P_{t + 2s_{0} + 2s + 5} \}\\
\mathcal{T}_{2} &\coloneqq \{ P_{t + 2s_{0} + 2s + 6}, \ldots, P_{2t + 2s_{0} + 2s + 6} \}.
\end{align*}
Notice that, with $\mathcal{P}$ being $R$-blank, each of the $4$ transactions above is also $R$-blank.

Next we define two disks $\Delta_{1}$ and $\Delta_{2}$ as follows.
Let $U_{1}$ denote the cycle consisting of a subpath of $C_{s_{0} + s + 2}$ together with a subpath of $P_{t + s_{0} + s + 3}$ such that the $U_{1}$-disk in $\rho'$ avoids the paths in $\mathcal{S}_{2}$ and $\mathcal{T}_{2}$.
Similarly, let $U_{2}$ denote the cycle consisting of a subpath of $C_{s_{0} + s + 2}$ together with a subpath of $P_{t + s_{0} + s + 4}$ such that the $U_{2}$-disk in $\rho'$ avoids the paths in $\mathcal{T}_{1}$ and $\mathcal{S}_{1}$.
Then, for each $i \in [2]$, we define $\Delta_{i}$ to be the $U_{i}$-disk in $\rho'$.

The two disks $\Delta_{1}$ and $\Delta_{2}$ are our candidates for the new leaf societies.
Notice that for each $i \in [2]$, $\Delta^{*}_{i} \subseteq \Delta_{i}$.
Also note that at this point each $\Delta_{i}$ also contains the infrastracture that we will utilize to define the reserved nests that surround the new leaf societies.

\textbf{Step 2: Deciding the outcome.} Recall that $\mathcal{P}$ is $\rho'$-flat and $R$-blank.
This implies that every vortex of $\rho$ is contained in either $\Delta^{*}_{1}$ or $\Delta^{*}_{2}$.
Moreover, every vertex of $R$ is a vertex of the union of the crop of $G'$ by $\Delta^{*}_{1}$ and the crop of $G'$ by $\Delta^{*}_{2}$.
For each $i \in [2]$, we say that $\Delta^{*}_{i}$ is \emph{promising} if either $\Delta^{*}_{i}$ contains a vortex of $\rho$ or there is a vertex of $R$ that is a vertex of the crop of $G'$ by $\Delta^{*}_{i}$.

Now by assumption, either $\rho'$ has a vortex or $V(G') \cap R \neq \emptyset$.
Therefore, there is $i \in [2]$ that is promising.
If both are promising we proceed with \textbf{Step 3a}.
Otherwise we proceed with \textbf{Step 3b}.

\textbf{Step 3a: The split.} To ensure that the two new leaf societies we wish to produce are connected to the nest in reserve for $l$, we shall redefine it utilizing the outermost cycles of $\mathcal{C}$ and the paths in $\mathcal{T}_{1} \cup \mathcal{T}_{2}$.

Let $\mathcal{C}_{l} = \{ C_{s_{0} + s + 3}, \ldots, C_{2s_{0} + s + 2} \}$, $\Delta^{\mathsf{out}}_{l}$ be the $C_{2s_{0} + s + 2}$-disk in $\rho'$, and $\Delta^{\mathsf{in}}_{l}$ be the $C_{s_{0} + s + 3}$-disk in $\rho'$.
Note that $|\mathcal{C}_{l}| = s_{0}$.
Moreover, we define $\mathcal{R}_{l}$ by cropping each path $R \in \mathcal{R}$ to the minimal subpath with one endpoint in $V(\Omega_{\Delta^{\mathsf{in}}_{l}})$ and the other being the endpoint of $R$ not drawn in $\Delta^{\mathsf{out}}_{l}$.

We now fix the following radial linkages which will be used to append the new leaf societies to $\mathfrak{T}$.
For every $T \in \mathcal{T}_{1} \cup \mathcal{T}_{2}$ let $T'$ be a minimal $V(\Omega_{\Delta^{\mathsf{out}}_{l}})$-$(V(C_{1}) \cap N(\rho'))$-subpath of $T$.
We then set $\mathcal{T}'_{1} \coloneqq \{ T' \mid T \in \mathcal{T}_{1} \}$ and $\mathcal{T}'_{2} \coloneqq \{ T' \mid T \in \mathcal{T}_{2} \}$.

Next, we create new nests for each of the future leaf societies arising from $\Delta_{1}$ and $\Delta_{2}$.
For each $j \in [s_{0} + s + 1]$, let $C^{1}_{j}$ be the cycle contained in $C_{j} \cup P_{t + 1 + j}$ which is fully drawn within $\Delta_{1}$ and such that the $C^{1}_{j}$-disk in $\rho'$ contains $\Delta^{*}_{1}$.
Then, set
$$\mathcal{C}'_{1} \coloneqq \{ C^{1}_{j} \mid j \in [s_{0} + s + 1] \}.$$
Similarly, for each $j \in [s_{0} + s + 1]$, let $C^{2}_{j}$ be the cycle contained in $C_{j} \cup P_{t + s_{0} + s + 3 + j}$ which is fully drawn within $\Delta_{2}$ and such that the $C^{2}_{j}$-disk in $\rho'$ contains $\Delta^{*}_{2}$.
Then, set
$$\mathcal{C}'_{2} \coloneqq \{ C^{2}_{j} \mid j \in [s_{0} + s + 1] \}.$$
Also note that by construction for both $i \in [2]$, $\mathcal{T}_{i}$ is orthogonal to $\mathcal{C}'_{i}$ as well as to $\mathcal{C}_{l}$.

It remains to discuss how we obtain the desired nest tree $\mathfrak{T}'$.
First, we append to $\mathfrak{T}$ two new leaf nodes, say $l_{1}$ and $l_{2}$ as the children of $l$.
For $l$, we discard the nest in reserve for $l$ as well as the boundary cycle for $l$ and place $\mathcal{C}_{l}$ as the nest for $l$ instead.
We moreover replace the radial linkage $\mathcal{R}$ by $\mathcal{R}_{l}$.
Now, for each $i \in [2]$, we associate $l_{i}$ with the nest $\mathcal{C}'_{i}$.
The outermost $s_{0}$ cycles serve as nest in reserve for $l_{i}$, the cycle $C^{i}_{s + 1}$ as the boundary cycle for $l_{i}$, and the rest as the leaf nest for $l_{i}$.
Moreover, we set $\mathcal{T}'_{i}$ to be the radial linkage associated to the edge $ll_{i}$ of (the underlying tree of) $\mathfrak{T}'$.
Finally, it follows by construction that all required properties holds, especially that $\mathfrak{T}'$ remains $R$-consistent.

\textbf{Step 3b: Getting a strictly tighter nest tree.} We may now assume that exactly one of $\Delta^{*}_{1}$ or $\Delta^{*}_{2}$ is not promising.
Without loss of generality we may assume that $\Delta^{*}_{2}$ is not promising.

Now, let $C_{2s_{0} + s + 3}$ be the boundary cycle for $l$ and consider $\mathcal{C}^{*} = \{ C_{1}, \ldots, C_{2s_{0} + s + 3} \}$.
Also, let $\Delta^{*}$ be the $C_{2s_{0} + s + 2}$-disk in $\rho$.

Now, let $\mathcal{T}^{*}$ be defined by cropping each path $T \in \mathcal{T}_{1}$ to a minimal $V(\Omega')$-$(V(C_{1}) \cap N(\rho))$-subpath of $T$, and let $\mathcal{R}^{*}$ be the $\Delta$-truncation in $\rho$ of $\mathcal{R}$.
Note that both $\mathcal{T}^{*}$ and $\mathcal{R}^{*}$ are radial linkages orthogonal to $\mathcal{C}^{*}$.

We may now apply \zcref{prop:connected_linkages}, which is possible since $|\mathcal{C}^{*}| \geq t + 3,$ in order to obtain a radial linkage $\mathcal{L}$ for $\mathcal{C}^{*}$ with the additional property that every edge of a path in $\mathcal{L}$ drawn in the exterior of $\Delta^{*}$ is an edge of a path in $\mathcal{R}^{*}$.
In particular, the endpoints of $\mathcal{L}$ on $V(\Omega')$ are exactly those of $\mathcal{R}^{*}$.
This allows us to extend each path in $\mathcal{L}$ through the non-truncated version of the corresponding path in $\mathcal{R}^{*}$, so as to reach their original endpoint drawn outside of $\Delta^{*}$,
thereby obtaining the radial linkage $\mathcal{L}'$.

The next step is to define the new tighter nest.
This construction is similar to before.
For each $i \in [2s_{0} + s + 3]$, let $C'_{i}$ be the cycle in $C_{i} \cup P_{t + 1 + i}$ such that the $C'_{i}$-disk in $\rho$ contains $\Delta^{*}_{1}$ and avoids the drawing of every path in $\mathcal{T}_{2}$.
Moreover, set $\mathcal{C}' = \{ C'_{i} \mid i \in [2s_{0} + s + 3] \}$.

Clearly, if $\Delta^{\mathsf{out}}$ is the $C'_{2s_{0} + s + 3}$-disk in $\rho$ that contains $\Delta^{*}_{1}$, the crop of $G$ by $\Delta^{\mathsf{out}}$ is a proper subgraph of $G'$, as no path in $\mathcal{T}_{2}$ is contained in it.

Now, by applying \zcref{lemma:orthogonalRadial}, we may assume (up to finding an even tighter nest than $\mathcal{C}'$) that $\mathcal{L}'$ is orthogonal to $\mathcal{C}'$.
By replacing $\mathcal{C}^{*}$ with $\mathcal{C}'$ and $\mathcal{R}$ with $\mathcal{L}'$ in $\mathfrak{T}$ we obtain an $R$-consistent nest tree $\mathfrak{T}'$ that is strictly tighter, thereby concluding the proof.
\end{proof}

\paragraph{Obtaining the final nest tree.}
We are now in the position to show a crucial lemma towards the proof of \zcref{thm:blanksocietyclassification}, which will inductively build a nest tree that either has a target number of leaves, or certifies that all leaf societies have bounded depth.

\begin{lemma}\label{lemma:finalNestTree} There exist polynomial functions $\mathsf{nest}_{\ref{lemma:finalNestTree}} \colon \mathbb{N}^{2} \to \mathbb{N}$ and $\mathsf{depth}_{\ref{lemma:finalNestTree}} \colon \mathbb{N}^{6} \to \mathbb{N}$ such that for all integers $s \geq t \geq r + 1 \geq 4$, $b \geq 0$, $d \geq 2$, and $\ell \geq 1$ the following holds.

Let $(G, R)$ be an annotated graph and $\rho$ be a rendition of a society $(G, \Omega)$ in the disk with breadth $b$ and depth $d$.
Let $\mathcal{C}$ be a nest in $\rho$ of order $\mathsf{nest}_{\ref{lemma:finalNestTree}}(s, \ell)$ that is $R$-consistent.
Further, let $\mathcal{R}$ be a radial linkage in $\rho$ for $\mathcal{C}$ of order $t$ that is orthogonal to $\mathcal{C}$.

Then, there exists either
\begin{enumerate}
\item an $r$-mesh $M \subseteq G$ that is grounded and red in $\rho$ and such that $\mathcal{T}_{M}$ is a truncation of the tangle induced by a mesh whose horizontal paths are subpaths of distinct cycles from $\mathcal{C}$, or
\item an $R$-consistent nest tree $\mathfrak{T}$ in $\rho$ with linkage order $t$ and reserve $s$ such that, either
\begin{itemize}
\item $\mathfrak{T}$ has $\ell + 1$ leaves, or
\item at most $\ell$ leaves and every leaf society of $\mathfrak{T}$ has depth at most $\mathsf{depth}_{\ref{lemma:finalNestTree}}(b, d, r, s, t, \ell)$.
\end{itemize}
Further, the root nest $\mathcal{C}'$ of $\mathfrak{T}$ is a subset of $\mathcal{C}$. 
\end{enumerate}
Moreover, there exists an algorithm that finds $\mathfrak{T}$ in time $\mathbf{poly}(b + d + s + \ell) \cdot |E(G)|^{3}.$
\end{lemma}

Before we begin with the proof of \zcref{lemma:finalNestTree}, we fix a recursive function $\mathsf{nestrec}$ that represents the recursive application of \zcref{lemma:SplitLeafs} and ensures that the final nest tree we obtain has reserve $s_{0}$.
Let $s_{0} \geq t \geq r + 1 \geq 4$ and $\zeta \geq 1$ be integers. We define,
\begin{align*}
\mathsf{nestrec}(s_{0}, 0) &\coloneq 2\\
\mathsf{nestrec}(s_{0}, \zeta) &\coloneqq 2s_{0} + 2 + \mathsf{nestrec}(s_{0}, \zeta - 1).
\end{align*}
Unfolding the recursion we infer that
$$\mathsf{nestrec}(s_{0}, \zeta) = 2 + 2\zeta(s_{0} + 1).$$
Moreover, assuming we start with a rendition in the disk of breadth $b \geq 1$ and depth $d \geq 2$, the following value $p^{*}$ is an estimate on the depth of the leaf societies.
$$p^{*} < (b + 1)(2bd + 12 r^{2} s_{0} (t + (\zeta + 1) s_{0}).$$
Based on these estimates let us define the the two functions $\mathsf{nest}_{\ref{lemma:finalNestTree}}$ and $\mathsf{depth}_{\ref{lemma:finalNestTree}}$, assuming $\ell + 1 \geq 2$ is the goal number of leaves in the final nest tree.
\begin{align*}
\mathsf{nest}_{\ref{lemma:finalNestTree}}(s, \ell) &\coloneqq s + 1 + \mathsf{nestrec}(s, \ell) \in \mathbf{O}(\ell s), \ \text{and}\\
\mathsf{depth}_{\ref{lemma:finalNestTree}}(b, d, r, s, t, \ell) &\coloneqq p^{*} \in \mathbf{O}((b + 1)(bd + r^{2} s t + r^{2}s^{2}\ell)).
\end{align*}

\begin{proof}[Proof of \zcref{lemma:finalNestTree}]

We commence the proof with the construction of a first nest tree $\mathfrak{T}_{1} = (\mathcal{C}_{1}, \mathcal{R}_{1}, T_{1}, r, \phi^{1}_{1}, \phi_{2}^{1})$ in $(G, \Omega)$ from $(\mathcal{C}, \mathcal{R})$ with a single leaf as follows.
Let $(T_{1}, r_{1})$ be a graph with a single vertex $r_{1}$.
Then, let $\phi_{1}^{1}(r_{1}) = \mathcal{C}_{1} = \mathcal{C}$ and $\mathcal{R}_{1} = \mathcal{R}$.

Let $\mathcal{C} = \{ C_{1}, \ldots, C_{\mathsf{nest}_{\ref{lemma:finalNestTree}}(r, s, t, \ell)} \}$ and define $\mathcal{C}^{\mathsf{res}}_{r_{1}} = \{ C_{\mathsf{nest}_{\ref{lemma:finalNestTree}}(s, \ell) - s + 1}, \ldots, C_{\mathsf{nest}_{\ref{lemma:finalNestTree}}(s, \ell)} \}$ and $\mathcal{C}_{r_{1}} = \{ C_{1}, \ldots, C_{\mathsf{nest}_{\ref{lemma:finalNestTree}}(s, \ell) - s - 1} \}$.
Observe that $\mathcal{C}^{\mathsf{res}}_{r_{1}}$ corresponds to the nest in reserve for $r_{1}$, $\mathcal{C}_{r}$ correponds to the leaf nest for $r_{1}$, while $C_{\mathsf{nest}_{\ref{lemma:finalNestTree}}(s, \ell) - s}$ is the boundary cycle for $r_{1}$.
Also, $|\mathcal{C}^{\mathsf{res}}_{r_{1}}| = s$ and $|\mathcal{C}_{r_{1}}| = \mathsf{nestrec}(s, \ell)$.
Hence $\mathfrak{T}_{1}$ is a nest tree of $(G, \Omega)$ of cycle order $\mathsf{nestrec}(s, \ell)$, linkage order $t$, with a single leaf.

We prove the following claim by iteratively applying \zcref{lemma:SplitLeafs}.
Consider an integer $\zeta \in [\ell]$.

\textbf{Claim:} Assume there is no $r$-mesh in $G$ that is red in $\rho$.
Then, if $\mathfrak{T}_{\zeta}$ is an $R$-consistent nest tree in $(G, \Omega)$ of cycle order $\mathsf{nestrec}(s, \ell - \zeta + 1)$, linkage order $t$, reserve $s$ and $\zeta$ leaves, there is an $R$-consistent nest tree $\mathfrak{T}_{\zeta + 1}$ in $(G, \Omega)$ with cycle order $\mathsf{nestrec}(s, \ell - \zeta)$, linkage order $t$, and reserve $s$ that respects $\mathfrak{T}_{\zeta}$, such that, either every leaf society of $\mathfrak{T}_{\zeta + 1}$ has depth at most $\mathsf{depth}_{\ref{lemma:finalNestTree}}(b, d, r, s, t, \ell - \zeta + 1)$, or $\mathfrak{T}'$ has $\zeta + 1$ leaves.

Moreover, there exists an algorithm that, given $\mathfrak{T}$ produces $\mathfrak{T}'$ in time $\mathbf{poly}(b + d + s + \zeta) \cdot |E(G)|^{3}$.

Clearly, by iteratively applying the claim above, starting from $\mathfrak{T}_{1}$, we may construct a sequence of $\mathfrak{T}_{2}, \ldots, \mathfrak{T}_{\xi}$, such that $\xi \leq \ell + 1$ and $\mathfrak{T}_{\xi} = \mathfrak{T}$ is the desired nest tree and conclude.

So consider an integer $\zeta \in [\ell]$ and assume that we are given a nest tree $\mathfrak{T}_{\zeta}$ as above.

We begin by choosing some leaf society $(G', \Omega')$ of $\mathfrak{T}_{\zeta}$ and letting $\rho'$ be the restriction of $\rho$ to $(G', \Omega')$.
Also let $\mathcal{C}_{\zeta}$ be the leaf nest of $(G', \Omega')$ in $\rho'$ and $\mathcal{R}_{\zeta}$ be the radial linkage for this leaf of $\mathfrak{T}_{\zeta}$.
Moreover, if $\Delta_{\zeta}$ is the disk of the boundary cycle for the chosen leaf in $\rho$, let $\mathcal{R}'_{\zeta}$ be the $\Delta_{\zeta}$-truncation in $\rho$ of $\mathcal{R}_{\zeta}$.
Since, by assumption the cycle order of $\mathfrak{T}_{\zeta}$ is $s_{\zeta} \coloneqq \mathsf{nestrec}(s, \ell - \zeta + 1)$, we have that $|\mathcal{C}_{\zeta}| = s_{\zeta}$.

Also, set $s_{\zeta + 1} \coloneqq \mathsf{nestrec}(s, \ell - \zeta)$ and set $p_{\zeta}$ to be the value
$$p_{\zeta} \coloneqq (b + 1)(2bd + (2s_{\zeta} + (s_{\zeta}(r(r - 1)(2t + 2s + 2s_{\zeta + 1} + 6) + 1) + 1) + 2)) - 1$$
and observe that $p_{\zeta} \leq p^{*}$.

In what follows, whenever we encounter an outcome which produces a tighter nest tree, we mark it, and we assume that the other outcome is given.

If every leaf society has depth at most $p_{\zeta}$ we immediately conclude.
Therefore, we may assume that there exists a transaction $\mathcal{P}$ in $(G', \Omega')$ of order $p_{\zeta} + 1$.
By \zcref{lemma:flatTransaction}, we may assume that there is a transaction $\mathcal{P}' \subseteq \mathcal{P}$ in $(G', \Omega')$ of order $2s_{\zeta} + (s_{\zeta}(r(r - 1)(2t + 2s + 2s_{\zeta + 1} + 6) + 1) + 1) + 2$ that is $\rho'$-flat.
Then, by \zcref{lemma:exposedOrTighten}, we either find a \textbf{tighter} $R$-consistent nest tree, or a $\rho'$-exposed transaction $\mathcal{P}'' \subseteq \mathcal{P}'$ in $(G', \Omega')$ of order $s_{\zeta}(r(r - 1)(2t + 2s + 2s_{\zeta + 1} + 6) + 1) + 1$.

Now, we may call \zcref{prop:cozyNest}, to transform $\mathcal{C}_{\zeta}$ into a cozy nest in $(G', \Omega')$.
Notice that by doing so we may have messed with the orthogonality of $\mathcal{R}_{\zeta}$.
This can be easily mended by calling \zcref{prop:radialtoorthogonal} on $\mathcal{R}'_{\zeta}$, which we can do since by assumption $s_{\zeta} \geq t$, and updating the radial linkage accordingly.

This gives us permission to call \zcref{prop:planar_orthogonal_transaction} to obtain a $\rho'$-exposed transaction $\mathcal{P}'''$ of order $r(r - 1)(2t + 2s + 2s_{\zeta + 1} + 6)$ that is moreover orthogonal to $\mathcal{C}_{\zeta}$.
Next, we apply \zcref{lemma:blankTrans}, which we can do since by assumption $s_{\zeta} \geq r + 1$, and obtain one of two outcomes.

The first outcome is an $r$-mesh $M \subseteq G'$ that is grounded and red in $\rho'$ and such that $\mathcal{T}_{M} \subseteq \mathcal{T}_{M'}$, where $M'$ is a mesh whose horizontal paths are subpaths of distinct cycles from $\mathcal{C}_{\zeta}$.
Let $M''$ be mesh of order at least $r$ contained in $\bigcup (\mathcal{C} \cup \mathcal{R})$ whose horizontal paths are subpaths of distinct cycles from $\mathcal{C}$.
Clearly $M'$ exists since $s_{\zeta} \geq t \geq r$.
Now, by induction, since each subsequent nest tree we obtain respects the previous one, we know that the root nest of $\mathfrak{T}_{\zeta}$ is a subset of $\mathcal{C}$.
Moreover, since $s \geq r$, by definition of nest trees and Menger's theorem, we can observe that for any set $S \subseteq V(G)$ with $|S| < r$, there exists a path in $G - S$ that connects a vertex of $M' - S$ to a vertex of $M'' - S$.
This in fact implies that $\mathcal{T}_{M'} \subseteq \mathcal{T}_{M''}$ and therefore $\mathcal{T}_{M} \subseteq \mathcal{T}_{M''}$ which gives the first outcome of the lemma as desired.

Therefore, we may always assume that we obtain the second outcome which is an $R$-blank transaction $\mathcal{P}''' \subseteq \mathcal{P}''$ of order $2t + 2s + 2s_{\zeta + 1} + 6$.
Now, we may finally apply \zcref{lemma:SplitLeafs}, which we can do since by assumption $s_{\zeta} \geq 2s + 2 + s_{\zeta + 1} \geq t,$ and either obtain an $R$-consistent nest tree $\mathfrak{T}_{\zeta + 1}$ with cycle order $s_{\zeta + 1}$, linkage order $t$, reserve $s$, and $\zeta + 1$ leaves, in which case we conclude, or a \textbf{strictly tighter} $R$-consistent nest tree.
Note that, since this last step produces a strictly tighter nest tree, we are still making progress, even though we previously made the nest cozy.

It remains to discuss what we do in the case we produce a tighter nest tree.
Recall, that whenever this is the case, at least one edge of $(G', \Omega')$ is being ``pushed'' towards the ``outside'' of $(G', \Omega')$, and eventually fully outside the updated leaf society.
Since there are at most $|E(G')|$ edges and $s_{\zeta}$ many cycles in $\mathcal{C}_{\zeta}$, we cannot find a tighter nest tree more that $s_{\zeta} |E(G')|$ times.
As a result, we may assume that we always end up in an outcome which does not produce a tighter nest tree and with this we may conclude.
\end{proof}

\subsubsection{Obtaining the blank rendition}

We finally have all the tools required to produce the desired blank rendition.
What we have to show is that, given a nest tree as in the outcome of \zcref{lemma:finalNestTree}, we can produce the surface wall needed for outcome v) of \zcref{thm:blanksocietyclassification}.

We prove two lemmas.
The first shows how starting from a grounded annulus wall in a rendition of a society, we can build a nest tree as in the outcome of \zcref{lemma:finalNestTree} with large enough reserve, and then employing an argument using Menger's Theorem, use the structure provided by the nest tree to link back the leaf societies, with their surrounding nest, to what is left of the annulus wall, thereby building the surface wall.

\begin{lemma}\label{lemma:nestTreeToSurfaceWall} There exist polynomial functions $\mathsf{wall}_{\ref{lemma:nestTreeToSurfaceWall}} \colon \mathbb{N}^{3} \to \mathbb{N}$ and $\mathsf{depth}_{\ref{lemma:nestTreeToSurfaceWall}} \colon \mathbb{N}^{5} \to \mathbb{N}$ such that for all integers $r \geq 3$, $b \geq 0$, $\ell \geq 1$, and $d, k \geq 2$ the following holds.

Let $(G, R)$ be an annotated graph and $\rho$ be a rendition of a society $(G, \Omega)$ in the disk with breadth $b$ and depth $d$.
Further, let $W \subseteq G$ be a $\mathsf{wall}_{\ref{lemma:nestTreeToSurfaceWall}}(r, k, \ell)$-annulus wall that is grounded in $\rho$ such that every vortex of $\rho$ and every vertex of $R$ is contained in the disk of the inner base cycle of $W$.

Then, there exists either
\begin{enumerate}
\item an $r$-mesh $M \subseteq G$ that is grounded and red in $\rho$ and such that $\mathcal{T}_{M}$ is a truncation of $\mathcal{T}_{W}$, or
\item an extended $k$-surface-wall $D \subseteq G$ grounded in $\rho$ such that the base cycles of $D$ are the $k$ outermost base cycles of $W$ and for every vortex segment $S$ of $D$ there is a $\rho$-aligned disk $\Delta_{S}$ contained in the disk bounded by the trace of the inner cycle of $S$ that avoids the trace of the simple cycle of $D$ such that
\begin{itemize}
\item every vortex of $\rho$ is contained in $\Delta_{S}$ for some vortex segment $S$ of $D$,
\item every vertex of $R$ is a vertex of the crop of $G$ by $\Delta_{S}$ for some vortex segment $S$ of $D$, and
\item if for some vortex segment $S$ of $D$, $\Delta_{S}$ contains no vortex of $\rho$, then at least one vertex of $R$ is a vertex of the crop of $G$ by $\Delta_{S}$,
\end{itemize}
and either
\begin{enumerate}
\item $D$ has signature $(0, 0, \ell + 1)$, or
\item $D$ has signature $(0, 0, \ell')$, where $\ell' \leq \ell$, and for every vortex segment $S$ of $D$, the $\Delta_{S}$-society has depth at most $\mathsf{depth}_{\ref{lemma:nestTreeToSurfaceWall}}(r, k, \ell, b, d)$.
\end{enumerate}
\end{enumerate}
Moreover, $\mathsf{wall}_{\ref{lemma:nestTreeToSurfaceWall}}(r, k, \ell) \in \mathbf{O}(r\ell + k\ell^{3})$, $\mathsf{depth}_{\ref{lemma:nestTreeToSurfaceWall}}(r, k, \ell, b, d) \in \mathbf{O}((b + 1)(bd + r^{4}\ell + r^{3} k \ell^{3} + r^{2} k^{2} \ell^{5}))$, and there exists an algorithm that finds one of the two outcomes in time $\mathbf{poly}(r + k + \ell + b + d) \cdot |E(G)|^{3}.$
\end{lemma}
\begin{proof} We begin by defining the functions $\mathsf{wall}_{\ref{lemma:nestTreeToSurfaceWall}}$ and $\mathsf{depth}_{\ref{lemma:nestTreeToSurfaceWall}}$ as follows.
\begin{align*}
\mathsf{wall}_{\ref{lemma:nestTreeToSurfaceWall}}(r, k, \ell) &\coloneqq 4k + \mathsf{nest}_{\ref{lemma:finalNestTree}}(r + 4 k (\ell + 2)^{2}, \ell)\\
\mathsf{depth}_{\ref{lemma:nestTreeToSurfaceWall}}(r, k, \ell, b, d) &\coloneqq \mathsf{depth}_{\ref{lemma:finalNestTree}}(b, d, r, r + 4 k (\ell + 2)^{2}, r + 4 k (\ell + 2)^{2}, \ell).
\end{align*}

Let $\{ C_{1}, \ldots, C_{\mathsf{wall}_{\ref{lemma:nestTreeToSurfaceWall}}(r, k, \ell, b, d)} \}$ denote the base cycles of $W$ from innermost to outermost.
Moreover, let $\mathcal{C} = \{ C_{1}, C_{\mathsf{wall}_{\ref{lemma:nestTreeToSurfaceWall}}(r, k, \ell, b, d) - k} \}$ and $\mathcal{R}$ contain, for each vertical path $P$ of $W$ except $4k$ many consecutive vertical paths which we reserve for the vertical paths of the wall segment of $D$, a minimal $(V(C_{\mathsf{wall}_{\ref{lemma:nestTreeToSurfaceWall}}(r, k, \ell, b, d) - k + 1}) \cap N(\rho))$-$(V(C_{1}) \cap N(\rho))$-subpath of $P$.
Also, let $\Delta$ be the $C_{\mathsf{wall}_{\ref{lemma:nestTreeToSurfaceWall}}(r, k, \ell, b, d) - k + 1}$-disk in $\rho$ and $\rho'$ be the restriction of $\rho$ by $\Delta$ which is a rendition of the $\Delta$-society in $\rho$, say $(G', \Omega')$.

Observe that by definition, $\mathcal{C}$ is a nest in $(G', \Omega')$ of order $\mathsf{nest}_{\ref{lemma:finalNestTree}}(r + 4 k (\ell + 2)^{2}, \ell)$ and $\mathcal{R}$ is a radial linkage in $(G', \Omega')$ for $\mathcal{C}'$ of order at least $r + 4 k (\ell + 2)^{2}$.
Moreover, $\mathcal{C}$ is $R$-consistent by the assumptions on $W.$

We may now call upon \zcref{lemma:finalNestTree} for $(G', \Omega')$, $\mathcal{C}$, and $\mathcal{R}$ with parameters $s = t = r + 4k(\ell + 2)^{2} \geq r + 1.$

In the fist case, we get an $r$-mesh $M \subseteq G$ that is grounded and red in $\rho$ and such that $\mathcal{T}_{M}$ is a truncation of the tangle induced by a mesh $M'$ whose horizontal paths are subpaths of distinct cycles from $\mathcal{C}$.
Since the cycles of $\mathcal{C}$ are base cycles of $W$ and $W$ has at least $|\mathcal{C}|$ radial paths, the order of $\mathcal{T}_{W}$ is at least $|\mathcal{C}| \geq r$.
Moreover, since $\mathcal{C}$ is a subset of the base cycles of $W$, it must be that $\mathcal{T}_{M'} \subseteq \mathcal{T}_{W}$, and we may conclude with $\mathcal{T}_{M} \subseteq \mathcal{T}_{W}$ as desired.

In the second case, we get an $R$-consistent nest tree $\mathfrak{T}$ in $\rho'$ with linkage order $r + 4 k (\ell + 2)^{2}$ and reserve $r + 4 k (\ell + 2)^{2}$ such that, either
\begin{itemize}
\item $\mathfrak{T}$ has $\ell + 1$ leaves, or
\item at most $\ell$ leaves and every leaf society of $\mathfrak{T}$ has depth at most $\mathsf{depth}_{\ref{lemma:nestTreeToSurfaceWall}}(r, k, \ell, b, d)$.
\end{itemize}
We now show how to utilize $\mathfrak{T}$ to obtain the desired extended $t$-surface-wall $D$ in $(G, \Omega)$.

Let $\zeta \leq \ell + 1$ denote the exact number of leaves of $\mathfrak{T}$. 
Now, let $G'' \subseteq G'$ be the graph defined by taking the union of all cycles from all nests involved in $\mathfrak{T}$ and all paths of all radial linkages involved in $\mathfrak{T}$, except for the radial linkage of the root node of $\mathfrak{T}$, where we take precisely $4 k \zeta(\zeta + 1)$ paths, and the radial linkages of leaf nodes of $\mathfrak{T}$, where we take precisely $4 k (\zeta + 1)$ paths.

Let $X \subseteq V(G'')$ be the endpoints in $V(\Omega')$ of all paths in the radial linkage of the root node, and let $Y \subseteq V(G'')$ be the set of all endpoints of the radial linkages of leaf nodes on the inner cycle of their corresponding leaf nest.
Note that $|X| = |Y| = 4 k \zeta(\zeta + 1)$.

We now utilize Menger's theorem in $G''$, asking for an $X$-$Y$ linkage $\mathcal{P}$ of order $4 k \zeta(\zeta + 1)$.
We claim that such a linkage exists.
Towards a contradiction, let us assume that this is not the case.
Then, there exists a set $Z \subseteq V(G'')$ with $|Z| < 4 k \zeta(\zeta + 1)$ that intersects every $X$-$Y$-path in $G''$.
First, let us observe that there exists at least one leaf node $l$ of $\mathfrak{T}$ for which, if $\mathcal{C}_{l}$ and $\mathcal{R}_{l}$ are the corresponding leaf nest and radial linkage, there is at least one cycle in $\mathcal{C}_{l}$ and at least one path in $\mathcal{R}_{l}$ that $Z$ cannot intersect.
Indeed, let $r = t_{0}, t_{1}, \ldots, t_{\eta}, t_{\eta + 1} = l$ be the root to leaf path reaching $l$ in $\mathfrak{T}$, and let $\mathcal{C}_{r}, \mathcal{R}_{r}, \ldots, \mathcal{C}_{t_{i}}, \mathcal{R}_{t_{i}}, \ldots, \mathcal{R}_{t_{\eta - 1}}, \mathcal{C}_{l}$, be an alternating sequence of nest and radial linkages from $\mathfrak{T}$ such that for $i \in [\eta]$, $\mathcal{C}_{t_{i}}$ is the nest associated to $t_{i}$ and $\mathcal{R}_{t_{i}}$ is the radial linkage associated to the edge $t_{i}t_{i + 1}$ in $\mathfrak{T}$.

Observe that, by definition of nest trees, for each choice of a cycle $C_{t_{i}} \in \mathcal{C}_{t_{i}}$ and each choice of path $R_{t_{i}} \in \mathcal{R}_{t_{i}}$, the graph $\bigcup_{i} C_{t_{i}} \cup R_{t_{i}}$ is connected.
Therefore, since $\mathfrak{T}$ has linkage order and reserve at least $4 k \zeta(\zeta + 1)$, there is such a choice of cycles and paths that is disjoint from $Z$.
In this graph we can clearly find an $X$-$Y$-path, which is a contradiction, and therefore our claim holds.

Moreover, since we choose precisely $4 k (\zeta + 1)$ paths from each leaf radial linkage, the $X$-$Y$-linkage we find is orthogonal to each leaf nest.
There is one more technical issue that we have to resolve.
Menger's Theorem does not allow us to control exactly how the endpoints of these paths may be shuffled along $\Omega'$.
However, by a simple pigeonhole argument, since $\mathfrak{T}$ has $\zeta$ leaves and we chose the numbers so that $\nicefrac{|\mathcal{P}|}{(\zeta + 1)} = 4 k \zeta$, there exist disjoint segments $I_{1}, \ldots, I_{\zeta}$ of $\Omega'$ such that there is a subset $\mathcal{P}_{i} \subseteq \mathcal{P}$ with $\mathcal{P}_{i} = 4 k $, of paths with one endpoint in $V(I_{i})$ and the other in the inner cycle of the leaf nest for the \textsl{same} leaf node.

Let $\mathcal{P}' =  \bigcup_{i \in \zeta} \mathcal{P}_{i}$.
We may now extend each path in $\mathcal{P}'$ via the vertical path of $W$ with which it shares an endpoint in $V(\Omega')$, all the way to its other endpoint in $V(C_{\mathsf{wall}_{\ref{lemma:nestTreeToSurfaceWall}}(r, k, \ell, b, d)})$.
With this step we have managed to attach all leaf nests onto the remainder of $W$ which gives us the desired extended $k$-surface-wall $D$ with signature $(0, 0, \zeta)$.
By the assumptions on $\mathfrak{T}$, all desired properties for $D$ follow directly.
\end{proof}

The next and final lemma shows that if we obtain an extended surface wall as above with a sufficient number of vortex segments we may still conclude with a red mesh.

\begin{lemma}\label{lemma:SurfaceWallToRedMesh} For every integer $r \geq 3$ the following holds.
Let $(G, R)$ be an annotated graph and $\rho$ be a rendition of a society $(G, \Omega)$ in the disk.
Further, let $D \subseteq G$ be an extended $(r(r - 1))$-surface-wall grounded in $\rho$ with signature $(0, 0, r(r - 1) - 1)$ such that, for every vortex segment $S$ of $D$ there is a $\rho$-aligned disk $\Delta_{S}$ contained in the disk bounded by the trace of the inner cycle of $S$ that avoids the trace of the simple cycle of $D$ and there is a vertex of $R$ that is a vertex of the crop of $G$ by $\Delta_{S}$.
Then, there exists an $r$-mesh $M \subseteq G$ that is grounded and red in $\rho$ and such that $\mathcal{T}_{M}$ is a truncation of $\mathcal{T}_{D}$.

Moreover, there exists an algorithm that finds $\mu$ in time $\mathbf{poly}(r) \cdot |E(G)|$.
\end{lemma}
\begin{proof}
Let $\mathcal{C} = \{ C_{1}, \ldots, C_{r(r - 1)} \}$ be the base cycles of $D$ ordered from innermost to outermost and $\mathcal{P} = \{ P_{1}, \ldots, P_{4br(r - 1)} \}$ be the vertical paths of $D$ ordered from left to right so that $P_{1}, \ldots, P_{2r(r - 1)}$ and $P_{4br(r - 1) - 2r(r - 1) + 1}, \ldots, P_{4br(r - 1)}$ are the vertical paths of the wall segment of $D$.
With respect to this ordering, for every vortex segment $S$ of $D$, let $\mathcal{C}_{S} = \{ C^{S}_{1}, \ldots, C^{S}_{r(r - 1)} \}$ be the nest of $S$ ordered from outermost to innermost and $\mathcal{R}_{S} = \{ R^{S}_{1}, \ldots, R^{S}_{4r(r - 1)} \}$ be the rails of $S$ ordered left to right.
Moreover, if $S_{i}$ is the $i$-th vortex segment from left to right in this ordering, let $\mathcal{P}_{S_{i}}$ be the subset of $\mathcal{P}$ that are vertical paths of $S_{i}$ ordered from left to right as well.

We first explain how to obtain an $(2(r + 1) \times r(r - 1))$-mesh $M \subseteq$ such that for every $j \in [r(r-1) - 1]$ the subgraph of $G$ induced by the $(r + 1, j)$-brick of $M$ contains a vertex of $R$.
We do so by explaining how to define its $2(r+1)$ horizontal paths $T_{1}, \ldots, T_{2(r+1)}$ and its $r(r - 1)$ vertical paths $Q_{1}, \ldots, Q_{r(r - 1)}$.

We define the $i$-th horizontal path $T_{i}$ as $P_{i}$ when $i \leq r + 1$ and as $P_{4br(r - 1) r + i}$ otherwise.
Moreover, we define the $r(r - 1)$-th vertical path $Q_{r(r - 1)}$ as the subpath of $C_{r(r - 1)}$ with endpoints on $T_{1}$ and $T_{2r}$ that intersects all $T_{i}$'s.

For $i \in [r(r - 1) - 1]$, we define the $i$-th vertical path as follows.
We start from a vertex of $T_{i}$ on $C_{i}$, and move towards the right, until we hit the $i$-th vertical path $P^{S_{i}}_{i} \in \mathcal{P}_{S_{i}}$.
From there we move upwards along $P^{S_{i}}_{i}$, continuing on $R^{S_{i}}_{i} \in \mathcal{R}_{S_{i}}$, until we hit the $i$-th cycle $C^{S_{i}}_{i} \in \mathcal{C}_{S_{i}}$.
Then we continue on $C^{S_{i}}_{i}$ until we reach $R^{S_{i}}_{4r(r - 1) - i + 1} \in \mathcal{R}_{S_{i}}$, by following the subpath of $C^{S_{i}}_{i}$ that does not intersect any $R^{S_{i}}_{j}$, where $i < j < 4r(r - 1) - i + 1$.
From there, we continue downwards, onto the $4r(r - 1) - i + 1$-th vertical path $P^{S_{i}}_{4r(r - 1) - i + 1} \in \mathcal{P}_{S_{i}}$, until we reach $C_{i}$ once more.
We repeat the above and visit all vortex segments to the right of $S_{i}$ in precisely the same way.
Finally, we continue along $C_{i}$ towards the right, until we intersect $T_{2(r+1)}$, where we stop.

It is not hard to see that $M$ is indeed the desired mesh.
See~\zcref{fig:SurfaceWallToRedMesh} for an illustration of the above procedure.
In particular we may observe that, the subgraph induced by the $(r+1, j)$-brick of $M$, $j \in r(r - 1) - 1$, contains the crop of $G$ by $\Delta_{S_{i}}$ and therefore at least one vertex of $R$.
We may now conclude by applying \zcref{lemma:MiddleRedMeshToFullyRed} which gives us the desired $r$-mesh $M' \subseteq G$ that is grounded and red in $\rho$.
Clearly $\mathcal{T}_{M'} \subseteq \mathcal{T}_{D}$.
\end{proof}

\begin{figure}[ht]
\centering
\includegraphics{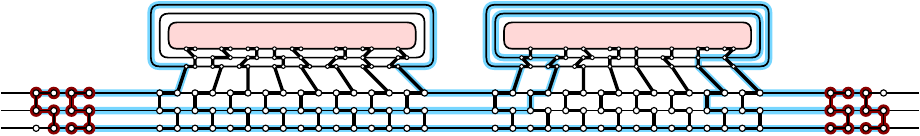}
\caption{\label{fig:SurfaceWallToRedMesh} An illustration for the proof of \zcref{lemma:SurfaceWallToRedMesh} of how a $(4 \times 3)$-mesh is found  within an (subgraph of an) extended $3$-surface-wall with signature $(0, 0, 2).$}
\end{figure}

\subsection{The proof of \texorpdfstring{\zcref{thm:blanksocietyclassification}}{Theorem 4.2}}\label{sec:ProofOfSociety}

We are finally ready to prove \zcref{thm:blanksocietyclassification}.
We repeat the statement here for convenience.

\blanksocietyclassification*
We begin by discussing the functions involved.
Let $t^{*} = 3t$.
This value signifies the size of a clique minor we want as an outcome of \zcref{thm:societyclassification}.
This value now fully determines the following functions:
\begin{align*}
\mathsf{apex}^{\mathsf{genus}}_{\ref{thm:blanksocietyclassification}}(t) &\coloneqq \mathsf{apex}^{\mathsf{genus}}_{\ref{thm:societyclassification}}(t^{*})\\
\mathsf{loss}_{\ref{thm:blanksocietyclassification}}(t) &\coloneqq \mathsf{loss}_{\ref{thm:societyclassification}}(t^{*})
\end{align*}
This further determines the breadth of the rendition in outcome iv) of \zcref{thm:societyclassification} which is $b^{2} \leq b^{1} = \nicefrac{1}{2}(t^{*} - 3)(t^{*} - 4) - 1$.
The order of transaction we need to make it homogeneous is $p^{*} = p^{2}$.

Moreover, the order of surface-wall and the number of vortex segments it has that we are looking for in the end to be able to land in one of the two cases of outcome v) of our proof is, $k^{1} = k + r(r - 1)$ and $\ell^{1} = b^{1} + r(r - 1) - 2$ respectively.

This determines the breadth of the final rendition which is: $\ell^{1} - 1 = \nicefrac{3}{2}(t - 1)(3t - 4) + r(r - 1) - 3.$

Also, the order of the surface-wall we will need to obtain in outcome iv) of \zcref{thm:societyclassification} in order to be able to further refine it for outcome v) of our result is, $k^{*} = (3k^{**} + 2)^{2}$, where $k^{**} = \mathsf{wall}_{\ref{lemma:nestTreeToSurfaceWall}}(r, k^{1}, \ell^{1})$.

These requirements determine the following functions:
\begin{align*}
\mathsf{cost}_{\ref{thm:blanksocietyclassification}}(r, t, k) &\coloneqq \mathsf{cost}_{\ref{thm:societyclassification}}(t^{*}, k^{*})\\
\mathsf{apex}_{\ref{thm:blanksocietyclassification}}^{\mathsf{fin}}(r, t, k, p) &\coloneqq \mathsf{apex}_{\ref{thm:societyclassification}}^{\mathsf{fin}}(t^{*}, k^{*}, p^{*})
\end{align*}
Under these assumptions the depth of the rendition we get in outcome iv) is $d^{2} \leq d^{1} = \mathsf{depth}_{\ref{thm:societyclassification}}(t^{*}, k^{*}, p^{*})$.
Then finally we can define the last two functions:
\begin{align*}
\mathsf{nest}_{\ref{thm:blanksocietyclassification}}(r, t, k) &\coloneqq \mathsf{nest}_{\ref{thm:societyclassification}}(t^{*}, k^{*})\\
\mathsf{depth}_{\ref{thm:blanksocietyclassification}}(r, t, k, p) &\coloneqq \mathsf{depth}_{\ref{lemma:nestTreeToSurfaceWall}}(r, k^{1}, \ell^{1}, b^{1}, d^{1})
\end{align*}

Based on the estimates of the functions involved in \zcref{thm:societyclassification} and \zcref{lemma:finalNestTree} we derive the following bounds for the previously defined functions.
\begin{align*}
\mathsf{nest}_{\ref{thm:blanksocietyclassification}}(r, t, k) &\in \mathbf{O}((k + r^{2})^{18}(t^{2} + r^{2})^{54}),\\
\mathsf{apex}^{\mathsf{genus}}_{\ref{thm:blanksocietyclassification}}(t) &\in \mathbf{O}(t^{8}), \ \ \mathsf{loss}_{\ref{thm:blanksocietyclassification}}(t) \in \mathbf{O}(t^{3}), \ \ \text{and} \ \ \mathsf{cost}_{\ref{thm:blanksocietyclassification}}(r, t, k) \in \mathbf{O}((k + r^{2})^{2}(t^{2} + r^{2})^{6}),\\
\mathsf{apex}^{\mathsf{fin}}_{\ref{thm:blanksocietyclassification}}(r, t, k, p) &\in \mathbf{O}((k + r^{2})^{74}(t^{2} + r^{2})^{222} + p^{74}), \ \ \text{and}\\
\mathsf{depth}_{\ref{thm:blanksocietyclassification}}(r, t, k, p) &\in \mathbf{O}((k + r^{2})^{102}(t^{2} + r^{2})^{308} + t^{4}p^{102}).
\end{align*}

\begin{proof}
Consider the variables we have defined above.
We commence with an application of \zcref{thm:societyclassification} on $(G, \Omega)$, $\rho$, and $\mathcal{C}$, with $t = t^{*}$, $p = p^{*}$, and $k = k^{*}$.
Observe that this is valid as $s \geq \mathsf{nest}_{\ref{thm:blanksocietyclassification}}(r, t, k) = \mathsf{nest}_{\ref{thm:societyclassification}}(t^{*}, k^{*})$.
Let $(G', \Omega')$ be the $C_{s - \mathsf{cost}_{\ref{thm:societyclassification}}(t^{*}, k^{*})}$-society in $\rho$.

Then, $G'$ contains a set $A \subseteq V(G')$ such that we land in one of the following outcomes which we discuss separately.

\textbf{Outcome i):} There exists a $K_{t^{*}}$-minor-model $\mu$ in $G$ controlled by $M$.

Here, we may apply \zcref{prop:redClique} to $G$ and $\mu$ and obtain one of two things.
The first outcome is a red-minor-model $\mu'$ of $K_{t}$ such that $\mathcal{T}_{\mu'} \subseteq \mathcal{T}_{\mu}$, in which case we conclude.
The second outcome is a separation $(X, Y) \in \mathcal{T}_{\mu}$ of order at most $t - 1$ such that $(Y \setminus X) \cap R = \emptyset$, in which case we also conclude.

\textbf{Outcome ii):} There exists a flat, isolated crosscap transaction $\mathcal{P}$ of order $p^{*}$ in $(G' - A, \Omega')$, with $|A| \leq \mathsf{apex}^\mathsf{genus}_{\ref{thm:societyclassification}}(t^{*})$, and a nest $\mathcal{C}'$ in $\rho$ of order $s - (\mathsf{loss}_{\ref{thm:societyclassification}}(t^{*}) + \mathsf{cost}_{\ref{thm:societyclassification}}(t^{*}, k^{*}))$ to which $\mathcal{P}$ is orthogonal.

Here, we may apply \zcref{lemma:homoTransactions} to $\mathcal{P}$ and $\mathcal{C}'$ and obtain a subtransaction $\mathcal{Q} \subseteq \mathcal{P}$ of order $p$ that is homogeneous.
It follows that any subtransaction of an isolated transaction remains isolated and we may conclude.

\textbf{Outcome iii):} There exists a flat, isolated handle transaction $\mathcal{P}$ of order $p^{*}$ in $(G' - A, \Omega')$, with $|A| \leq \mathsf{apex}^\mathsf{genus}_{\ref{thm:societyclassification}}(t^{*})$, and a nest $\mathcal{C}'$ in $\rho$ of order $s - (\mathsf{loss}_{\ref{thm:societyclassification}}(t^{*}) + \mathsf{cost}_{\ref{thm:societyclassification}}(t^{*}, k^{*}))$ to which $\mathcal{P}$ is orthogonal.

Let $\mathcal{R}$ and $\mathcal{Q}$ be the constituent planar, flat, and isolated transactions of order $p^{*}$ that make up $\mathcal{P}$.
We may apply \zcref{lemma:homoTransactions} to both $\mathcal{R}$ and $\mathcal{Q}$ independently, along with $\mathcal{C}'$, and obtain a subtransaction $\mathcal{R'} \subseteq \mathcal{R}$ of order $p$ that is homogeneous, and a subtransaction $\mathcal{Q'} \subseteq \mathcal{Q}$ that is also homogeneous, which concludes this case as well.

\textbf{Outcome iv):} There exists a rendition $\rho'$ of $(G - A, \Omega)$ in $\Delta$ with breadth $b_{2}$ and depth at most $d_{2}$, $|A| \leq \mathsf{apex}^\mathsf{fin}_{\ref{thm:societyclassification}}(t^{*}, k^{*}, p^{*})$, and an extended $k^{*}$-surface-wall $D$ with signature $(0, 0, b^{2})$, such that $D$ is grounded in $\rho'$, the base cycles of $D$ are the cycles $C_{s - \mathsf{cost}_{\ref{thm:societyclassification}}(t^{*}, k^{*}) - 1 -k^{*}}, \ldots, C_{s - \mathsf{cost}_{\ref{thm:societyclassification}}(t^{*}, k^{*}) - 1}$, and there exists a bijection between the vortices $v$ of $\rho'$ and the vortex segments $S_v$ of $D$, where $v$ is the unique vortex contained in the disk $\Delta_{S_v}$ defined by the trace of the inner cycle of the nest of $S_v$, and $\Delta_{S_v}$ is chosen to avoid the trace of the simple cycle of $D$.

Here our goal is to apply \zcref{lemma:nestTreeToSurfaceWall}.
We first have to prep $D$ for that.
We discard most of $D$ and keep a $k^{*}$-subwall $W$ of the wall segment of $D$.
Clearly, $W$ is grounded in $\rho$ and of order $(3k^{**} + 2)^{2}$.
We can therefore call upon \zcref{lemma:nestTreeToSurfaceWall} and obtain a $3k^{**}$-subwall $W' \subseteq W$ that is homogeneous in $\rho$.
Note that if $W'$ is red in $\rho$ we may immediately conclude as $3k^{**} \geq r$.
Therefore, we may assume that $W'$ is blank.
The next step is to define from $W'$ a $k^{**}$-annulus wall $W'' \subseteq W'$.
This can be done by iteratively peeling of $k^{**}$ layers from $W'$ starting from its perimeter, and using the remaining $k^{**}$ vertical paths to define the vertical paths of $W'$.
Also note, that by construction, and since $W'$ is blank, every vortex of $\rho$ and every vertex of $R \setminus A$, is contained in the interior of the disk of the inner base cycle of $W'$.

We are now in the position to apply \zcref{lemma:nestTreeToSurfaceWall} to $W'$ and as a result, obtain either an $r$-mesh $M \subseteq G - A$ that is grounded and red in $\rho'$ and such that $\mathcal{T}_{M} \subseteq \mathcal{T}_{W'}$, in which case we may terminate, or obtain an extended $k^{1}$-surface-wall $D' \subseteq G - A$ grounded in $\rho'$ such that the base cycles of $D'$ are the $k^{1}$ outermost base cycles of $W'$ and for every vortex segment $S'$ of $D'$ there is a $\rho'$-aligned disk $\Delta_{S'}$ contained in the disk bounded by the trace of the inner cycle of $S'$ that avoids the trace of the simple cycle of $D'$ such that
\begin{itemize}
\item every vortex of $\rho'$ is contained in $\Delta_{S'}$ for some vortex segment $S'$ of $D'$,
\item every vertex of $R \setminus A$ is a vertex of the crop of $G - A$ by $\Delta_{S'}$ for some vortex segment $S'$ of $D$, and
\item if for some vortex segment $S'$ of $D'$, $\Delta_{S'}$ contains no vortex of $\rho'$, then at least one vertex of $R \setminus A$ is a vertex of the crop of $G - A$ by $\Delta_{S'}$,
\end{itemize}
and either
\begin{enumerate}
\item $D'$ has signature $(0, 0, b^{1} + r(r - 1) - 1)$, or
\item $D'$ has signature $(0, 0, b')$, where $b \leq b^{1} + r(r - 1) - 2)$, and for every vortex segment $S'$ of $D'$, the $\Delta_{S'}$-society has depth at most $\mathsf{depth}_{\ref{lemma:nestTreeToSurfaceWall}}(r, k^{1}, \ell^{1}, b^{1}, d^{1})$.
\end{enumerate}
Now, in case we are in the first outcome above, and $D'$ has signature $(0, 0, b^{1} + r(r - 1) - 1)$, we know that there exist at least $r(r - 1) - 1$ distinct vortex segments $S'$ of $D'$ such that there is a vertex of $R \setminus A$ that is a vertex of the crop of $G - A$ by $\Delta_{S'}$.
Since $k^{1} \geq r(r - 1)$, we can invoke \zcref{lemma:SurfaceWallToRedMesh} and conclude with an $r$-mesh $M \subseteq G - A$ that is grounded and red in $\rho'$ and such that $\mathcal{T}_{M} \subseteq \mathcal{T}_{D}$.
In this case we conclude since clearly $\mathcal{T}_{D} \subseteq \mathcal{T}_{M}$.

Hence, lastly, we may assume that $D'$ has signature $(0, 0, b')$, where $b \leq b^{1} + r(r - 1) - 2)$, and for every vortex segment $S'$ of $D'$, the $\Delta_{S'}$-society has depth at most $\mathsf{depth}_{\ref{lemma:nestTreeToSurfaceWall}}(r, k^{1}, \ell^{1}, b^{1}, d^{1})$.
In this case, we define a new rendition $\rho''$ from $\rho'$, by removing for each disk $\Delta_{S'}$ all cells of $\rho'$ from its interior, and replacing them with a new vortex cell whose closure is $\Delta_{S'}$.
Since, $k^{1} \geq k$ and each vortex of $\rho''$ has depth at most $\mathsf{depth}_{\ref{lemma:nestTreeToSurfaceWall}}(r, k^{1}, \ell^{1}, b^{1}, d^{1})$, our proof concludes.
\end{proof}

\section{A local structure theorem for red minors}\label{sec:localstructure}
In this section, building on \zcref{thm:blanksocietyclassification}, we obtain a local structure theorem for annotated graph minors akin to the local structure theorem for graph minors, that exhibits the distribution of red vertices, relative to the almost embedding of the graph.
The proof we give essentially follows the structure of the proof of the local structure theorem in \cite{GorskySW2025Polynomial}, with modifications made to account for our setting of annotated graphs.

\paragraph{Blank landscapes.}
We begin with the description of an object that will coherently describe the decomposition of the graph we seek to obtain.
The definition we present here is again a mild variation on what can be found in \cite{GorskySW2025Polynomial} and represents a very specific description of how our graph and in particular the surface wall within it embeds into the surface.
Its use is contained to the proof of the local structure theorem itself.
\medskip

Let $k, w \geq 4$ be integers, let $(G,R)$ be an annotated graph, and let $\Sigma$ be a surface of Euler-genus $g$.
Let $h$, $c$, and $b$ be non-negative integers where $g=2h+c$ and $c \neq 0$ if and only if $\Sigma$ is non-orientable.
Moreover, let $D \subseteq G$ be a $k$-surface-wall with signature $(h,c,b)$, let $W \subseteq G-A$ be a $w$-mesh in $G$, and let $\mathcal{T}_D$ and $\mathcal{T}_W$ be the tangles they respectively define.
Finally, let $A \subseteq V(G) \setminus V(D)$.

The tuple $\Lambda = (A,W,D,\rho)$ is called a \emph{blank $\Sigma$-landscape} of \emph{detail $k$} if
\begin{description}
    \item[L1~~] $\rho$ is a blank $\Sigma$-rendition of $(G - A, R \setminus A),$
    \item[L2~~] $D$ and $W$ are grounded in $\rho,$
    \item[L3~~] $W$ is flat in $\rho,$
    \item[L4~~] the disk bounded by the trace of the simple cycle of $D$ in $\rho$ avoids the traces of the other base cycles of $D,$
    \item[L5~~] the tangle $\mathcal{T}_D$ is a truncation of the tangle $\mathcal{T}_W,$
    \item[L6~~] if $C$ is a cycle from the nest of some vortex-segment of $D,$ then the trace of $C$ is a contractible closed curve in $\Sigma,$
    \item[L7~~] $\rho$ has exactly $b$ vortices and there exists a bijection between the vortices $v$ of $\rho$ and the vortex segments $S_v$ of $D$ such that $v$ is the unique vortex of $\rho$ that is contained in the $v$-disk $\Delta_{C_1}$ of the inner cycle of $S_v$, where $\Delta_{C_1}$ avoids the trace of the simple cycle of $D,$ and
    \item[L8~~] for every vortex $v$ of $\rho$, the society induced by the outer cycle from the nest of the corresponding vortex segment has a cross or $\sigma(v)$ contains a vertex of $R$.
\end{description}
We refer to the vortex segments as the \emph{vortices} of $\Lambda$ and call $A$ the \emph{apex set}.
The integer $b$ is called the \emph{breadth} of $\Lambda$ and the \emph{depth} of $\Lambda$ is the depth of $\rho$.
We say that $\Lambda$ is \emph{centred} at the mesh $W$.

\paragraph{Blank layouts.}
Let $k \geq 4$, $l$, $d$, $b$, $r$, and $a$ be non-negative integers and $\Sigma$ be a surface.
We say that an annotated graph $(G,R)$ with a mesh $M$ has a \emph{blank $k$-$(a,b,d,r)$-$\Sigma$-layout} $\Lambda$ \emph{centred at $M$} if there exists a set $A \subseteq V(G)$ of size at most $a$ and a submesh $M'\subseteq M$ such that there exists a blank $\Sigma$-landscape $(A,M',D,\rho)$ of detail $k$, breadth $b$, and depth $d$ for $G$ where every vortex of $\rho$ has a linear decomposition of adhesion at most $d$, and $M'$ is a $w$-mesh with $w \geq a+b(2d+1)+6+r$.

\begin{theorem}\label{thm:strongest_localstructure}
There exist functions $\mathsf{apex}_{\ref{thm:strongest_localstructure}},\mathsf{depth}_{\ref{thm:strongest_localstructure}} \colon \mathbb{N}^3 \to \mathbb{N}$ and $\mathsf{mesh}_{\ref{thm:strongest_localstructure}} \colon \mathbb{N}^4 \to \mathbb{N}$ such that for all integers $t \geq 1,$ $r, k \geq 4,$ and $w \geq 3$, every annotated graph $(G,R)$, and every $\mathsf{mesh}_{\ref{thm:strongest_localstructure}}(r, t, k, w)$-mesh $M \subseteq G$ one of the following holds.
\begin{enumerate}
    \item there exists a separation $(X, Y) \in \mathcal{T}_M$ of order at most $t-1$ such that $(Y \setminus X) \cap R = \emptyset$,
    \item $(G, R)$ has a red $K_{t}$-minor model controlled by $M$,
    \item $(G, R)$ has a red, flat $r$-mesh $M'$ such that $\mathcal{T}_{M'}$ is a truncation of $\mathcal{T}_{M}$, or
    \item $(G, R)$ has a blank $k$-$(\mathsf{apex}_{\ref{thm:strongest_localstructure}}(t, r, k), \nicefrac{3}{2}(t-1)(3t-4) + r(r-1)-3, \mathsf{depth}_{\ref{thm:strongest_localstructure}}(t, r, k), w)$-$\Sigma$-layout $\Lambda$ centred at $M$ and the surface $\Sigma$ has genus less than $9t^2$.
\end{enumerate}
Moreover, it holds that

{\centering
  $ \displaystyle
    \begin{aligned}
        \mathsf{apex}_{\ref{thm:strongest_localstructure}}(t,r,k),~ \mathsf{depth}_{\ref{thm:strongest_localstructure}}(t,r,k) \in  & \ \mathbf{O}\big((t+r+k)^{2833952}\big), \text{ and} \\
        \mathsf{mesh}_{\ref{thm:strongest_localstructure}}(t,r,k,w) \in                                                             & \ \mathbf{O}\big((t+r+k)^{5667844} + t^2w \big) .
    \end{aligned}
  $
\par}

There also exists an algorithm that, given $t$, $k$, $r$, $w$, a graph $G$, and a mesh $M$ as above as input finds one of these outcomes in time $\mathbf{poly}(t+k+r+w)|E(G)|^3$.
\end{theorem}

\paragraph{Extending the surface and other tools.}
To prove \zcref{thm:strongest_localstructure} we import two lemmas from~\cite{KawarabayashiTW2021Quickly} (see Lemma 10.2 and Lemma 10.5) that will allow us, starting from \zcref{lemma:flatMeshToHomegeneousMesh}, to inductively extend the surface by adding handles and crosscaps.
We note that we adopt the minor changes that are made to these statements in \cite{GorskySW2025Polynomial}, which involves only considering a radial linkage instead of a linkage from $\Omega$ to the vortex.
The cost of this is very minor, reflecting solely in a mildly increased number of cycles lost.
These tools are powerful enough so that we barely have to modify them for our purposes.
We note that the original statements do not deal with blank renditions.
However, since our condition for blankness simply involves checking the location of red vertices, the \emph{reconciliation lemma} from \cite{KawarabayashiTW2021Quickly} (see Lemma 5.15, and also see Proposition 5.5 and Corollary 5.6 in \cite{ChoiGKMW2025OddCyclePackingtreewidth}) directly translates to preserve blankness.
Thus we state our version of these theorems to also preserve blankness.

\begin{proposition}[Kawarabayashi, Thomas, and Wollan \cite{KawarabayashiTW2021Quickly}]\label{lemma:integrate-crosscap}
Let $s$ and $p$ be non-negative integers.
Let $(G,\Omega)$ be a society with a blank, cylindrical rendition $\rho_0$ in the disk $\Delta$ with a nest $\mathcal{C}=\{ C_1, \ldots , C_{s+9} \}$ around the vortex $c_0$.
Let $X_1,X_2$ be disjoint segments of $\Omega$ such that there exist
\begin{itemize}
    \item a radial linkage $\mathcal{R}$ orthogonal to $\mathcal{C}$ starting in $X_1$, and
    \item a blank, flat, isolated crosscap transaction $\mathcal{P}$ of order at least $p+2s+7$ with all endpoints in $X_2$ and disjoint from $\mathcal{R}$.
\end{itemize}
Let $\Sigma^*$ be a surface, homeomorphic to the projective plane minus an open disk, which is obtained from $\Delta$ by adding a crosscap to the interior of $c_0$.

Then there exists a crosscap transaction $\mathcal{P}'\subseteq \mathcal{P}$ of order $p$, consisting of the middle $p$ paths of $\mathcal{P}$, and a blank rendition $\rho_1$ of $(G,\Omega)$ in $\Sigma^*$ (around $\bigcup \mathcal{C}$) with a unique vortex $c_1$ and the following hold:
\begin{enumerate}
    \item $\mathcal{P}'$ is disjoint from $\sigma(c_1)$,

    \item the vortex society of $c_1$ in $\rho_1$ has a blank cylindrical rendition $\rho_1'$ with a nest $\mathcal{C}'=\{ C_1',\dots,C_s'\}$ around the unique vortex $c_1'$,

    \item every element of $\mathcal{R}$ has an endpoint in $V(\sigma_{\rho_1'}(c_1'))$,

    \item $\mathcal{R}$ is orthogonal to $\mathcal{C}'$ and for every $i\in[s]$ and every $R\in\mathcal{R}$, $C_i'\cap R= C_{i+8}\cap R$. Moreover,

    \item let $\mathcal{R}=\{ R_1,\dots,R_{\ell}\}$.
    For each $i\in[\ell]$ let $x_i$ be the endpoint of $R_i$ in $X_1$, and let $y_i$ be the last vertex of $R_i$ on $c_1$ when traversing along $R_i$ starting from $x_i$; if $x_1,x_2,\dots,x_{\ell}$ appear in $\Omega$ in the order listed, then $y_1,y_2,\dots,y_{\ell}$ appear on $\mathsf{boundary}(c_1)$ in the order listed.

    \item Finally, let $\Delta'$ be the open disk bounded by the trace of $C_{s+8}$ in $\rho_0$.
    Then $\rho_0$ restricted to $\Delta\setminus\Delta'$ is equal to $\rho_1$ restricted to $\Delta\setminus \Delta'$.
\end{enumerate}
Moreover, there exists an algorithm that computes this outcome in time $\mathbf{poly}(sp|\mathcal{R}|)|V(G)|$.
\end{proposition}

\begin{proposition}[Kawarabayashi, Thomas, and Wollan \cite{KawarabayashiTW2021Quickly}]\label{lemma:integrate-handle}
Let $s$ and $p$ be non-negative integers.
Let $(G,\Omega)$ be a society with a blank cylindrical rendition $\rho_0$ in the disk $\Delta$ with a nest $\mathcal{C}=\{ C_1,\dots,C_{s+9}\}$ around the vortex $c_0$.
Let $X_1,X_2$ be disjoint segments of $\Omega$ such that there exist
\begin{itemize}
    \item a radial linkage $\mathcal{R}$ orthogonal to $\mathcal{C}$ starting in $X_1$, and
    \item a blank, flat, isolated handle transaction $\mathcal{P}$ of order at least $2p+4s+12$ with all endpoints in $X_2$ and disjoint from $\mathcal{R}$.
\end{itemize}
Let $\mathcal{P}_1$ and $\mathcal{P}_2$ be the two planar transactions such that $\mathcal{P}=\mathcal{P}_1\cup\mathcal{P}_2$.
Let $\Sigma^+$ be a surface, homeomorphic to the torus minus an open disk, which is obtained from $\Delta$ by adding a handle to the interior of $c_0$.

Then there exist transactions $\mathcal{P}_1'\subseteq \mathcal{P}_1$ and $\mathcal{P}_2'\subseteq\mathcal{P}_2$ of order $p$, each $\mathcal{P}_i'$ consisting of the middle $p$ paths of $\mathcal{P}_i$, $i\in[2]$, such that $\mathcal{P}'=\mathcal{P}_1'\cup\mathcal{P}_2'$ is a handle transaction, and a blank rendition $\rho_1$ of $(G,\Omega)$ in $\Sigma^+$ (around $\bigcup \mathcal{C}$) with a unique vortex $c_1$ and the following hold:
\begin{enumerate}
    \item $\mathcal{P}'$ is disjoint from $\sigma(c_1)$,

    \item the vortex society of $c_1$ in $\rho_1$ has a blank cylindrical rendition $\rho_1'$ with a nest $\mathcal{C}'=\{ C_1',\dots,C_s'\}$ around the unique vortex $c_1'$,

    \item every element of $\mathcal{R}$ has an endpoint in $V(\sigma_{\rho_1'}(c_1'))$,

    \item $\mathcal{R}$ is orthogonal to $\mathcal{C}'$ and for every $i\in[s]$ and every $R\in\mathcal{R}$, $C_i'\cap R= C_{i+8}\cap R$. Moreover,

    \item let $\mathcal{R}=\{ R_1,\dots,R_{\ell}\}$.
    For each $i\in[\ell]$ let $x_i$ be the endpoint of $R_i$ in $X_1$, and let $y_i$ be the last vertex of $R_i$ on $c_1$ when traversing along $R_i$ starting from $x_i$; if $x_1,x_2,\dots,x_{\ell}$ appear in $\Omega$ in the order listed, then $y_1,y_2,\dots,y_{\ell}$ appear on $\mathsf{boundary}(c_1)$ in the order listed.

    \item Finally, let $\Delta'$ be the open disk bounded by the trace of $C_{s+8}$ in $\rho_0$.
    Then $\rho_0$ restricted to $\Delta\setminus\Delta'$ is equal to $\rho_1$ restricted to $\Delta\setminus \Delta'$.
\end{enumerate}
Moreover, there exists an algorithm that computes this outcome in time $\mathbf{poly}(sp|\mathcal{R}|)|V(G)|$.
\end{proposition}

Another important ingredient in the proof of the local structure theorem is the ability to build layouts via finding surface walls and $K_t$-minors within them.
For this purpose we first introduce so-called ``surface configurations''.

\medskip
Let $(G,\Omega)$ be a society with a cylindrical rendition and a nest $\mathcal{C}$ around the vortex $c_0$.
Further let $\mathcal{P}_1,\dots,\mathcal{P}_{\ell}$ be a set of transactions on $(G,\Omega)$ as well as $\mathcal{R}$ be a radial linkage such that
\begin{itemize}
    \item $V(\mathcal{P}_i)\cap V(\mathcal{P}_j)=\emptyset$ for all $i\neq j\in[\ell]$ as well as $V(\mathcal{R})\cap V(\mathcal{P}_i)=\emptyset$ for all $i\in[\ell]$,
    \item for each $i\in[\ell]$, $\mathcal{P}_i$ is orthogonal to $\mathcal{C}$ and $\mathcal{R}$ is orthogonal to $\mathcal{C}$, and
    \item there exist pairwise disjoint segments $I_1,J_1,I_2,J_2,\dots,I_{\ell},J_{\ell},R$ of $\Omega$ such that these segments appear on $\Omega$ in the order listed, for each $i\in[\ell]$, $\mathcal{P}_i$ is a $V(I_i)$-$V(J_i)$-linkage, and each path from $\mathcal{R}$ has one endpoint in $V(R)$.
\end{itemize}
We call the tuple $(G,\Omega,\mathcal{C},\mathcal{R},\mathfrak{P}=\{ \mathcal{P}_1,\dots,\mathcal{P}_{\ell}\})$ a \emph{surface configuration} of $(G,\Omega)$ and the sequence $(I_1,J_1,I_2,J_2,\dots,I_{\ell},J_{\ell},R)$ the \emph{signature} of $(G,\Omega,\mathcal{C},\mathcal{R},\mathfrak{P})$.
Finally we say that $(s,r,p_1,\dots,p_{\ell})$ is the \emph{strength} of $(G,\Omega,\mathcal{C},\mathcal{R},\mathfrak{P})$ if $|\mathcal{C}|=s$, $|\mathcal{R}|=r$ and $|\mathcal{P}_i|=p_i$ for all $i\in[\ell]$.

\begin{observation}[Gorsky, Seweryn, and Wiederrecht \cite{GorskySW2025Polynomial}]\label{obs:surface-configs-to-walls}
Let $k \geq 3$, $\ell$, $h$, and $c$ be non-negative integers with $\ell = h + c$.
Moreover, let $s \geq k$, $r \geq 4k$, $p_i \geq 4k$ for all $i \in [\ell]$.

Then, for every surface configuration $(G,\Omega,\mathcal{C},\mathcal{R},\mathfrak{P})$ of strength $(s,r,p_1,\dots,p_{\ell})$ with $h$ handle-transactions and $c$ crosscap-transactions, $G$ contains a $k$-surface-wall $W$ with $h$ handles and $k$ crosscaps as a subgraph such that the base cycles of $W$ are cycles from $\mathcal{C}$ and the vertical paths of the $k$-wall-segment of $W$ are subpaths of the paths from $\mathcal{R}$.
\end{observation}

We will also need the following result that allows us to extract minors from surface walls representing surfaces of sufficient genus.

\begin{proposition}[Gorsky, Seweryn, and Wiederrecht \cite{GorskySW2025Polynomial}]\label{prop:universal-surface-walls}
There exists a universal constant $c_{\ref{prop:universal-surface-walls}}$ such that for all non-negative integers $c,h,t$, every graph $H$ that embeds in a surface homeomorphic to the surface obtained from the sphere by adding $h$ handles and $c$ crosscaps is a minor of the extended $\big( c_{\ref{prop:universal-surface-walls}}g^4(|V(H)|+g)^2\big)$-surface wall $W$ with $h$ handles, $c$ crosscaps, and 0 vortices, where $g=2h+c$.
Moreover, if $H$ is a complete graph, then there exists an $H$-minor-model controlled by $W$.

In particular, $K_t$ is a minor for every extended $k$-surface-wall with $h$ handles and $c$ crosscaps where $2h+c=t^2$ and $k\in \mathbf{\Omega}(t^{12})$.
\end{proposition}

\paragraph{Proof of the annotated local structure theorem.}
We start by giving estimates on some of the functions involved in the proof which themselves help to give estimates on the functions mentioned in \zcref{thm:strongest_localstructure}.
The following functions determine the order of the transactions and the order of the nests we will need to find throughout the proof.
Both functions depend on a parameter $g$ we use to keep track on the number of ``genus-increasing'' steps we have performed in order to facilitate an inductive proof.
\begin{align*}
    \mathsf{radial}(g,t,k)\coloneqq~& (g+2)(8k+8c_{\ref{prop:universal-surface-walls}}t^{12}+4+ \nicefrac{4k}{2}(t-3)(t-4)+1)\\
    \mathsf{nest}(g,t,r,k)\coloneqq~ &(g+2)\big(8k+8c_{\ref{prop:universal-surface-walls}}t^{12} + \mathsf{radial}(g-1,t,k) + 14\\
    &+ \mathsf{cost}_{\ref{thm:blanksocietyclassification}}(r,t,4(k+c_{\ref{prop:universal-surface-walls}}t^{12}+1)) + \mathsf{loss}_{\ref{thm:societyclassification}}(t) \big)\\
    & + \mathsf{nest}_{\ref{thm:blanksocietyclassification}}(r,t,4(k+c_{\ref{prop:universal-surface-walls}}t^{12}+1)) + \nicefrac{4k}{2}(t-3)(t-4) + r + 2\\
    \mathsf{transaction}(g,t,r,k)\coloneqq~ & \max\big(4k+4c_{\ref{prop:universal-surface-walls}}t^{12}+24 + 4\mathsf{nest}(g-1,t,r,k) + \\
    & \ 2\mathsf{radial}(g-1,t,k), 3r(r-1)-2\big)
\end{align*}
With these values fixed, the functions of \zcref{thm:strongest_localstructure} are as follows.
\begin{align*}
    \mathsf{apex}_{\ref{thm:strongest_localstructure}}(t,r,k) \coloneqq~& 9t^2 \cdot \mathsf{apex}^\mathsf{genus}_{\ref{thm:blanksocietyclassification}}(t) + 16t^3\\
    & + \mathsf{apex}^\mathsf{fin}_{\ref{thm:blanksocietyclassification}}(r,t,4(k+c_{\ref{prop:universal-surface-walls}}t^{12}+1),\mathsf{transaction}(9t^2,t,r,k))\\
    \mathsf{depth}_{\ref{thm:strongest_localstructure}}(t,r,k) \coloneqq~& \mathsf{depth}_{\ref{thm:blanksocietyclassification}}(r,t,4(k+c_{\ref{prop:universal-surface-walls}}t^{12}+1),\mathsf{transaction}(9t^2,t,r,k))\\
    \mathsf{mesh}_{\ref{thm:strongest_localstructure}}(t,r,k,w) \coloneqq ~& \mathsf{rfw}_{\ref{thm:FlatRedMesh}}\Big(t,\mathsf{apex}_{\ref{thm:strongest_localstructure}}(t,k)+2\mathsf{nest}(9t^2,t,r,k)+9+w+\nicefrac{1}{2}(t-3)(t-4) \\
    & \cdot (2\mathsf{depth}_{\ref{thm:blanksocietyclassification}}(r,t,4(k+c_{\ref{prop:universal-surface-walls}}t^{12}+1),\mathsf{transaction}(9t^2,t,r,k))+1)\Big)
\end{align*}
Using the bounds provided in the proof of \zcref{thm:blanksocietyclassification}, we can give the following estimates for the functions above:
We have
\[ \mathsf{radial}(g,t,k) \in \mathbf{O}((g+t+k)^{13}) \text{ and } \mathsf{nest}(g,t,r,k), \mathsf{transaction}(g,t,r,k) \in \mathbf{O}((g+t+r+k)^{1728}) . \]
This in turn allows us to determine that
\begin{align*}
    &\mathsf{apex}_{\ref{thm:strongest_localstructure}}(t,r,k) \in \mathbf{O}((t+r+k)^{2045952}), \ \mathsf{depth}_{\ref{thm:strongest_localstructure}}(t,r,k) \in \mathbf{O}((t+r+k)^{2833952}), \text{ and}\\
    &\mathsf{mesh}_{\ref{thm:strongest_localstructure}}(t,r,k,w) \in \mathbf{O}((t+r+k)^{5667844} + t^2w).
\end{align*}

\begin{proof}[Proof of \zcref{thm:strongest_localstructure}]
Let $M$ be a $\mathsf{mesh}_{\ref{thm:strongest_localstructure}}(t,r,k,w)$-mesh in $G$.
We begin by applying \zcref{thm:FlatRedMesh}.
The first outcome of this theorem leads us directly into the first outcome of the result we wish to show.
The second outcome of \zcref{thm:FlatRedMesh} similarly fits our desired second outcome exactly.
Thus we may proceed under the assumption that we find an $r'$-submesh $M'$ of $M$ and a set $Z \subseteq V(G) \setminus V(M')$ with $|Z| < 9t^2$ such that $M'$ is a homogeneous, flat mesh in $G - Z$, with
\begin{align*}
    r'  \geq~&  \mathsf{apex}_{\ref{thm:strongest_localstructure}}(t,k) + 2\mathsf{nest}(9t^2,t,r,k) + 9 + w \\
        ~&      + \nicefrac{1}{2}(t-3)(t-4)(2\mathsf{depth}_{\ref{thm:blanksocietyclassification}}(r,t,4(k+c_{\ref{prop:universal-surface-walls}}t^{12}+1),\mathsf{transaction}(9t^2,t,r,k))+1) .
\end{align*}
Thus in particular, if $M'$ turns out to be red, we have landed in the third outcome we wish to prove.
We may therefore assume that $M'$ is blank.

Let $\mathcal{C}_0'=\{ {C^0_1}',\dots,{C^0}'_{\mathsf{nest}(9t^2,t,r,k)}\}$ be the set of the $\mathsf{nest}(9t^2,t,r,k)$ vertex-disjoint cycles obtained by iteratively peeling off the perimeters ${C^0_i}'$ of the $(x-2(i-1))$-submeshes of $M'$ that are obtained by this procedure, for $i\in[\mathsf{nest}(9t^2,t,r,k)]$.
Moreover, let $M_0$ be the $\big(\mathsf{apex}_{\ref{thm:strongest_localstructure}}(t,k)+\nicefrac{1}{2}(t-3)(t-4)(2\mathsf{depth}_{\ref{thm:strongest_localstructure}}(t,4(k+c_{\ref{prop:universal-surface-walls}}t^{12}+1),\mathsf{transaction}(9t^2,t,r,k))+1)+9+ r\big)$-submesh of $M'$ that is left over after removing the cycles in $\mathcal{C}'_0$ from $M'$ and let $U_0$ be the perimeter of $M_0$.

\begin{observation}\label{obs:localstrucobs1}
Let $I\subseteq [\mathsf{nest}(9t^2,t,r,k)]$ with $|I|=k$ and let $O$ be a cylindrical $k$-mesh with base cycles ${C^0_i}'$, $i\in I$, and rails taken from subpaths of the paths of $M$.
Then the tangle $\mathcal{T}_O$ is a truncation of the tangle $\mathcal{T}_{M_0}$.
\end{observation}

Let $\Sigma_0$ be the sphere.
As $M'$ is flat and blank in $G$, there exists a blank $\Sigma_0$-rendition $\delta_0$ for $G$ with one vortex $c_0$ such that $M'$ is grounded in $\delta_0$ and the trace of the perimeter of $M'$ bounds a closed disk $\Delta_0' \subseteq \Sigma_0$ that is disjoint from $c_0$ and every vertex in $V(M')$ that is a node in $\delta_0$ is drawn in $\Delta_0'$.

Let $X_0$ be the set of all vertices of the perimeter $U_0$ of $M_0$ that are grounded in $\delta_0$ and let $\Delta_0''$ be the \textsl{open} disk in $\Sigma_0$ that is bounded by the trace of $U_0$ and disjoint from the vortex of $\delta_0$.
Further, let $\Delta_0$ be the closed disk $\Sigma_0 \setminus \Delta_0''$.
It follows that the $\Delta_0$-society $(G_0,\Omega_0)$ has a blank, cylindrical rendition $\rho_0$ in the disk $\Delta_0$ with the nest $\mathcal{C}_0'$ around the unique vortex $c_0$.
Indeed, we may select $\mathsf{radial}(9t^2,t,k)$ many paths from the mesh $M'$ to obtain a radial linkage $\mathcal{R}_0'$ that links vertices from distinct rows on $U_0$ to the cycle ${C^0_1}'$.

By using \zcref{prop:cozyNest}, we may find a \textsl{cozy} nest $\mathcal{C}_0=\{ C_1^0,\dots,C^0_{\mathsf{nest}(9t^2,t,r,k)}\}$ around $c_0$ in $\rho_0$ in time $\mathbf{O}(\mathsf{nest}(9t^2,t,r,k)|E(G)|^2)$.
By using \zcref{prop:radialtoorthogonal} we can also find a radial linkage $\mathcal{R}_0$ of order $\mathsf{radial}(9t^2,t,k)$ that is end-identical with $\mathcal{R}_0'$ and orthogonal to $\mathcal{C}_0$ in time $\mathbf{O}(\mathsf{radial}(9t^2,t,k)\mathsf{nest}(9t^2,t,r,k)|E(G)|)$.

\begin{observation}\label{obs:localstrucobs2}
Let $I\subseteq [\mathsf{nest}(9t^2,t,r,k)]$ with $|I|=k$ and let $O$ be a cylindrical $k$-mesh with base cycles $C^0_i$, $i\in I$, and rails taken from subpaths of the paths of $M$.
Then the tangle $\mathcal{T}_O$ is a truncation of the tangle $\mathcal{T}_{M_0}$.
\end{observation}

Moreover, if $\overline{G_0}$ is the graph $G - (V(G_0)\setminus V(\Omega_0))$, then $(\overline{G_0},\Omega_0)$ has a vortex-free, blank rendition derived from $\delta_0$ in the surface obtained from $\Sigma_0$ by removing an open disk.
Finally, let $\mathcal{C}^*\coloneqq \{ C^0_{\mathsf{nest}(9t^2,t,r,k)-4k-4c_{\ref{prop:universal-surface-walls}}t^{12}-3},\dots,C^0_{\mathsf{nest}(9t^2,t,r,k)}\}$ be the outermost $4(k+c_{\ref{prop:universal-surface-walls}}t^{12}+1)$ cycles of $\mathcal{C}^0$.
Then $(G_0,\Omega_0,\mathcal{C}^*,\mathcal{R}_0,\emptyset)$ is a surface configuration of $(G_0,\Omega_0)$ for the surface $\Sigma_0$ of strength $(4(k+c_{\ref{prop:universal-surface-walls}}t^{12}+1),\mathsf{radial}(9t^2,t,k))$.

\paragraph{Inductively refining a near embedding: The setup.}
Observe that, as a consequence of \zcref{obs:surface-configs-to-walls}, the tuple $\Lambda_0=(A_0,M_0,D_0,\delta_0)$ is in fact a $\Sigma_0$-landscape of detail $4(k+c_{\ref{prop:universal-surface-walls}}t^{12}+1)$, where $D_0$ is a $(4(k+c_{\ref{prop:universal-surface-walls}}t^{12}+1))$-surface-wall without handles or crosscaps whose base cycles coincide with the cycles in $\mathcal{C}^*$.

We now prove that, given the following collection of objects for some $i\in[0,t^2-1]$, we can apply \zcref{thm:blanksocietyclassification} to the society $(G_i,\Omega_i)$ to either find a small separation that splits off all red vertices, find a red $K_t$-minor controlled by $M$, find a red, flat $r$-mesh whose tangle is a truncation of the tangle of $M$, create all of the objects below for $i+1$, or through the last outcome of \zcref{thm:blanksocietyclassification}, create a blank $\Sigma_i$-layout centred at $M$.
Since we have already shown how to reach the point $i=0$, we will eventually be able to assume that this last outcome of \zcref{thm:blanksocietyclassification} comes true.

For now, let us introduce the objects we are looking for.
Let $i\in[0,9t^2-1]$ be given together with
\begin{itemize}
    \item an apex set $A_i\subseteq V(G)$ with $A_{i-1}\subseteq A_i$, $A_{-1}\coloneqq\emptyset$, and $|A_i|\leq i\cdot \mathsf{apex}^\mathsf{genus}_{\ref{thm:blanksocietyclassification}}(t)+16t^3$,
    \item a surface $\Sigma_i$ obtained from the sphere by adding a total of $i$ handles and crosscaps in some combination,
    \item a blank $\Sigma_i$-rendition $\delta_i$ with a unique vortex $c_i$,
    \item a $(4(k+c_{\ref{prop:universal-surface-walls}}t^{12}+1))$-surface wall $D_i$ with the same amount of handles and crosscaps used to obtain $\Sigma_i$ from the sphere, such that the base cycles of $D_i$ coincide with the cycles of $\mathcal{C}^*$,
    \item a society $(G_i,\Omega_i)$ such that there exists a $\delta_i$-aligned disk $\Delta_i$ whose boundary intersects $\delta_i$ exactly in $V(\Omega_i)$ and the restriction $\rho_i$ of $\delta_i$ to $\Delta_i$ is a cylindrical rendition of $(G_i,\Omega_i)$ with $c_i$ being its unique vortex,
    \item a cozy nest $\mathcal{C}_i=\{C^i_1,\dots,C^i_{\mathsf{nest}(9t^2-i,t,r,k)} \}$ of order $\mathsf{nest}(9t^2-i,t,r,k)$ in $\rho_i$ around $c_i$,
    \item a family of transactions $\mathfrak{P}_i=\{ \mathcal{P}_1,\dots,\mathcal{P}_i\}$ on $(G_0,\Omega_0)$ such that $\mathcal{P}_j$ is a handle or crosscap transaction,
    \item a radial linkage $\mathcal{R}_i$ of order $\mathsf{radial}(9t^2-i,t,k)$ orthogonal to $\mathcal{C}_i\cup\mathcal{C}^*$ whose endpoints on $(G_0,\Omega_0)$ coincide with some of the endpoints of $\mathcal{R}_0$ and which is disjoint from the paths in $\bigcup\mathfrak{P}_i$, and
    \item all objects above are chosen such that $(G_0-A_i,\Omega_0,\mathcal{C}^*,\mathcal{R}_i,\mathfrak{P}_i)$ is a $\Sigma_i$-configuration of strength $(4(k+c_{\ref{prop:universal-surface-walls}}t^{12}+1),\mathsf{radial}(9t^2-i,t,k),p_1,\dots,p_i)$ with $p_j=4k+4c_{\ref{prop:universal-surface-walls}}t^{12}$ for all $j \in [i]$.
\end{itemize}

By \zcref{obs:surface-configs-to-walls}, the last point ensures the existence of the $(4(k+c_{\ref{prop:universal-surface-walls}}t^{12}+1))$-surface-wall $D_i$.
Moreover, if we ever reach the point $i=9t^2$, \zcref{prop:universal-surface-walls} guarantees that we have found a $K_{3t}$-minor controlled by $D_i$.
This would allow us to apply \zcref{prop:redClique} to get either of the first two outcomes we desire.
Therefore, this process must terminate.

\paragraph{Inductively refining a near embedding: Construction.}
Suppose for some $i\in[0,9t^2-1]$ the list of objects as above has already been constructed.

We begin by applying \zcref{thm:blanksocietyclassification} to $(G_i,\Omega_i)$.
If the outcomes of \zcref{thm:blanksocietyclassification} \textsl{i)}, \textsl{ii)}, or the first point of the outcome \textsl{v)} are reached, we are immediately done.
For the remaining options we present a case distinction.

\textbf{Case 1:}
Let $(G',\Omega')$ be the $C^i_{\mathsf{nest}(9t^2-i,t,r,k)-\mathsf{cost}_{\ref{thm:blanksocietyclassification}}(r,t,4(k+c_{\ref{prop:universal-surface-walls}}t^{12}+1))-1}$-society in $\rho_i$.
Case 1 contains two subcases that can be treated in almost identical ways.
Those are outcome \textsl{ii)} and outcome \textsl{iii)} of \zcref{thm:blanksocietyclassification} where both outcomes yield a set $A\subseteq V(G_i)$ of size at most $\mathsf{apex}^\mathsf{genus}_{\ref{thm:blanksocietyclassification}}(t)$ and a homogeneous, isolated transaction of order $\mathsf{transaction}(9t^2-i,t,r,k)$.
Independent of whether \textsl{ii)} or \textsl{iii)} is reached, if the transaction we receive is red instead we may apply \zcref{lemma:RedTransactionGivesRedMesh} and thus reach our third desired outcome, since we have
\begin{align*}
    \mathsf{transaction}(9t^2-i,t,r,k) &\geq 3r(r-1) - 2 \text{ and} \\
    \mathsf{nest}(9t^2-i,t,r,k)-\mathsf{cost}_{\ref{thm:blanksocietyclassification}}(r,t,4(k+c_{\ref{prop:universal-surface-walls}}t^{12}+1))-\mathsf{loss}_{\ref{thm:blanksocietyclassification}}(t) &\geq r+2 .
\end{align*}
Thus we know that the transaction we find is blank and isolated.
\smallskip

In the first of the two outcomes, $\mathcal{Q}$ is a blank, isolated crosscap transaction of order $\mathsf{transaction}(9t^2-i,t,r,k)$ in $(G'-A,\Omega')$ together with a nest $\mathcal{C}'$ in $\rho_i$ of order $\mathsf{nest}(9t^2-i,t,r,k)-\mathsf{cost}_{\ref{thm:blanksocietyclassification}}(r,t,4(k+c_{\ref{prop:universal-surface-walls}}t^{12}+1))-\mathsf{loss}_{\ref{thm:blanksocietyclassification}}(t)$ around $c_i$ to which $\mathcal{Q}$ is orthogonal.
\smallskip

In the second outcome, $\mathcal{Q}$ is a blank, isolated handle transaction of order $\mathsf{transaction}(9t^2-i,t,r,k)$ in $(G'-A,\Omega')$ together with a nest $\mathcal{C}'$ in $\rho_i$ of order $\mathsf{nest}(9t^2-i,t,r,k)-\mathsf{cost}_{\ref{thm:blanksocietyclassification}}(r,t,4(k+c_{\ref{prop:universal-surface-walls}}t^{12}+1))-\mathsf{loss}_{\ref{thm:blanksocietyclassification}}(t)$ around $c_i$ to which the two planar transactions $\mathcal{Q}_1$ and $\mathcal{Q}_2$ that make up $\mathcal{Q}$ are orthogonal.
\smallskip

We can now shave off a little off of the sides of the transaction in both cases as follows.
In the first outcome, we reduce the order of $\mathcal{Q}$ by $2\mathsf{radial}(9t^2-i-1,t,k)$ by shedding off the last $2\mathsf{radial}(9t^2-i+1,t,k)$ paths of $\mathcal{Q}$.
Alternatively, in the second outcome, we reduce the order of $\mathcal{Q}$ by $2\mathsf{radial}(9t^2-i-1,t,k)$ by removing from each $\mathcal{Q}_j$, $j\in[2]$, the last $\mathsf{radial}(9t^2-i-1,t,k)$ paths.
Let $\mathcal{Q}'$, $\mathcal{Q}'_1$, and $\mathcal{Q}'_2$, be the resulting transactions.
Moreover, notice that we may select a radial linkage $\mathcal{L}$ of order $\mathsf{radial}(9t^2-i-1,t,k)$ in $(G'-A,\Omega')$ that is orthogonal to $\mathcal{C}'$.
Indeed, we may select $\mathcal{Q}'$ and $\mathcal{L}$ such that there are segments $I_1$ and $I_2$ of $\Omega'$ where $\mathcal{Q}'$ has all endpoints in $V(I_1)$, $\mathcal{L}$ has its endpoints in $V(\Omega')$ in $V(I_2)$, and $I_1$ and $I_2$ are disjoint. 
It follows that $|\mathcal{Q}'|=4k+4c_{\ref{prop:universal-surface-walls}}t^{12}+14 + 4\mathsf{nest}(9t^2-i-1,t,r,k)$ and $|\mathcal{Q}_1'|=|\mathcal{Q}_2'|=2k+2c_{\ref{prop:universal-surface-walls}}t^{12}+7 + 2\mathsf{nest}(9t^2-i-1,t,r,k)$ in case $\mathcal{Q}'$ is a handle transaction.

In the following we discuss only how to proceed in the case where $\mathcal{Q}'$ is a crosscap transaction.
The second case, namely the one where $\mathcal{Q}'$ is a handle transactions, can be handled analogously with the exception that, instead of \zcref{lemma:integrate-crosscap} one needs to apply \zcref{lemma:integrate-handle} and some of the choices we make below need to be made for the transactions $\mathcal{Q}'_i$ and the objects derived from them individually.
Upon closer inspection, the reader will see that the size of $\mathcal{Q}'$ (and in particular of $\mathcal{Q}_1'$ and $\mathcal{Q}_2'$) was chosen to allow the application of any of these two lemmas.

We assume that $\mathcal{Q}'=\{ Q_1, \ldots  ,Q_{4k+4c_{\ref{prop:universal-surface-walls}}t^{12}+14 + 4\mathsf{nest}(9t^2-i-1,t,r,k)}\}$ are indexed naturally.
Let $a\coloneqq 2\mathsf{nest}(9t^2-i-1,t,r,k)+7$ and $b\coloneqq 2\mathsf{nest}(9t^2-i-1,t,r,k)+4k+4c_{\ref{prop:universal-surface-walls}}t^{12}+8$.
Moreover, let $\mathcal{A}\coloneqq \{ Q_1,\dots,Q_{a-1}\}$, $\mathcal{B}=\{ Q_{b+1},\dots,Q_{4k+4c_{\ref{prop:universal-surface-walls}}t^{12}+14 + 4\mathsf{nest}(9t^2-i-1,t,r,k)}\}$, and $\mathcal{I}\coloneqq \{ Q_a,\dots,Q_b\}$.
Thus $|\mathcal{I}| = 4k + 4c_{\ref{prop:universal-surface-walls}}t^{12} + 8 + 2\mathsf{nest}(9t^2-i-1,t,r,k)$ and $|\mathcal{A}| = |\mathcal{B}| = 2\mathsf{nest}(9t^2-i-1,t,r,k) + 6$.

We assume the cycles in $\mathcal{C}'$ to be numbered $C_1',\dots, C_{\mathsf{nest}(9t^2-i,t,r,k)-\mathsf{cost}_{\ref{thm:blanksocietyclassification}}(r,t,4(k+c_{\ref{prop:universal-surface-walls}}t^{12}+1))-\mathsf{loss}_{\ref{thm:blanksocietyclassification}}(t)}'$ from innermost to outermost.
Let $I=[\mathsf{nest}(9t^2-i,t,r,k)-\mathsf{cost}_{\ref{thm:blanksocietyclassification}}(r,t,4(k+c_{\ref{prop:universal-surface-walls}}t^{12}+1))-\mathsf{loss}_{\ref{thm:blanksocietyclassification}}(t)-3,\mathsf{nest}(9t^2-i,t,r,k)-\mathsf{cost}_{\ref{thm:blanksocietyclassification}}(r,t,4(k+c_{\ref{prop:universal-surface-walls}}t^{12}+1))-\mathsf{loss}_{\ref{thm:blanksocietyclassification}}(t)-8k-8c_{\ref{prop:universal-surface-walls}}t^{12}-\mathsf{radial}(9t^2-i-1,t,k)-9]\subseteq [\mathsf{nest}(9t^2-i,t,r,k)-\mathsf{cost}_{\ref{thm:blanksocietyclassification}}(r,t,4(k+c_{\ref{prop:universal-surface-walls}}t^{12}+1))-\mathsf{loss}_{\ref{thm:blanksocietyclassification}}(t)]$.
Then $|I|=8k+8c_{\ref{prop:universal-surface-walls}}t^{12}+\mathsf{radial}(9t^2-i-1,t,k)+6=2|\mathcal{I}|+|\mathcal{L}|+2$.
Observe that this means, in particular,
\begin{align*}
    2|\mathcal{I}|+|\mathcal{L}| =~& 8k+8c_{\ref{prop:universal-surface-walls}}t^{12} +4 + \mathsf{radial}(9t^2-i-1,t,k) \leq \mathsf{radial}(9t^2-i,t,k) = |\mathcal{R}_i|.
\end{align*}
Let us now select $\mathcal{I}'$ to be all $V(\Omega')$-$V(C_1')$-subpaths of the paths in $\mathcal{I}$.
It follows that $\mathcal{I}'$ is a radial linkage of order $2|\mathcal{I}|$.
Moreover, $\mathcal{I}'\cup\mathcal{L}$ is a radial linkage of order $8k+8c_{\ref{prop:universal-surface-walls}}t^{12}+\mathsf{radial}(9t^2-i-1,t,k)+4=|I|-2$ which is orthogonal to $\mathcal{C}'$.
This allows us to call \zcref{prop:connected_linkages} for $\mathcal{R}_i$, $\mathcal{I}'\cup\mathcal{L}$, $\mathcal{C}'$, and $I$.
As a result we obtain, in time $\mathbf{poly}(t+k)|E(G)|$, a radial linkage $\mathcal{R}'$ of order $|I|-2$ which shares its endpoints on $V(\Omega')$ with the endpoints of $\mathcal{R}_i$ and its endpoints on $C_1'$ with $\mathcal{I}'\cup\mathcal{L}$.
Moreover, $\mathcal{R}'$ is orthogonal to $\mathcal{C}'\setminus \{ C_i ~\colon~ i\in I\}$ and within the inner graph $G''$ of $C_{\mathsf{nest}(9t^2-i,t,r,k)-\mathsf{cost}_{\ref{thm:blanksocietyclassification}}(r,t,4(k+c_{\ref{prop:universal-surface-walls}}t^{12}+1))-\mathsf{loss}_{\ref{thm:blanksocietyclassification}}(t)-8k-8c_{\ref{prop:universal-surface-walls}}t^{12}-\mathsf{radial}(9t^2-i-1,t,k)-9}'$ it is disjoint from $\mathcal{Q}'\setminus\mathcal{I}$.
Let $\mathcal{J}\subseteq \mathcal{R}'$ be all those paths that do not meet $\mathcal{I}$ within $G''$.
Then $|\mathcal{J}|=|\mathcal{R}'|-8k-8c_{\ref{prop:universal-surface-walls}}t^{12}-4=\mathsf{radial}(9t^2-i-1,t,k)$.

Next we inspect our remaining nest a bit closer.
Let $\mathcal{C}''\subseteq \mathcal{C}$ be the set of the innermost $\mathsf{nest}(9t^2-i,t,r,k)-\mathsf{cost}_{\ref{thm:societyclassification}}(t,4(k+c_{\ref{prop:universal-surface-walls}}t^{12}+1))-\mathsf{loss}_{\ref{thm:societyclassification}}(t)-8k-8c_{\ref{prop:universal-surface-walls}}t^{12}-\mathsf{radial}(9t^2-i-1,t,k)-10$ cycles of $\mathcal{C}'$.
Let $\Omega''$ be the cyclic ordering of the ground vertices of $C'_{\mathsf{nest}(9t^2-i,t,r,k)-\mathsf{cost}_{\ref{thm:societyclassification}}(t,4(k+c_{\ref{prop:universal-surface-walls}}t^{12}+1))-\mathsf{loss}_{\ref{thm:societyclassification}}(t)-8k-8c_{\ref{prop:universal-surface-walls}}t^{12}-\mathsf{radial}(9t^2-i-1,t,k)-9}$ and let $(G'',\Omega'')$ be the resulting society with cylindrical rendition $\rho''$, which is the restriction of $\rho_i$ to the $c_i$-disk $\Delta'$ of $C'_{\mathsf{nest}(9t^2-i,t,r,k)-\mathsf{cost}_{\ref{thm:societyclassification}}(t,4(k+c_{\ref{prop:universal-surface-walls}}t^{12}+1))-\mathsf{loss}_{\ref{thm:societyclassification}}(t)-8k-8c_{\ref{prop:universal-surface-walls}}t^{12}-\mathsf{radial}(9t^2-i-1,t,k)-9}$ in $\rho_i$.
Moreover, let $\mathcal{J}'$ and $\mathcal{Q}''$ be the restrictions of $\mathcal{J}$ and $\mathcal{Q}'$ to $G''$.
We have
\begin{align*}
|\mathcal{C}''|=~&\mathsf{nest}(9t^2-i,t,r,k)-\mathsf{cost}_{\ref{thm:societyclassification}}(t,4(k+c_{\ref{prop:universal-surface-walls}}t^{12}+1))-\mathsf{loss}_{\ref{thm:societyclassification}}(t)-8k-8c_{\ref{prop:universal-surface-walls}}t^{12}\\
&-\mathsf{radial}(9t^2-i-1,t,k)-10\\
=~& (9t^2-i-1)\big(8k+8c_{\ref{prop:universal-surface-walls}}t^{12}+\mathsf{radial}(9t^2-i,t,k)+24\\
&+\mathsf{cost}_{\ref{thm:societyclassification}}(t,4(k+c_{\ref{prop:universal-surface-walls}}t^{12}+1))+\mathsf{loss}_{\ref{thm:societyclassification}}(t) \big)\\
&+ \mathsf{nest}_{\ref{thm:societyclassification}}(t,4(k+c_{\ref{prop:universal-surface-walls}}t^{12}+1)) +\nicefrac{4k}{2}(t-3)(t-4)+r+16\\
\geq~& (9t^2-i-1)\big(8k+8c_{\ref{prop:universal-surface-walls}}t^{12}+\mathsf{radial}(9t^2-i-1,t,k)+24\\
&+\mathsf{cost}_{\ref{thm:societyclassification}}(t,4(k+c_{\ref{prop:universal-surface-walls}}t^{12}+1))+\mathsf{loss}_{\ref{thm:societyclassification}}(t) \big)\\
&+ \mathsf{nest}_{\ref{thm:societyclassification}}(t,4(k+c_{\ref{prop:universal-surface-walls}}t^{12}+1))+\nicefrac{4k}{2}(t-3)(t-4)+r+16\\
=~& \mathsf{nest}(9t^2-i-1,t,r,k)+14
\end{align*}
and
\begin{align*}
|\mathcal{Q}''| =~& 4k+4c_{\ref{prop:universal-surface-walls}}t^{12}+14+4\mathsf{nest}(9t^2-i-1,t,r,k)\\
=~& 2 (2k+2c_{\ref{prop:universal-surface-walls}}t^{12})+4\mathsf{nest}(9t^2-i-1,t,r,k)+14 .
\end{align*}
Notice that the calculations above suffice for the application of \zcref{lemma:integrate-crosscap}.
Indeed, we exceed the necessary numbers by more than $2\mathsf{nest}(9t^2-i-1,t,r,k)$.
If instead we were dealing with a handle transaction we would next want to apply \zcref{lemma:integrate-handle}.
Here we would have constructed two transactions, namely $\mathcal{Q}''_1$ and $\mathcal{Q}''_2$, each of order at least $2k+2c_{\ref{prop:universal-surface-walls}}t^{12}+2\mathsf{nest}(9t^2-i-1,t,r,k)+7$ and thus, together, they form a transaction $\mathcal{Q}''$ of the order above.
This explains our choice for the size of $\mathcal{Q}''$.

So we are now ready to apply \zcref{lemma:integrate-crosscap} to $(G'',\Omega'')$, $\mathcal{C}''$, $\mathcal{J}'$, and $\mathcal{Q}''$ with $p=4k+4c_{\ref{prop:universal-surface-walls}}t^{12}+2$, rendition $\rho''$ and the disk $\Delta'$.
As a result we obtain the following list of objects.
Let $\Sigma^*$ be a surface homeomorphic to the projective plane minus an open disk which is obtained from $\Delta'$ by adding a crosscap to the interior of $c_i$.
Then there exists $\mathcal{I}'\subseteq \mathcal{Q}''$ which is exactly the restriction of $\mathcal{I}$ to $G''$, and a blank rendition $\rho'''$ of $(G'',\Omega'')$ in $\Sigma^*$ with a unique vortex $c_{i+1}'$ and the following properties hold:
\begin{itemize}
    \item $\mathcal{I}'$ is disjoint from $\sigma(c_{i+1}')$,
    \item the vortex society $(G_{i+1},\Omega_{i+1})$ of $c_{i+1}'$ in $\rho'''$ has a cylindrical rendition $\rho_{i+1}$ with nest $\mathcal{C}_{i+1}=\{ C_1^{i+1},\dots,C^{i+1}_{\mathsf{nest}(9t^2-i-1,t,r,k)} \}$ around the unique vortex $c_{i+1}$,
    \item every element of $\mathcal{J}'$ has an endpoint in $V(\sigma_{\rho_{i+1}}(c_{i+1}))$, and
    \item $\mathcal{J}'$ is orthogonal to $\mathcal{C}_{i+1}$. Moreover,
    \item let $\mathcal{J}'=\{J_1,\dots, J_{\mathsf{radial}(9t^2-i-1,t,k)}\}$.
    For each $j\in[\mathsf{radial}(9t^2-i-1,t,k)]$ let $x_j$ be the endpoint of $J_j$ in $V(\Omega'')$ and let $y_j$ be the endpoint of $J_j$ on $c_{i+1}'$; then if $x_1,x_2,\dots,x_{\mathsf{radial}(9t^2-i-1,t,k)}$ appear on $\Omega''$ in the order listed, then $y_1,y_2,\dots,y_{\mathsf{radial}(9t^2-i-1,t,k)}$ appear on $N_{\rho'''}(c_{i+1}')$ in the order listed.
    \item Finally, let $\Delta''$ be the open disk bounded by the trace of the outermost cycle of $\mathcal{C}''$ in $\rho''$.
    Then $\rho''$ restricted to $\Delta'\setminus \Delta''$ is equal to $\rho'''$ restricted to $\Delta'\setminus\Delta''$.
\end{itemize}
Now let $\delta_{i+1}$ be obtained by first unifying the renditions $\rho'''$ and $\rho_{i+1}$ along the vortex society of $c_{i+1}'$, then combining the resulting rendition of $(G'',\Omega'')$ in the disk $\Delta'$ with the rendition $\rho''$ along to boundary of $\Delta''$, and then reintegrating $\rho''$ into $\delta_i$.
Moreover, let $\Sigma_{i+1}$ be obtained from $\Sigma_i$ by removing the interior of $\Delta'$ and replacing it with $\Sigma^*$.
Notice that this means that $\Sigma_{i+1}$ is obtained from $\Sigma_i$ by adding a single crosscap.
We also set $A_{i+1}\coloneqq A_i\cup A$ and obtain 
\begin{align*}
|A_{i+1}|\leq~& i\cdot\mathsf{apex}^\mathsf{genus}_{\ref{thm:blanksocietyclassification}}(t)+16t^3+\mathsf{apex}^\mathsf{genus}_{\ref{thm:blanksocietyclassification}}(t)\\
\leq ~& (i+1)\mathsf{apex}^\mathsf{genus}_{\ref{thm:blanksocietyclassification}}(t)+16t^3.
\end{align*}
So the first three points of our invariant are maintained.
\smallskip

Notice that, by construction, we have that the radial linkage $\mathcal{R}'$ can be extended onto the restriction of $\mathcal{R}_i$ to the proper outer graph of the closed curve obtained by following along the vertices on $\Omega'$.
This allows us to, firstly, extend the crosscap transaction $\mathcal{I}$ of order $4k+4c_{\ref{prop:universal-surface-walls}}t^{12}+2$ along $\mathcal{R}_i$ to obtain a crosscap transaction $\mathcal{P}_{i+1}$ on $(G_0,\Omega_0)$ whose paths are disjoint from all paths in $\mathcal{P}_j$, $j\in[i]$.
Secondly, we can extend $\mathcal{R}'$ along $\mathcal{R}_i$ to form the radial linkage $\mathcal{R}_{i+1}$ such that it is orthogonal to both $\mathcal{C}^*$ and $\mathcal{C}_{i+1}$ and satisfies
\begin{align*}
 |\mathcal{R}_{i+1}| = \mathsf{radial}(t^2-i-1,t,k),
\end{align*}
its endpoints coincide with some of the endpoints of $\mathcal{R}_0$, and it is disjoint from the paths in $\mathfrak{P}_{i+1}=\mathfrak{P}_i \cup \{ \mathcal{P}_{i+1} \}$.
Indeed, it is straightforward to see that $(G_0-A_{i+1},\Omega_0,\mathcal{C}^*,\,\mathcal{R}_{i+1},\mathfrak{P}_{i+1})$ is a $\Sigma_{i+1}$-configuration of strength $(4(k+c_{\ref{prop:universal-surface-walls}}t^{12}+1),\mathsf{radial}(t^2-i,t,k),p_1,\dots,p_{i+1})$ with $p_j=4k+4c_{\ref{prop:universal-surface-walls}}t^{12}$.
So the last three points of our invariant are also satisfied.
\smallskip

By \zcref{obs:surface-configs-to-walls} the $\Sigma_{i+1}$-configuration found in the previous paragraph yields the existence of a $(4(k+c_{\ref{prop:universal-surface-walls}}t^{12}+1))$-surface wall $D_{i+1}$ with the amount of crosscaps and handles used to obtain $\Sigma_{i+1}$ from the sphere such that the base cycles of $D_{i+1}$ coincide with the cycles of $\mathcal{C}^*$.
This, together with the society $(G_{i+1},\Omega_{i+1})$, the cylindrical rendition $\rho_{i+1}$, and the nest $\mathcal{C}_{i+1}$ around the unique vortex $c_{i+1}$ of both $\rho_{i+1}$ and $\delta_{i+1}$, satisfies the remaining three points of our invariant and thus this step is complete.
\smallskip

Observe that, in the case where $i=9t^2-1$ we have reached a situation where $K_{3t}$ embeds in $\Sigma_{i+1}$.
This means that \zcref{prop:universal-surface-walls} implies the existence of a $K_{3t}$-model controlled by $D_{i+1}$.
So in this case we would be done via \zcref{prop:redClique}.
Hence, we may assume that $i+1\in[9t^2-1]$.
\bigskip

\textbf{Case 2:}
The remaining case to be discussed is the second part of outcome \textsl{v)} of \zcref{thm:blanksocietyclassification}.
Here we are given a set $A \subseteq V(G_i)$ with $|A| \leq \mathsf{apex}^\mathsf{fin}_{\ref{thm:blanksocietyclassification}}(r,t,4(k+c_{\ref{prop:universal-surface-walls}}t^{12}+1),\mathsf{transaction}(9t^2,t,r,k))$, a rendition $\rho'$ of $(G_i - A, \Omega_i)$ in $\Delta_i$ with breadth $b \in [\nicefrac{3}{2}(t-1)(3t-4) + r(r-1)-3]$ and depth at most $\mathsf{depth}_{\ref{thm:blanksocietyclassification}}(r,t,4(k+c_{\ref{prop:universal-surface-walls}}t^{12}+1),\mathsf{transaction}(9t^2-i,t,r,k))$, and an extended $(4(k+c_{\ref{prop:universal-surface-walls}}t^{12}+1))$-surface-wall $D$ with signature $(0,0,b)$, such that $D$ is grounded in $\rho'$, the base cycles of $D$ are the cycles $C_{\mathsf{nest}(9t^2-i,t,r,k)-\mathsf{cost}_{\ref{thm:blanksocietyclassification}}(r,t,4(k+c_{\ref{prop:universal-surface-walls}}t^{12}+1))-1-k-c_{\ref{prop:universal-surface-walls}}t^{12}}^i,\dots,C_{\mathsf{nest}(9t^2-i,t,r,k)-\mathsf{cost}_{\ref{thm:blanksocietyclassification}}(r,t,4(k+c_{\ref{prop:universal-surface-walls}}t^{12}+1))-1}^i$, and there exists a bijection between the vortices $v$ of $\rho'$ and the vortex segments $S_v$ of $D$, where $v$ is the unique vortex contained in the disk $\Delta_{S_v}$ defined by the trace of the inner cycle of the nest of $S_v$ where $\Delta_{S_v}$ is chosen to avoid the trace of the simple cycle of $D$.

At this point the additional infrastructure on the order of $c_{\ref{prop:universal-surface-walls}}t^{12}$ has served its purpose and we can concentrate on just building our extended $k$-surface wall.
For this purpose, we must now combine the radial linkage $\mathcal{R}_i$ we found in our induction with the rails we have for the vortex segments in $D$.
Notice that the union of all rails of $D$ form a radial linkage $\mathcal{R}''$ of order $16b(k+c_{\ref{prop:universal-surface-walls}}t^{12})$ such that every vortex segment provides exactly $16(k+c_{\ref{prop:universal-surface-walls}}t^{12})$ of these paths.
Moreover, there are $4bk$ cycles among $C_{\mathsf{nest}(t^2-i,t,r,k)-\mathsf{cost}_{\ref{thm:blanksocietyclassification}}(r,t,4(k+c_{\ref{prop:universal-surface-walls}}t^{12}+1))-1-4bk}^i,\dots,C_{\mathsf{nest}(t^2-i,t,r,k)-\mathsf{cost}_{\ref{thm:blanksocietyclassification}}(r,t,4(k+c_{\ref{prop:universal-surface-walls}}t^{12}+1))-1}^i$.
Now let $\mathcal{R}'$ be a radial linkage of order $4bk$ formed by selecting $4k$ rails from each vortex segment of $D$.
By definition of $D$, for every $j \in [\mathsf{nest}(t^2-i,t,r,k)-\mathsf{cost}_{\ref{thm:blanksocietyclassification}}(r,t,4(k+c_{\ref{prop:universal-surface-walls}}t^{12}+1))-1-4bk, \mathsf{nest}(t^2-i,t,r,k)-\mathsf{cost}_{\ref{thm:blanksocietyclassification}}(r,t,4(k+c_{\ref{prop:universal-surface-walls}}t^{12}+1))-1]$ and every vortex $v$ of $\rho$, both the nest of $S_v$ and $v$ itself are disjoint from $C^i_j$.

Let $I = [\mathsf{nest}(t^2-i,t,r,k)-\mathsf{cost}_{\ref{thm:blanksocietyclassification}}(r,t,4(k+c_{\ref{prop:universal-surface-walls}}t^{12}+1))-1-4bk+1,\mathsf{nest}(t^2-i,t,r,k)-\mathsf{cost}_{\ref{thm:blanksocietyclassification}}(r,t,4(k+c_{\ref{prop:universal-surface-walls}}t^{12}+1))-1]$.
Then we have $|I| = 4bk+2$.
Therefore, we may call upon \zcref{prop:connected_linkages} for $\mathcal{C}'$, $I$, $\mathcal{R}'$ and $\mathcal{R}_i$ to obtain a radial linkage $\mathcal{R}_{i+1}$ of order $4bk$ whose endpoints on the outermost cycles of the nests of the $S_v$ coincide with the vertices of the rails of $S_v$ for every vortex segment $S_v$ of $\rho'$ and whose other endpoints are a subset of the endpoints of $\mathcal{R}_i$ on $V(\Omega_0)$.
Moreover, the endpoints of the paths among $\mathcal{R}_i$ that lead to $S_v$ appear consecutively on $\Omega_0$.
Hence, we may obtain an extended $k$-surface wall $D_{i+1}$ from $D_i$ by discarding some of the cycles and paths in each of the handle and crosscap segments and integrating the nests of the $S_v$ along the paths in $\mathcal{R}_{i+1}$.
We may also obtain a $\Sigma_i$-rendition $\delta_{i+1}$ of breadth at most $b$ and depth at most
\begin{align*}
    \mathsf{depth}_{\ref{thm:blanksocietyclassification}}(r,t,4(k+c_{\ref{prop:universal-surface-walls}}t^{12}+1),&~\mathsf{transaction}(t^2-i,t,r,k)) \leq \\
    & \quad \mathsf{depth}_{\ref{thm:blanksocietyclassification}}(r,t,4(k+c_{\ref{prop:universal-surface-walls}}t^{12}+1),\mathsf{transaction}(t^2,t,r,k))
\end{align*}
for $G-A_{i+1}$, where $A_{i+1\coloneqq A_i\cup A}$ with
\begin{align*}
    |A_{i+1}| \leq~& |A_i| + \mathsf{apex}^\mathsf{fin}_{\ref{thm:blanksocietyclassification}}(r,t,4(k+c_{\ref{prop:universal-surface-walls}}t^{12}+1),\mathsf{transaction}(t^2,t,r,k))\\
    \leq~& i\cdot \mathsf{apex}^\mathsf{genus}_{\ref{thm:blanksocietyclassification}}(t) + \mathsf{apex}^\mathsf{fin}_{\ref{thm:blanksocietyclassification}}(r,t,4(k+c_{\ref{prop:universal-surface-walls}}t^{12}+1),\mathsf{transaction}(t^2,t,r,k)) + 16t^3.
\end{align*}
It follows, in particular from the choice of the size of $M_0$, that $\Lambda=(A_{i+1},M_0,D_{i+1},\delta_{i+1})$ is a blank $k$-$(\mathsf{apex}_{\ref{thm:strongest_localstructure}}(t^2,t,k),\nicefrac{3}{2}(t-1)(3t-4) + r(r-1)-3,\mathsf{depth}_{\ref{thm:strongest_localstructure}}(t,k),r)$-$\Sigma_i$-layout centred at $M$ as desired.
With this, our proof is complete.
\end{proof}

\section{A polynomial grid theorem for bidimensionality}\label{sec:localtoglobal}

In this final section we demonstrate how we can use \zcref{thm:strongest_localstructure} from the previous section to obtain polynomial bounds for bidimensionality.
We begin with some necessary definitions that will facilitate our induction.

\paragraph{Highly linked sets.}
Let $\alpha \in [2/3, 1)_{\mathbb{R}}$.
Moreover, let $G$ be a graph and $X \subseteq V(G)$ be a vertex set. 
A set $S \subseteq V(G)$ is said to be an \emph{$\alpha$-balanced separator} for $X$ if for every component $C$ of $G - S$ it holds that $|V(C) \cap X| \leq \alpha|X|$. 
Let $k$ be a non-negative integer.
We say that $X$ is a \emph{$(k, \alpha)$-linked set} of $G$ if there is no $\alpha$-balanced separator of size at most $k$ for $X$ in $G$.

Given a $(3k, \alpha)$-linked set $X$ of $G$ we define $$\mathcal{T}_{X} \coloneqq \{ (A, B) \in \mathcal{S}_{k+1}(G) ~\!\colon\!~ |X \cap B| > \alpha|X| \}.$$ 
It is not hard to see that $\mathcal{T}_{S}$ is a tangle of order $k+1$ in $G$.

We need an algorithmic way to find, given a highly linked set, a large wall whose tangle is a truncation of the tangle induced by the highly linked set.
This is done in \cite{ThilikosW2024Excluding} by algorithmatising a proof of Kawarabayashi et al.\ from \cite{KawarabayashiTW2021Quickly}. 

\begin{proposition}[Thilikos and Wiederrecht \cite{ThilikosW2024Excluding} (see Theorem 4.2.)]\label{prop:AlgoGrid}
Let $k\geq 3$ be an integer and $\alpha\in [2/3,1)$.
There exist universal constants $c_1, c_2\in\mathbb{N}\setminus\{ 0\}$, and an algorithm that, given a graph $G$ and a $(c_1k^{20},\alpha)$-linked set $X\subseteq V(G)$ computes in time $2^{\mathbf{O}(k^{c_2})}|V(G)|^2|E(G)|\log(|V(G)|)$ a $k$-wall $W\subseteq G$ such that $\mathcal{T}_W$ is a truncation of $\mathcal{T}_X$.
\end{proposition}

Additionally, we need to be able to find balanced separators efficiently.

\begin{proposition}[Reed \cite{Reed1992Finding}]\label{prop:FindSep}
    There exists an algorithm that takes as input an integer $k$, a graph $G$, and a set $X \subseteq V(G)$ of size at most $3k+1$, and finds, in time $2^{\mathbf{O}(k)}m$, either a $\nicefrac{2}{3}$-balanced separator of size at most $k$ for $X$ or correctly determines that $X$ is $(k,\nicefrac{2}{3})$-linked in $G$.
\end{proposition}

We require one last ingredient related to the vortices of $\Sigma$-renditions.
The definition of depth for societies we utilized so far was particularly well suited with the way we refined societies in our proof.
However, here we need a decomposition based analogue of this definition that will permit to inductively extend a vortex of bounded depth further into the tree-decomposition via adhesions depending on the depth of the vortex itself.

\paragraph{Linear decompositions of vortices.}
Let $(G,\Omega)$ be a society.
A \emph{linear decomposition} of $(G,\Omega)$ is a labelling $v_1,v_2,\dots,v_n$ of $V(\Omega)$ such that $v_1,v_2,\dots,v_n$ appear in $\Omega$ in the order listed, together with sets $(X_1,X_2,\dots,X_n)$ such that
\begin{enumerate}
    \item $X_i\subseteq V(G)$ and $v_i\in X_i$ for all $i\in[n]$,
    \item $\bigcup_{i\in[n]}X_i=V(G)$ and for every $uv\in E(G)$ there exists $i\in[n]$ such that $u,v\in X_i$, and
    \item for every $x\in V(G)$ the set $\{ i\in[n] ~\!\colon\!~ x\in X_i \}$ forms an interval in $[n]$.
\end{enumerate}
The \emph{adhesion} of a linear decomposition is $\max \{ |X_i\cap X_{i+1}| ~\!\colon\!~ i\in[n-1] \}.$
The \emph{width} of a linear decomposition is $\max \{ |X_i| ~\!\colon\!~ i\in[n] \}.$

It easily follows that every society with a linear decomposition of adhesion at most $k$ has depth at most $2k.$
What interests us here is the reverse statement which was shown by Robertson and Seymour in \cite{RobertsonS1990Graphb}.

\begin{proposition}[\cite{RobertsonS1990Graphb}]\label{prop:DepthToLinearDec} Let $k$ be a non-negative integer and $(G, \Omega)$ be a society of depth at most $k.$
Then $(G, \Omega)$ has a linear decomposition of adhesion at most $k.$
\end{proposition}

Finally, we define the \emph{width} of a $\Sigma$-rendition of a graph $G$ as the minimum non-negative integer $w$ such that every vortex society of $\rho$ admits a linear decomposition of width at most $w.$

\paragraph{Local to global.}
We are now ready for the proof of \zcref{thm:BidimensionalityIntro}.
In fact we prove a stronger statement and \zcref{thm:BidimensionalityIntro} will follow as a corollary.
The statement reads as follows.

\begin{theorem}\label{thm:Bidimensionality}
There exist functions $\mathsf{link}_{\ref{thm:Bidimensionality}}$, $\mathsf{apex}_{\ref{thm:Bidimensionality}}$, $\mathsf{breadth}_{\ref{thm:Bidimensionality}}$, $\mathsf{width}_{\ref{thm:Bidimensionality}}\colon\mathbb{N} \to \mathbb{N}$ such that for every integer $k \geq 4$ and every annotated graph $(G, R)$, one of the following holds:
\begin{enumerate}
    \item $(G, R)$ contains $\mathsf{R}_{k}$ as a red minor, or
    \item $(G, R)$ has a tree-decomposition $(T, \beta)$ of adhesion at most $\mathsf{link}_{\ref{thm:Bidimensionality}}(k)$ and a (possibly empty) subset $L \subseteq V(T)$ of leaves of $T$ such that, for all $t \in V(T)$ one the following holds:
    \begin{itemize}
        \item $t \in L$ is a leaf with unique neighbour $d$ such that $\beta(t) \cap R \subseteq \beta(d)$, or
        \item if $(G_{t}, R_{t})$ is the annotated torso of $(G, R)$ at $t$, then there exists a set $A \subseteq V(G_t)$ of size at most $\mathsf{apex}_{\ref{thm:Bidimensionality}}(k)$ and a surface $\Sigma$ of Euler-genus less than $9 k^{4}$, such that $G_t - A$ has a $\Sigma$-rendition of breadth at most $\mathbf{O}(k^{4})$ and width at most $\mathsf{width}_{\ref{thm:Bidimensionality}}(k)$ such that $\widetilde{c} = V(\sigma(c))$ for all non-vortex cells $c\in C(\rho)$, and for every vertex $u \in R_{t} \setminus A$ there exists some vortex $v \in C(\rho)$ such that $u \in V(\sigma(c)) \setminus \tilde{v}$.
    \end{itemize}
\end{enumerate}
Moreover,
\begin{align*}
\mathsf{link}_{\ref{thm:Bidimensionality}}(k) \ &\in \ \mathbf{O}(k^{226713760}),\\ \mathsf{apex}_{\ref{thm:Bidimensionality}}(k), \mathsf{width}_{\ref{thm:Bidimensionality}}(k) \ &\in \ \mathbf{O}(k^{5667904})\text{, and}\\
\mathsf{breadth}_{\ref{thm:Bidimensionality}}(k) \ &\in \ \mathbf{O}(k^{4}).
\end{align*}
There also exists an algorithm that finds one of these two outcomes in time $2^{k^{\mathbf{O(1)}}} \cdot |E(G)|^{3} |V(G)| \log(|V(G)|).$
\end{theorem}

We actually prove an even stronger version that will facilitate our induction.
This is common for the local-to-global step.

\begin{theorem}\label{thm:BidimensionalityInduction}
There exist functions $\mathsf{link}_{\ref{thm:BidimensionalityInduction}}$, $\mathsf{apex}_{\ref{thm:BidimensionalityInduction}}$, $\mathsf{breadth}_{\ref{thm:BidimensionalityInduction}}$, $\mathsf{width}_{\ref{thm:BidimensionalityInduction}} \colon \mathbb{N} \to \mathbb{N}$ such that for every integer $k \geq 4$ and every annotated graph $(G, R)$, and every set $X \subseteq V(G)$ with $|X|\leq 3\mathsf{link}_{\ref{thm:BidimensionalityInduction}}(k) + 1$ one of the following holds:
\begin{enumerate}
    \item $(G, R)$ contains $\mathsf{R}_{k}$ as a red minor, or
    \item $(G, R)$ has a rooted tree-decomposition $(T, r, \beta)$ of adhesion at most $3\mathsf{link}_{\ref{thm:BidimensionalityInduction}}(k) + \mathsf{apex}_{\ref{thm:BidimensionalityInduction}}(k) + \mathsf{width}_{\ref{thm:BidimensionalityInduction}}(k) + 3$ and a (possibly empty) subset $L \subseteq V(T)$ of leaves of $T$ such that, $X \subseteq \beta(r)$ and for all $t \in V(T)$ one the following holds:
    \begin{itemize}
        \item $t \in L$ is a leaf with a unique neighbour $d$ such that $\beta(t) \cap R \subseteq \beta(d)$, or
        \item if $(G_t, R_t)$ is the annotated torso of $(G, R)$ at $t$, then there exists a set $A \subseteq V(G_t)$ of size at most $3\mathsf{link}_{\ref{thm:BidimensionalityInduction}}(k) + \mathsf{apex}_{\ref{thm:BidimensionalityInduction}}(k) + 1$ and a surface $\Sigma$ of Euler-genus less than $9k^4$ such that
        \begin{enumerate}
            \item $G_t - A$ has a $\Sigma$-rendition of breadth at most $\mathsf{breadth}_{\ref{thm:BidimensionalityInduction}}(k)$ and width at most $2\mathsf{link}_{\ref{thm:BidimensionalityInduction}}(k) + \mathsf{width}_{\ref{thm:BidimensionalityInduction}}(k) + 1$ such that $\tilde{c} = V(\sigma(c))$ for all non-vortex cells $c\in C(\rho)$, and
            for every vertex $u \in R_{t} \setminus A$, there exists some vortex $v \in C(\rho)$ such that $u \in V(\sigma(c)) \setminus \tilde{v}$, and
            \item if $d$ is a child of $t$ with $\beta(d) \cap \beta(t) \subseteq A \cup \widetilde{c}$ for some non-vortex cell $c\in C(\rho)$, then $d\in L$.
        \end{enumerate}
    \end{itemize}
\end{enumerate}
Moreover,
\begin{align*}
\mathsf{link}_{\ref{thm:BidimensionalityInduction}}(k) \ &\in \ \mathbf{O}(k^{226713760}),\\ \mathsf{apex}_{\ref{thm:BidimensionalityInduction}}(k), \mathsf{width}_{\ref{thm:BidimensionalityInduction}}(k) \ &\in \ \mathbf{O}(k^{5667904})\text{, and}\\
\mathsf{breadth}_{\ref{thm:BidimensionalityInduction}}(k) \ &\in \ \mathbf{O}(k^{4}).
\end{align*}
There also exists an algorithm that finds one of these two outcomes in time $2^{k^{\mathbf{O}(1)}} \cdot |E(G)|^{3} |V(G)| \log(|V(G)|).$
\end{theorem}
\begin{proof} We begin with defining the functions involved.
Let $c_{1}$ be the constant from \zcref{prop:AlgoGrid}.
Then
\begin{align*}
\mathsf{apex}_{\ref{thm:BidimensionalityInduction}}(k) &\coloneqq \mathsf{apex}_{\ref{thm:strongest_localstructure}}(3k - 1, k^{2}, 4),\\
\mathsf{width}_{\ref{thm:BidimensionalityInduction}}(k) &\coloneqq 2\mathsf{depth}_{\ref{thm:strongest_localstructure}}(3k - 1, k^{2}, 4),\\
\mathsf{breadth}_{\ref{thm:BidimensionalityInduction}}(k) &\coloneqq \nicefrac{3}{2}(k^{2}(3k^{2} - 1) - 6k + 10)\text{, and}\\
\mathsf{link}_{\ref{thm:BidimensionalityInduction}}(k) &\coloneqq c_{1} \mathsf{mesh}_{\ref{thm:strongest_localstructure}}(3k - 1, k^{2}, 4, 3)^{20} + k^{2} - 1.
\end{align*}

Let us assume that $(G, R)$ does not contain $\mathsf{R}_{k}$ as a red minor.
We proceed by induction on $|V(G) \setminus X|$ and start by discussing two principal cases.

\textbf{Principal case 1:} Assume that $|V(G)| < 3 \mathsf{link}_{\ref{thm:BidimensionalityInduction}}(k) + 1.$
Then, we can define $T$ to consist of a single vertex $r$ and set $\beta(r) \coloneqq V(G)$ which trivially satisfies the second outcome of the theorem.
Therefore we may assume that $|V(G)| \geq 3 \mathsf{link}_{\ref{thm:BidimensionalityInduction}}(k) + 1.$

\textbf{Principal case 2:} Assume that $|X| < 3 \mathsf{link}_{\ref{thm:BidimensionalityInduction}}(k) + 1.$
In this case we can choose any vertex $u \in V(G) \setminus X$ and set $X' \coloneqq X \cup \{ u \},$ thereby achieving $|V(G) \setminus X'| < |V(G) \setminus X|.$
Then, we conclude by induction.
Therefore, we may also assume that $|X| = 3\mathsf{link}_{\ref{thm:BidimensionalityInduction}}(k) + 1.$

We now call \zcref{prop:FindSep} with $X$ and $s = \mathsf{link}_{\ref{thm:BidimensionalityInduction}}(k)$ which has two possible outcomes.
Either $X$ is $(s, \nicefrac{2}{3})$-linked, or there is a $\nicefrac{2}{3}$-balanced separator of size at most $s$ for $X.$
We treat these two cases separately.

\textbf{Case 1:} There exists a $\nicefrac{2}{3}$-balanced separator $S$ of size at most $s$ for $X.$

In this case, for each component $H$ of $G - S,$ we consider the graph $H' \coloneqq G[V(H) \cup S]$ together with the set $X_{H} \coloneqq S \cup (V(H) \cap X).$
Notice that, since $S$ is a $\nicefrac{2}{3}$-balanced separator for $X$ of size at most $s,$ we have that $$|X_{H}| \leq s + \bigg\lfloor \frac{2}{3} (3s + 1) \bigg\rfloor \leq 3s < 3s + 1.$$
Hence, $H'$ and $X_{H}$ satisfy the properties of one of the two principal cases.
In either case, we obtain a rooted tree-decomposition $(T_{H}, r_{H}, \beta_{H})$ together with a set of leaves $L_{H} \subseteq V(T_{H})$ for $H'$ with $X_{H} \subseteq \beta_{H}(r_{H}).$
We define $(T, r, \beta)$ as follows.
The tree $T$ is the tree obtained from the disjoint union of all $T_{H}$ by introducing a new vertex $r$ adjacent to all $r_{H}.$
The bags of $T$ are set to be $\beta(t) \coloneqq \beta_{H}(t)$ in case $t \in V(H)$ for some component $H$ of $G - S$ and otherwise we have $t = r$ and we set $\beta(r) \coloneqq X \cup S.$
Moreover, we set $L$ to be the union of all $L_{H}.$
Notice that $(T, r, \beta)$ is indeed a rooted tree-decomposition as required and, in particular, $|\beta(r)| \leq 4s + 1$ and so $G_{r}$ satisfies the second bullet point of the second outcome of our claim.

\textbf{Case 2:} The set $X$ is $(s, \nicefrac{2}{3})$-linked.

Since $s \geq c_{1} \mathsf{mesh}_{\ref{thm:strongest_localstructure}}(3k - 1, k^{2}, 2, 3)^{20},$ we can apply \zcref{prop:AlgoGrid} to $X$ and obtain a $\mathsf{mesh}_{\ref{thm:strongest_localstructure}}(3k - 1, k^{2}, 4, 3)$-wall $W$ such that $\mathcal{T}_{W} \subseteq \mathcal{T}_{X}.$
We may now call upon \zcref{thm:strongest_localstructure}.
There are $4$ possible outcomes:
\begin{enumerate}
    \item there exists a separation $(I, J) \in \mathcal{T}_W$ of order at most $k^{2} - 1$ such that $(J \setminus I) \cap R = \emptyset,$
    \item $(G, R)$ has a red $K_{k^{2}}$-minor controlled by $W,$
    \item $(G, R)$ has a red, flat $(3k - 1)$-mesh $M$ such that $\mathcal{T}_{M} \subseteq \mathcal{T}_{W},$ or
    \item $(G, R)$ has a blank $4$-$(\mathsf{apex}_{\ref{thm:strongest_localstructure}}(3k - 1, k^{2}, 4), \mathsf{breadth}_{\ref{thm:strongest_localstructure}}(3k - 1, k^{2}), \mathsf{depth}_{\ref{thm:strongest_localstructure}}(3k - 1, k^{2}, 4), 3)$-$\Sigma$-layout $\Lambda$ centred at $W$ and the surface $\Sigma$ has genus less than $9k^4.$
\end{enumerate}
We further distinguish subcases based on the outcome above.

\textbf{Case 2.i:} There exists a separation $(I, J) \in \mathcal{T}_W$ of order at most $k^{2} - 1$ such that $(J \setminus I) \cap R = \emptyset.$

Let $S = I \cap J$ with $|S| \leq k^{2} - 1 \leq s.$
By definition of tangles, there exists a unique component $H_{0}$ of $G - S$ such that $(V(G) \setminus V(H_{0}), V(H_{0}) \cup S) \in \mathcal{T}_{W}.$
Call $H_{0}$ the \emph{$\mathcal{T}_{W}$-big component} of $G - S.$
In fact, it must be that $V(H_{0}) \subseteq J \setminus I,$ as otherwise we contradict the definition of a tangle.
Therefore $V(H_{0}) \cap R = \emptyset.$

Now, let $H$ be any component of $G - S$ different than $H_{0}$ and let $X_{H} \coloneqq S \cup (V(H) \cap X).$
Since $H \neq H_{0},$ $H_{0}$ is the $\mathcal{T}_{W}$-big component of $G - S,$ and $\mathcal{T}_{W} \subseteq \mathcal{T}_{X},$ we have that $$|X_{H}| \leq s + \bigg\lfloor \frac{2}{3} (3s + 1) \bigg\rfloor \leq 3s < 3s + 1.$$
Hence, for each component $H$ of $G - S$ other than $H_{0},$ $H$ together with $X_{H}$ falls into one of the two principal cases.
Therefore, for each such $H,$ we obtain a rooted tree-decomposition $(T_{H}, r_{H}, \beta_{H})$ and a set $L_{H}$ of leaves of $T_{H}$ with $X_{H} \subseteq \beta_{H}(r_{H})$ that satisfies the second outcome of our claim.

We define the desired rooted tree-decomposition $(T, r, \beta)$ with the set $L$ of leaves as follows.
Let $T$ be the tree obtained from the disjoint union of the $T_{H}$ for all components $H \neq H_{0}$ of $G - S$ by introducing a new isolated vertex $d$ and then a new vertex $r$ adjacent to $d$ and all $r_{H}.$
We set $\beta(d) \coloneqq V(H_{0}) \cup X \cup S,$ $\beta(r) \coloneqq X \cup S,$ and $\beta(t) \coloneqq \beta_{H}(t)$ for all $t \in V(T)$ such that there is a component $H \neq H_{0}$ of $G - S$ with $t \in V(T_{H}).$
Finally, we set $L$ to be the union of all $L_{H}$ and the set $\{ d \}.$
It follows that $(T, r, \beta)$ together with $L$ satisfies all necessary conditions.

\textbf{Case 2.ii:} $(G, R)$ contains a red $K_{k^{2}}$-minor controlled by $W.$

In this case we immediately conclude with a $\mathsf{R}_{k}$-minor in $(G, R),$ as a red $K_{k^{2}}$ contains any annotated graph on $k^{2}$ vertices.

\textbf{Case 2.iii:} $(G, R)$ has a red, flat $(3k - 1)$-mesh $M$ such that $\mathcal{T}_{M} \subseteq \mathcal{T}_{W}.$

Here we conclude by applying \zcref{lemma:RedMeshToRedGrid} to $M$ which gives us the desired $\mathsf{R}_{k}$-minor in $(G, R).$

\textbf{Case 2.iv:} There exists a blank $4$-$(\mathsf{apex}_{\ref{thm:strongest_localstructure}}(3k - 1, k^{2}, 4), \mathsf{breadth}_{\ref{thm:strongest_localstructure}}(3k - 1, k^{2}), \mathsf{depth}_{\ref{thm:strongest_localstructure}}(3k - 1, k^{2}, 4), 3)$-$\Sigma$-layout $\Lambda$ of $(G, R)$ centred at $W$ and the surface $\Sigma$ has genus less than $9k^4.$
Also let $A$ be the apex set of $\Lambda.$

By definition of layouts, there exists a set $A \subseteq V(G)$ with $|A| \leq \alpha \coloneqq \mathsf{apex}_{\ref{thm:strongest_localstructure}}(3k - 1, k^{2}, 4)$ and a submesh $M \subseteq W$ such that there exists a blank $\Sigma$-landscape $(A, M, D, \rho)$ of detail $4,$ breadth $b \coloneqq \nicefrac{3}{2}(k^{2} - 1)(3k^{2} - 4) + (3k - 1)(3k - 2) - 3 = \mathsf{breadth}_{\ref{thm:BidimensionalityInduction}}(k),$ and depth $d \coloneqq \mathsf{depth}_{\ref{thm:strongest_localstructure}}(3k - 1, k^{2}, 4), 3)$ for $G,$ where every vortex of $\rho$ has a linear decomposition of adhesion at most $d,$ and $M$ is a $w$-mesh with $s \geq w \geq \alpha + b(2d + 1) + 6 + 3.$

First, by definition of layouts, $M$ is grounded in $\rho.$
This means that for every non-vortex cell $c \in C(\rho),$ no entire horizontal or vertical path of $M$ can be contained in $\sigma(c).$
Observe that the tuple $(V(\sigma(c)) \cup A, V(G) \setminus (\sigma(c) \setminus \tilde{c}))$ is a separation of $G$ of order at most $\alpha + 3 < w,$ and therefore by the previous observation, we have that $(\sigma(c) \cup A, V(G) \setminus (\sigma(c) \setminus \tilde{c})) \in \mathcal{T}_{M}.$
Since $\mathcal{T}_{M} \subseteq \mathcal{T}_{W} \subseteq \mathcal{T}_{X},$  this implies that $|V(\sigma(c)) \cap X| \leq \lfloor \nicefrac{2}{3} (3s + 1) \rfloor$ for all non-vortex cells $c \in C(\rho).$

Moreover, for every vortex cell $v \in C(\rho)$ and every bag $Y^{i}_{c}$ of the linear decomposition of $v$ of adhesion at most $d,$ by definition there is a set $S^{i}_{v}$ of size at most $2d + 1,$ composed of the at most two adhesion sets of $Y^{i}_{v}$ and the vertex of $Y^{i}_{v}$ in $\tilde{v}.$
We may also observe that the tuple $(Y^{i}_{v} \cup A, V(G) \setminus (Y^{i}_{v} \setminus S^{i}_{v}))$ defines a separation of $G$ of order at most $\alpha + 2d + 1 < w,$ and therefore once more we have that $(Y^{i}_{v} \cup A, V(G) \setminus (Y^{i}_{v} \setminus S^{i}_{v})) \in \mathcal{T}_{M}.$
Since $\mathcal{T}_{M} \subseteq \mathcal{T}_{W} \subseteq \mathcal{T}_{X},$  this implies that $|V(\sigma(v)) \cap X| \leq \lfloor \nicefrac{2}{3} (3s + 1) \rfloor$ for all vortex cells $v \in C(\rho)$ and for all $i \in [x_{v}]$ with $x_{v}$ being the number of bags in the linear decomposition of $v.$

Now, for every non-vortex cell $c \in C(\rho)$ let $H_{c} \coloneqq G[A \cup V(\sigma(c))]$ and $X_{c} \coloneqq A \cup \tilde{c} \cup (X \cap V(\sigma(c))).$

For every vortex cell $c \in C(\rho)$ and every $i \in [x_{c}],$ with $x_{c}$ being the number of bags in the linear decomposition of $c,$ let $H^{i}_{c} \coloneqq G[Y^{i}_{c} \cup A]$ and $X^{i}_{c}$ be the union of $A,$ $X \cap V(\sigma(c)),$ the (at most two) adhesion sets of $Y^{i}_{c},$ and the vertex of $Y^{i}_{c}$ in $\tilde{c}.$
It follows that $$|X_{c}| \leq \alpha + 3 + \bigg\lfloor \frac{2}{3} (3s + 1) \bigg\rfloor \leq \alpha + 3 + 2s \leq 3s < 3s + 1,$$
for all non-vortex cells $c \in C(\rho),$ and $$|X^{i}_{v}| \leq \alpha + 1 + 2d + \bigg\lfloor \frac{2}{3} (3s + 1) \bigg\rfloor \leq s + 2s \leq 3s < 3s + 1,$$
for all vortex cells $v \in C(\rho)$ and all $i \in [x_{v}].$

Now each $H_{c}$ together with its respective $X_{c},$ as well as, each $H^{i}_{v}$ together with its respective $X^{i}_{v},$ falls into one of the principal cases.
Therefore, we obtain a rooted tree-decomposition $(T_{c}, r_{c}, \beta_{c})$ and a set $L_{c}$ for each non-vortex cell $c,$ as well as a rooted tree-decomposition $(T^{i}_{v}, r^{i}_{v}, \beta^{i}_{v})$ with a set $L^{i}_{v}$ for each $i \in [x_{v}]$ for each vortex cell $v.$

It remains to combine them into a rooted tree-decomposition for the entire graph.
Let $T$ be the tree obtained from the disjoint union of all $T_{c}$ and $T^{i}_{v}$ by adding a new vertex $r$ adjacent to all $r_{c}$ and $r^{i}_{v}.$
Set $\beta(r) \coloneqq N(\rho) \cup A \cup X \cup B,$ where $B$ denotes the union of all adhesion sets of the linear decompositions of all vortices $v \in C(\rho).$
For all $t \in V(T)$ such that there is a non-vortex cell $c \in C(\rho)$ with $t \in V(T_{c})$ or a vortex cell $v \in C(\rho)$ and some $i \in [x_{v}]$ such that $t \in V(T^{i}_{v})$ we set $\beta(t) \coloneqq \beta_{c}(t)$ and $\beta(t) \coloneqq \beta^{i}_{v}(t)$ respectively.
Also, let $L$ be defined as the union of all sets $L_{c}$ and $L^{i}_{v}.$

As for the width of $\rho$ it follows that $$|\beta(r) \cap Y^{i}_{v}| \leq 2d + 1 + \bigg\lfloor \frac{2}{3} (3s + 1) \bigg\rfloor \leq 2d + 1 + 2s,$$
for every vortex cell $v$ and every $i \in [x_{v}],$ as advertised.
Finally since the layout we started with is blank, $(T, r, \beta)$ is as desired.
This concludes our proof.
\end{proof}

\section{Excluding an apex graph}\label{sec:ApexExclusion}
We now turn to discussing our version of the structure theorem for apex-minor-free-graphs with polynomial bounds.
As mentioned in the introduction, we will not provide detailed proofs for the statements we prove in this section, as all methods we present here have received exhaustive exposition already over the course of this manuscript.
Instead, we will give rough proof sketches and the intermittent statements needed to reach the goal of \zcref{thm:ApexMinorIntro}.

First, recall that an \emph{apex graph} is a graph $H$ such that there exists a vertex $v \in V(H)$ whose removal turns $H$ planar.
To illustrate the type of adjustment we need to make to prove the structure theorem for apex-minor-free-graphs, we first sketch a proof of a variant of the flat wall theorem.
We note that from this point on, we make no attempts at optimising our bounds beyond wanting to guarantee that they are polynomial.

\begin{theorem}\label{thm:ApexFlatMesh}
    There exists a function $\mathsf{afw}_{\ref{thm:ApexFlatMesh}} \colon \mathbb{N}^{2} \to \mathbb{N}$ such that for all integers $t \geq 5$ and $r \geq 2$ the following holds.

    Let $H$ be an apex graph with $t$ vertices and let $G$ be a graph with an $\mathsf{afw}_{\ref{thm:ApexFlatMesh}}(t, r)$-mesh $M \subseteq G$.
    Then there exists either
    \begin{itemize}
        \item an $H$-minor-model $\mu$ in $G$ such that $\mathcal{T}_{\mu}$ is a truncation of $\mathcal{T}_{M}$, or
        \item an $r$-submesh $M' \subseteq M$ such that $M'$ is a flat mesh in $G$.
    \end{itemize}
    Furthermore, $\mathsf{afw}_{\ref{thm:ApexFlatMesh}}(t, r) \in \mathbf{poly}(t + r)$ and there exists an algorithm that either find $\mu$ or $M'$ and $\rho$ as above in time $\mathbf{poly}(t + r) \cdot |E(G)|$.
\end{theorem}
\begin{proof}[Proof Sketch.]
    We start by applying \zcref{prop:FlatMesh} whilst searching for a $K_t$-minor model.
    Of course, if we find said $K_t$-minor model, we have also found an $H$-minor model.
    Thus we must instead find a submesh $M'$ of $M$ and a set $Z \subseteq V(G) \setminus V(M')$ such that $M'$ is a flat mesh in $G' \coloneqq G - Z$, where $|Z| \leq t^2 - t$.
    Let $k \coloneqq t^2 - t$.
    We now colour the neighbourhoods of the vertices in $Z$ red, which gives us an annotated version $(G',R)$ of $G'$.
    
    After applying \zcref{lemma:flatMeshToHomegeneousMesh}, we either find a blank mesh in $(G',R)$, which thus has no neighbours in $Z$ in $G$ and is therefore flat in $G$, or we find a red mesh $M''$.
    If we chose the desired mesh size to be $\max(r,k (t-1)^4)$, we can then find a grid minor in $G$ in which at least $(t-1)^4$ vertices are seen by the same vertex $v \in Z$.
    This allows us to now fairly aggressively assemble any planar graph on $t-1$ vertices and thus construct $H$ using $v$.
\end{proof}

This leads us to the version of the society classification theorem we need, the proof of which does require some more explanation.

\begin{lemma}\label{thm:apexsocietyclassification}
    There exist polynomial functions $\mathsf{loss}_{\ref{thm:apexsocietyclassification}} \colon \mathbb{N} \rightarrow \mathbb{N}$, $\mathsf{nest}_{\ref{thm:apexsocietyclassification}}, \mathsf{cost}_{\ref{thm:apexsocietyclassification}} \colon \mathbb{N}^2 \rightarrow \mathbb{N}$, and $ \mathsf{apex}_{\ref{thm:apexsocietyclassification}}$,  $\mathsf{depth}_{\ref{thm:apexsocietyclassification}} \colon \mathbb{N}^3\to\mathbb{N}$, such that for all integers $t, k, p \geq 1$ the following holds.

    Let $s \geq \mathsf{nest}_{\ref{thm:apexsocietyclassification}}(t, k)$ be an integer and let $H$ be an apex graph on $t$ vertices.
    Let $(G,\Omega)$ be a society and $\rho$ be a cylindrical rendition of $(G, \Omega)$ in a disk $\Delta$ with a cozy nest $\mathcal{C} = \{ C_1, \ldots , C_s \}$ around the vortex $c_{0}$ and a radial linkage $\mathcal{R}$ for $\mathcal{C}$ of order $s$ that is orthogonal to $\mathcal{C}$.
    Further, let $M$ be the $s$-cylindrical mesh contained in $\bigcup (\mathcal{C} \cup \mathcal{R})$, and $(G', \Omega')$ be the $C_{s - \mathsf{cost}_{\ref{thm:apexsocietyclassification}}(t,k)}$-society in $\rho.$
    
    Then $G'$ contains a set $A \subseteq V(G')$ such that one of the following exists:
    \begin{enumerate}
        \item An $H$-minor-model in $G$ controlled by $M$.

        \item A flat, isolated crosscap transaction $\mathcal{P}$ of order $p$ in $(G',\Omega')$ and a nest $\mathcal{C}'$ in $\rho$ of order $s - (\mathsf{loss}_{\ref{thm:apexsocietyclassification}}(t)+\mathsf{cost}_{\ref{thm:apexsocietyclassification}}(t,k))$ around $c_0$ to which $\mathcal{P}$ is orthogonal.

        \item A flat, isolated handle transaction $\mathcal{P}$ of order $p$ in $(G',\Omega')$ and a nest $\mathcal{C}'$ in $\rho$ of order $s - (\mathsf{loss}_{\ref{thm:apexsocietyclassification}}(t)+\mathsf{cost}_{\ref{thm:apexsocietyclassification}}(t,k))$ around $c_0$ to which $\mathcal{P}$ is orthogonal.

        \item A rendition $\rho'$ of $(G - A, \Omega)$ in $\Delta$ with breadth $b \in [\nicefrac{1}{2}(t-3)(t-4)-1]$ and depth at most $\mathsf{depth}_{\ref{thm:apexsocietyclassification}}(t,k,p)$, $|A| \leq \mathsf{apex}_{\ref{thm:apexsocietyclassification}}(t,k,p)$, and an extended $k$-surface-wall $D$ with signature $(0,0,b)$, such that $D$ is grounded in $\rho'$, the base cycles of $D$ are the cycles $C_{s-\mathsf{cost}_{\ref{thm:apexsocietyclassification}}(t,k)-1-k},\dots,C_{s-\mathsf{cost}_{\ref{thm:societyclassification}}(t,k)-1}$, and there exists a bijection between the vortices $v$ of $\rho'$ and the vortex segments $S_v$ of $D$, where $v$ is the unique vortex contained in the disk $\Delta_{S_v}$ defined by the trace of the inner cycle of the nest of $S_v$, and $\Delta_{S_v}$ is chosen to avoid the trace of the simple cycle of $D$.
        Furthermore, for each $a \in A$ and any $u \in N_G(a) \setminus A$ there exists a vortex $v \in C(\rho')$ such that $u \in V(\sigma_{\rho'}(v)) \setminus N(\rho')$.
    \end{enumerate}
    
    In particular, the set $A$, the $K_t$-minor-model, the transaction $\mathcal{P}$, the rendition $\rho'$, and the extended surface-wall $D$ can each be found in time $\mathbf{poly}( t + s + p + k ) \cdot |E(G)|^3$.
\end{lemma}

\begin{proof}[Proof Sketch.]
    We begin by simply applying \zcref{thm:societyclassification}.
    Clearly, the first outcome will find us the desired $H$-minor model.
    
    The crosscap transaction of the second and the handle transaction of the third outcome can be used quite similarly to what we presented previously.
    First, we take the associated apex set and colour the neighbours of these apices red.
    Then we simply apply \zcref{lemma:homoTransactions} and if this yields a blank transaction, respectively two blank transactions in the third outcome, we have arrived at the second or third outcome of the apex society classification variant we want to prove.
    Therefore we must consider the case in which we find a red transaction.
    But this simply allows us to apply \zcref{lemma:RedTransactionGivesRedMesh} and thus we are done.

    What remains is therefore somehow to transform the fourth outcome of \zcref{thm:societyclassification} into the fourth outcome of the theorem we are trying to prove.
    As before, we take the apex set $A$ that comes with this step and colour all of their neighbours red.
    Note that in the resulting annotated graph no red vertices lies outside the inner graph of $C_1$ with respect to $\rho$, as all apices are found within $c_0$.
    Thus, after possibly sacrificing $C_1$, we know that all red vertices lie in the original vortex $c_0$.

    We then follow the proof strategy presented in \zcref{sec:ProofOfSociety} closely.
    The blank rendition $\rho'$ yielded by the fourth outcome of \zcref{thm:societyclassification} is essentially only used to confirm that we have a rendition of our graph in the disk with a bounded number of vortices of bounded depth that agrees with $\rho$ on the embedding on the outer graph of $C_1$ in $\rho$.
    From there we follow the proof exactly, including the use of nest trees.
    Of course, if this ever yields a red mesh, we can proceed as in the sketch of \zcref{thm:ApexFlatMesh} to find the $H$-minor model we desire.
    Otherwise we find a rendition $\rho'$ of $(G-A,\Omega)$ with polynomially bounded breadth $b$ and polynomially bounded depth and an extended $k$-surface-wall $D$ with signature $(0, 0, b)$, such that $D$ is grounded in $\rho'$, and there exists a bijection between the vertices $v$ of $\rho'$ and the vortex segments $S_{v}$ of $D$, where $v$ is the unique vortex contained in $\Delta_{S_{v}}$ defined by the trace of the inner cycle of the nest of $S_{v}$, and $\Delta_{S_{v}}$ is chosen to avoid the trace of the simple cycle of $D$.

    Note that in the above statement we did not make a typo, the rendition we find does in fact use the \textsl{original apex set from our application of \zcref{thm:societyclassification}}.
    In particular, the proof of the fourth outcome of \zcref{thm:blanksocietyclassification} does not yield \textsl{any additional vertices on top of $A$}.
    Thus, we can now guarantee the last part of the fourth outcome of the statement we are trying to prove here, namely that all neighbours of $A$ live in the interior of the vortices of $\rho'$.
\end{proof}

To now derive \zcref{thm:ApexMinorIntro} from \zcref{thm:ApexFlatMesh} and \zcref{thm:apexsocietyclassification}, one can simply follow the proof we gave for \zcref{thm:strongest_localstructure}.
For the sake of completeness, we state here the technical version of \zcref{thm:ApexMinorIntro} which can be proven using the methods we outline.
To be able to do this, we present slightly altered versions of landscapes and layouts.

Let $k, w \geq 4$ be integers, let $G$ be a graph, and let $\Sigma$ be a surface of Euler-genus $g$.
Let $h$, $c$, and $b$ be non-negative integers where $g=2h+c$ and $c \neq 0$ if and only if $\Sigma$ is non-orientable.
Moreover, let $D \subseteq G$ be a $k$-surface-wall with signature $(h,c,b)$, let $W \subseteq G-A$ be a $w$-mesh in $G$, and let $\mathcal{T}_D$ and $\mathcal{T}_W$ be the tangles they respectively define.
Finally, let $A \subseteq V(G) \setminus V(D)$.

The tuple $\Lambda = (A,W,D,\rho)$ is called an \emph{apex $\Sigma$-landscape} of \emph{detail $k$} if
\begin{description}
    \item[L1~~] $\rho$ is a $\Sigma$-rendition of $(G - A, R \setminus A),$
    \item[L2~~] $D$ and $W$ are grounded in $\rho,$
    \item[L3~~] $W$ is flat in $\rho,$
    \item[L4~~] the disk bounded by the trace of the simple cycle of $D$ in $\rho$ avoids the traces of the other base cycles of $D,$
    \item[L5~~] the tangle $\mathcal{T}_D$ is a truncation of the tangle $\mathcal{T}_W,$
    \item[L6~~] if $C$ is a cycle from the nest of some vortex-segment of $D,$ then the trace of $C$ is a contractible closed curve in $\Sigma,$
    \item[L7~~] $\rho$ has exactly $b$ vortices and there exists a bijection between the vortices $v$ of $\rho$ and the vortex segments $S_v$ of $D$ such that $v$ is the unique vortex of $\rho$ that is contained in the $v$-disk $\Delta_{C_1}$ of the inner cycle of $S_v$, where $\Delta_{C_1}$ avoids the trace of the simple cycle of $D,$
    \item[L8~~] for every vortex $v$ of $\rho$, the society induced by the outer cycle from the nest of the corresponding vortex segment has a cross, and
    \item[L9~~] for every $a \in A$ and every $u \in N_G(A) \setminus A$ there exists a vortex $v \in C(\rho)$ such that $u \in V(\sigma_{\rho}(v)) \setminus N(\rho)$.
\end{description}
The integer $b$ is called the \emph{breadth} of $\Lambda$ and the \emph{depth} of $\Lambda$ is the depth of $\rho$.
We say that $\Lambda$ is \emph{centred} at the mesh $W$.

Let $k \geq 4$, $l$, $d$, $b$, $r$, and $a$ be non-negative integers and $\Sigma$ be a surface.
We say that a graph $G$ with a mesh $M$ has an \emph{apex $k$-$(a,b,d,r)$-$\Sigma$-layout} $\Lambda$ \emph{centred at $M$} if there exists a set $A \subseteq V(G)$ of size at most $a$ and a submesh $M'\subseteq M$ such that there exists an apex $\Sigma$-landscape $(A,M',D,\rho)$ of detail $k$, breadth $b$, and depth $d$ for $G$ where every vortex of $\rho$ has a linear decomposition of adhesion at most $d$, and $M'$ is a $w$-mesh with $w \geq a+b(2d+1)+6+r$.

\begin{theorem}\label{thm:apexlocalstructure}
There exist polynomial functions 
$\mathsf{breadth}_{\ref{thm:apexlocalstructure}},\mathsf{genus}_{\ref{thm:apexlocalstructure}} \colon \mathbb{N} \to \mathbb{N}$, $\mathsf{apex}_{\ref{thm:apexlocalstructure}},\mathsf{depth}_{\ref{thm:apexlocalstructure}} \colon \mathbb{N}^2 \to \mathbb{N}$ and $\mathsf{mesh}_{\ref{thm:apexlocalstructure}} \colon \mathbb{N}^3 \to \mathbb{N}$ such that for all integers $t \geq 1$, $k \geq 4$, and $r \geq 3$, every graph $G$, every apex graph $H$ on $t$ vertices, and every $\mathsf{mesh}_{\ref{thm:apexlocalstructure}}(t, k, r)$-mesh $M \subseteq G$ one of the following holds.
\begin{enumerate}
    \item $G$ has an $H$-minor model controlled by $M$, or
    \item $G$ has an apex $k$-$(\mathsf{apex}_{\ref{thm:apexlocalstructure}}(t, k), \mathsf{breadth}_{\ref{thm:apexlocalstructure}}(t), \mathsf{depth}_{\ref{thm:apexlocalstructure}}(t, k), r)$-$\Sigma$-layout $\Lambda$ centred at $M$ and the surface $\Sigma$ has genus less than $\mathsf{genus}_{\ref{thm:apexlocalstructure}}(t)$.
\end{enumerate}
There also exists an algorithm that, given $t$, $k$, $r$, a graph $G$, and a mesh $M$ as above as input finds one of these outcomes in time $\mathbf{poly}(t+k+r)|E(G)|^3$.
\end{theorem}
\begin{proof}[Proof Sketch.]
    As noted above, the proof needed to derive this theorem simply replaces the initial application of \zcref{thm:FlatRedMesh} in the proof of \zcref{thm:strongest_localstructure} by \zcref{thm:ApexFlatMesh} and the applications of \zcref{thm:blanksocietyclassification} by \zcref{thm:apexsocietyclassification}.
    Though it would be more accurate to say that we do the appropriate replacements in the proof of the Local Structure Theorem in \cite{GorskySW2025Polynomial}, as that proof is even closer to what we need now than the proof of \zcref{thm:strongest_localstructure} we present.
    As side notes, all other statements we used to prove \zcref{thm:strongest_localstructure} still work for this setting, with the notable exception of \zcref{lemma:integrate-crosscap} and \zcref{lemma:integrate-handle}, for which the versions of these statements found in \cite{GorskySW2025Polynomial} should be used.
    Following this proof it is particularly noticeable that the apices in the second outcome of the statement only appear at a single point in the proof, which is when \zcref{thm:apexsocietyclassification} reaches its fourth outcome and the proof concludes.
\end{proof}

We note that reasonable estimates for the functions $\mathsf{genus}_{\ref{thm:apexlocalstructure}}$ and $\mathsf{breadth}_{\ref{thm:apexlocalstructure}}$ are considerably more modest than what can be expected for the other functions involved in the statement of \zcref{thm:apexlocalstructure}.
Further, it is clear that explicit bounds for the three theorems we present in this section can be derived from our methods.


\newpage
\phantomsection
\addcontentsline{toc}{section}{References}
\bibliographystyle{alphaurl}
\bibliography{literature}

\end{document}